\documentclass[a4wide, 12pt]{amsart}
\usepackage{amssymb, amsmath, amsthm, mathrsfs, bbm}
\usepackage{a4wide, url}
\usepackage[arrow, matrix, tips, curve]{xy}
\usepackage{enumerate}

\DeclareMathOperator{\ob    }{ob}

\DeclareMathOperator{\Hom   }{Hom}

\DeclareMathOperator{\el    }{el}
\DeclareMathOperator{\cod   }{cod}

\newcommand{\two}{\ensuremath{\mathrm{2}}}
\newcommand{\op}{\ensuremath{\mathrm{op}}}
\newcommand{\coop}{\ensuremath{\mathrm{coop}}}
\newcommand{\Cat}{\ensuremath{\mathrm{Cat}}}
\newcommand{\bicat}{\ensuremath{\mathrm{Bicat}}}
\newcommand{\Set}{\ensuremath{\mathrm{Set}}}
\newcommand{\Gray}{\ensuremath{\mathrm{Gray}}}
\newcommand{\Fam}{\ensuremath{\mathrm{Fam}}}
\newcommand{\Fib}{\ensuremath{\mathrm{Fib}}}

\newcommand{\Mon}{\ensuremath{\mathrm{Mon}}}
\newcommand{\Mnd}{\ensuremath{\mathrm{Mnd}}}
\newcommand{\Enr}{\ensuremath{\mathrm{Enr}}}
\newcommand{\Alg}{\ensuremath{\mathrm{Alg}}}

\newcommand{\A}{\mathscr{A}}
\newcommand{\B}{\mathscr{B}}
\newcommand{\C}{\mathscr{C}}
\newcommand{\D}{\mathscr{D}}
\newcommand{\E}{\mathscr{E}}
\newcommand{\K}{\mathscr{K}}
\newcommand{\V}{\mathscr{V}}

\mathchardef\mhyphen="2D

\renewcommand{\bar}{\overline}

\newcommand{\To}{\Rightarrow}
\newcommand{\e}[1]{{{#1}\_}}
\renewcommand{\c}[2]{\varphi({#1},{#2})}
\newcommand{\comma}[2]{{({#1}\downarrow{#2})}}

\newcommand{\cd}[2][]{\vcenter{\hbox{\xymatrix#1{#2}}}}
\newcommand{\Mapsto}[1]{\cd[]{ \ar@{|->}[r]^{{#1}} & }}

\newcommand{\pullback}[1][dr]{\save*!/#1-1.2pc/#1:(-1,1)@^{|-}\restore}

\newcommand{\ltwocell    }[3][0.5]{\ar@{}[#2] \ar@{=>}?(#1)+/r 0.2cm/;?(#1)+/l 0.2cm/^{#3}}
\newcommand{\rtwocell    }[3][0.5]{\ar@{}[#2] \ar@{=>}?(#1)+/l 0.2cm/;?(#1)+/r 0.2cm/^{#3}}
\newcommand{\dtwocell    }[3][0.5]{\ar@{}[#2] \ar@{=>}?(#1)+/u 0.2cm/;?(#1)+/d 0.2cm/^{#3}}
\newcommand{\utwocell    }[3][0.5]{\ar@{}[#2] \ar@{=>}?(#1)+/d 0.2cm/;?(#1)+/u 0.2cm/_{#3}}

\newcommand{\dltwocell   }[3][0.5]{\ar@{}[#2] \ar@{=>}?(#1)+/ur 0.2cm/;?(#1)+/dl 0.2cm/^{#3}}
\newcommand{\drtwocell   }[3][0.5]{\ar@{}[#2] \ar@{=>}?(#1)+/ul 0.2cm/;?(#1)+/dr 0.2cm/^{#3}}
\newcommand{\ultwocell   }[3][0.5]{\ar@{}[#2] \ar@{=>}?(#1)+/dr 0.2cm/;?(#1)+/ul 0.2cm/^{#3}}
\newcommand{\urtwocell   }[3][0.5]{\ar@{}[#2] \ar@{=>}?(#1)+/dl 0.2cm/;?(#1)+/ur 0.2cm/^{#3}}

\newcommand{\dotltwocell }[3][0.5]{\ar@{}[#2] \ar@{:>}?(#1)+/r 0.2cm/;?(#1)+/l 0.2cm/^{#3}}
\newcommand{\dotrtwocell }[3][0.5]{\ar@{}[#2] \ar@{:>}?(#1)+/l 0.2cm/;?(#1)+/r 0.2cm/^{#3}}
\newcommand{\dotdtwocell }[3][0.5]{\ar@{}[#2] \ar@{:>}?(#1)+/u 0.2cm/;?(#1)+/d 0.2cm/^{#3}}
\newcommand{\dotutwocell }[3][0.5]{\ar@{}[#2] \ar@{:>}?(#1)+/d 0.2cm/;?(#1)+/u 0.2cm/_{#3}}
\newcommand{\doturtwocell}[3][0.5]{\ar@{}[#2] \ar@{:>}?(#1)+/dl 0.2cm/;?(#1)+/ur 0.2cm/^{#3}}

\newcommand{\twocong    }[2][0.5]{\ar@{}[#2] \save ?(#1)*{\cong} \restore}
\newcommand{\twoeq      }[2][0.5]{\ar@{}[#2] \save ?(#1)*{=}     \restore}

\newcommand{\lthreecell  }[3][0.5]{\ar@{}[#2] \ar@3{->}?(#1)+/r 0.2cm/;?(#1)+/l 0.2cm/^{#3}}
\newcommand{\rthreecell  }[3][0.5]{\ar@{}[#2] \ar@3{->}?(#1)+/l 0.2cm/;?(#1)+/r 0.2cm/^{#3}}
\newcommand{\dthreecell  }[3][0.5]{\ar@{}[#2] \ar@3{->}?(#1)+/u 0.2cm/;?(#1)+/d 0.2cm/^{#3}}
\newcommand{\uthreecell  }[3][0.5]{\ar@{}[#2] \ar@3{->}?(#1)+/d 0.2cm/;?(#1)+/u 0.2cm/_{#3}}

\newcommand{\dlthreecell }[3][0.5]{\ar@{}[#2] \ar@3{->}?(#1)+/ur 0.2cm/;?(#1)+/dl 0.2cm/^{#3}}
\newcommand{\drthreecell }[3][0.5]{\ar@{}[#2] \ar@3{->}?(#1)+/ul 0.2cm/;?(#1)+/dr 0.2cm/^{#3}}
\newcommand{\ulthreecell }[3][0.5]{\ar@{}[#2] \ar@3{->}?(#1)+/dr 0.2cm/;?(#1)+/ul 0.2cm/^{#3}}
\newcommand{\urthreecell }[3][0.5]{\ar@{}[#2] \ar@3{->}?(#1)+/dl 0.2cm/;?(#1)+/ur 0.2cm/^{#3}}

\newcommand{\pair        }[4][10pt]{\ar@/^#1/[#2]^-{#3} \ar@/_#1/[#2]_-{#4}}
\newcommand{\dotpair     }[4][10pt]{\ar@{.>}@/^#1/[#2]^{#3} \ar@{.>}@/_#1/[#2]_{#4}}
\newcommand{\triple      }[5]{\ar@/^#5/[#1]^{#2} \ar@/_#5/[#1]_{#4} \ar@/12pt/[#1]|{#3}}

\newcommand{\ddtwocell   }[4][0.5]{
\ar@{}[#2] \ar@{=>}?(#1)+/u 0.5cm/;?(#1)+/u 0.1cm/_{#3}
\ar@{}[#2] \ar@{=>}?(#1)+/d 0.1cm/;?(#1)+/d 0.5cm/_{#4}}

\newcommand{\lltwocell   }[4][0.5]{
\ar@{}[#2] \ar@{=>}?(#1)+/r 0.5cm/;?(#1)+/r 0.1cm/_{#3}
\ar@{}[#2] \ar@{=>}?(#1)+/l 0.1cm/;?(#1)+/l 0.5cm/_{#4}}

\swapnumbers
\theoremstyle{plain}
\newtheorem{Thm}{Theorem}[subsection]
\newtheorem{Prop}[Thm]{Proposition}
\newtheorem{Cor}[Thm]{Corollary}

\theoremstyle{definition}
\newtheorem{Defn}[Thm]{Definition}
\newtheorem{Con}[Thm]{Construction}
\newtheorem{Ex}[Thm]{Example}

\newtheorem{Rk}[Thm]{Remark}

\begin{document}

\title{Fibred 2-categories and bicategories}
\author{Mitchell Buckley}
\address{Department of Computing, Macquarie University, NSW 2109, Australia}
\email{mitchell.buckley@mq.edu.au} 
\date{\today}

\begin{abstract}
We generalise the usual notion of fibred category; first to \emph{fibred 2-categories} and then to \emph{fibred bicategories}. Fibred 2-categories  correspond to 2-functors from a 2-category into $\two\Cat$. Fibred bicategories correspond to trihomomorphisms from a bicategory into $\bicat$. We describe the Grothendieck construction for each kind of fibration and present a few examples of each. 
Fibrations in our sense, between bicategories, are closed under composition and are stable under \emph{equiv-comma}. The free such fibration on a homomorphism is obtained by taking an \emph{oplax comma} along an identity.
\end{abstract}

\maketitle
\tableofcontents

\section{Introduction}

Fibred categories were first developed by Grothendieck \cite{groth1971, groth1995} to describe notions of descent in algebraic geometry. 
Some of this material was then extended by Gray \cite{gray1966} as tool for understanding \v{C}ech cohomology and by Giraud \cite{giraud1964, giraud1971} for non-abelian cohomology.
Later, Street described fibrations internal to any bicategory \cite{street1974, street1980} together with internal two-sided fibrations. 
Some more recent work on internal fibrations can be found in \cite{reihl2010}. 
Fibrations have found strong application in categorical logic: in describing comprehension schema for categories \cite{gray1969, lawvere1970} and via indexed categories \cite{pare1978}.
For an overview of applications to categorical logic and type theory see \cite{jacobs1999}.
Fibrations were also used by Benabou \cite{benabou1985} to describe some foundations of category theory.

Fibred 2-categories (also called 2-fibrations) were investigated by Hermida \cite{hermida1999} where the projection $\Fib \to \Cat$ was used as a canonical example. A definition of 2-fibration is given that very nearly (but not entirely) captures the full structure required for a Grothendieck construction. This definition was extended in a preprint by Bakovic \cite{bakovic2012} to strict homomorphisms of bicategories. In that paper he describes the action-on-objects of a Grothendieck construction sending trihomomorphisms $B^\coop \to \bicat$ to fibrations of bicategories. The paper also presents a large number of examples and from each strict homomorphism of bicategories he constructs a `canonical fibration' associated to that homomorphism. Fibrations of bicategories are characterised by a certain right biadjoint right inverse. The action-on-objects of a pseudo-inverse to the Grothendieck construction is also partially described.

Our goal is to establish precise definitions of 2-fibration and fibration of bicategories by describing a complete Grothendieck construction in each case. By `complete' we mean a construction that is provably a 3-equivalence or triequivalence. In the most general case fibrations of bicategories should be triequivalent to trihomomorphisms into $\bicat$. Among other things this means dealing with the fibers of non-strict homomorphisms and understanding the properties of cartesian 1- and 2-cells. Our second goal is to understand in what sense these fibrations are closed under `pullback' and composition. Third, we aim to describe the `free fibration' on a homomorphism of bicategories.

In doing this we find that the existing definitions of fibration need to be adjusted: cartesian 2-cells must be preserved by both pre- and post-composition with any 1-cell. Without this change it is not possible to construct a pseudo-inverse to this Grothendieck construction and fibrations of bicategories do not truly correspond to trihomomorphisms $\B^\coop \to \bicat$. In general, fibrations of bicategories can be non-strict homomorphisms $P\colon \E \to \B$ and we use the fact that fibrations are locally iso-lifting to show that any fibration is equivalent to a strict (and better behaved) fibration. We prove many of the usual results concerning cartesian 1-cells in this new context. The main result is a proof that the Grothendieck construction described partially in \cite{bakovic2012} can be extended to a full triequivalence once the necessary adjustments are made. We give the construction of the free fibration (Bakovic's canonical fibration) using the \emph{oplax comma} construction. We also show that fibrations are closed under composition and closed under \emph{equiv-comma}.

In section 1 we give an introduction and remind the reader of the basic theory of fibred categories. Subsection 1.1 includes the basic definitions of cartesian arrow, fibration, cleavage, et cetera. We then outline some standard properties of cartesian arrows and give a brief description of the original Grothendieck construction. 

In section 2 we outline the theory of fibred 2-categories. We say that a 2-category is fibred when it is the domain of a \emph{2-fibration}. In line with the usual theory, we require that these 2-functors have \emph{cartesian 1-cells} which are cartesian in the normal sense but also have a 2-dimensional lifting property. A 2-fibration $P$ also has \emph{cartesian 2-cells} which are cartesian as 1-cells for the action of $P$ on hom-sets; we ask also that cartesian 2-cells be closed under horizontal composition.
In subsection 2.1 we give definitions of cartesian 1-cell, cartesian 2-cell and 2-fibration. We prove 2-categorical versions of the standard results concerning cartesian 1-cells.  Subsection 2.2 gives the Grothendieck construction for 2-categories: this is an equivalence that sends 2-fibrations $P \colon E \to B$ to 2-functors from $B^\coop \to \two\Cat$. Subsection 2.3 contains some examples of 2-fibrations. 
Many of the results found in this section correspond well with the classical theory and become somewhat routine once the right foundations are established.

In section 3 we outline the theory of fibred bicategories. We say that a bicategory is fibred when it is the domain of a (bicategorical) \emph{fibration}. Fibrations of bicategories have the same structure as 2-fibrations except that cartesian 1-cells have a much weaker lifting property defined by \emph{bipullback} in $\Cat$. This significantly weakens the usual results concerning cartesian 1-cells: multiple invertible 2-cells are introduced into every calculation and many uniqueness properties are reduced to `unique up to isomorphism' or weaker. Despite these complications the usual results can be stated in a form that is consistent with this bicategorical setting. Subsection 3.1 covers these new definitions and results. The fact that fibrations are locally fibred means that they locally have the iso-lifting property; as a result many of the complications mentioned above can be simplified. In subsection 3.2 we show that every fibration is equivalent to one with a somewhat simpler structure. Subsection 3.3 describes the Grothendieck construction for fibrations between bicategories: a triequivalence that sends fibred bicategories to trihomomorphisms into $\bicat$! Some of the heavier calculations have been omitted as they are not immediately helpful to the reader.
Subsection 3.4 contains a number of examples.

In section 4 we investigate how fibrations between bicategories behave under composition, pullback and comma. Fibrations are closed under composition (4.1) and stable under (bi-)pullback (4.2). We define the oplax comma of two homomorphisms (4.3) and show that the free fibration on a homomorphism is given by taking the oplax comma with the identity.

\subsection{Standard notation}

Suppose $P\colon E\to B$ is a functor. A map $f\colon a \to b$ in $E$ is \emph{cartesian} when
\begin{equation*}
\cd[]{ E(z,a) \ar[r]^{f_*} \ar[d]_{P_{za}} & E(z,b) \ar[d]^{P_{zb}} \\ B(Pz,Pa) \ar[r]_{Pf_*} & B(Pz,Pb)}
\end{equation*}
is a pullback. This is the same as saying that for each $g \colon c \to b$ with $Pg = Pf.h$ there exists a unique $\hat h \colon c \to a$ with $P\hat h = h$ and $g = f.\hat h$.

A functor $P$ is a \emph{fibration} when for all $e\in E$ and $f\colon b \to Pe$ in $B$ there is a cartesian map $h\colon a \to e$ with $Ph=f$. In this case we say that $h$ is a cartesian lift of $f$. We say that $E$ is a \emph{fibred category} when it is the domain of a fibration.
A \emph{cleavage} for a fibration $P$ is a function $\c{\ }{\ }$ that describes a choice of cartesian lifts.
That is, $\c{f}{e}\colon f^*e \to e$ is the chosen cartesian lift of $f$ at $e$.
A fibration equipped with a cleavage is called a \emph{cloven fibration}.
Every fibration can be equipped with a cleavage using the axiom of choice so we implicitly regard all fibrations as cloven.
If a cloven fibration has $\c{g.f}{e} = \c{g}{e}.\c{f}{g^*e}$ and $\c{1_{Pe}}{1} = 1_e$ then we say the cleavage is \emph{split} and call it a \emph{split fibration}.

The following results are easy to verify:
Cartesian lifts of any $f$ are unique up to isomorphism in the slice over their common codomain.
If $f$ and $g$ are cartesian then so is $g.f$.
If $g.f$ and $g$ are cartesian then so is $f$.
If $f$ is cartesian and $Pf$ is an isomorphism, then $f$ is an isomorphism.

Suppose that
\begin{equation*}
 \cd[@C-4pt]{ E \ar[dr]_{P} \pair{rr}{F}{G} \dtwocell{rr}{\alpha} & & D \ar[dl]^{Q} \\ & B & }
\end{equation*}
where $P$ and $Q$ are fibrations.
We say that $F$ is \emph{cartesian} when it preserves cartesian maps and $P=QF$.
If $P$ and $Q$ are cloven fibrations and $F$ has the property that $F(\c{v}{e}) = \c{v}{Fe}$ we say that $F$ is \emph{split}.
A natural transformation $\alpha$ is called \emph{vertical} when $1_P = Q\alpha$ as pictured.

We use $\Fib(B)$ to denote the 2-category of fibrations over $B$, cartesian functors and vertical natural transformations. Let $\Cat$ be the 2-category of small categories. We use $\Hom(B^\op,\Cat)$ to denote the 2-category of pseudo-functors, pseudo-natural transformations and modifications. The \emph{Grothendieck construction} is a 2-functor $\el \colon \Hom(B^\op,\Cat) \to \Fib(B)$. It sends each pseudo-functor $F \colon B^\op \to \Cat$ to the obvious projection $\el F \to B$ from the \emph{category of elements}. The category of elements has objects pairs $(a,x)$ where $a \in B$ and $x \in Fa$; arrows are pairs $(f,u)\colon (a,x) \to (b,y)$ where $f \colon a \to b$ and $u \colon x \to Ff(y)$. The Grothendieck construction is an equivalence.

Suppose the following square is a pullback. If $P$ is a fibration then $P'$ is a fibration and $F'$ is a cartesian functor; $F'$ also reflects cartesian maps.
\begin{equation*}
 \cd[]{ 
 D \ar[r]^{F'} \ar[d]_{P'} & E \ar[d]^{P}\\
 C \ar[r]_{F} & B
}
\end{equation*}
Fibrations are also closed under composition.
There is a 2-monad on $\Cat/B$ whose category of algebras is $\Fib(B)$. The monad acts by taking the comma category of each functor $F \colon A \to B$ with $1_B$; the corresponding projection into $B$ is the \emph{free fibration on $F$}.

A functor $P\colon E \to B$ is a \emph{Street fibration} when for all $e\in E$ and $f\colon b \to Pe$ in $B$ there is a cartesian map $h\colon a \to e$ with $Ph\cong f$. This isomorphism is in the slice over $Pe$. Morphisms of Street fibrations $P$,$Q$ are pairs $(F,F')$ where $QF \cong F'P$ and $F$ preserves cartesian maps. This is only a slight variation on the ordinary notion of fibration but is useful for considering fibrations internal to a 2-category.

\section{Fibred 2-Categories} \label{sec:strict}

We present the basic data and properties of fibrations of 2-categories. All notions in this section are completely 2-categorical unless otherwise indicated.

\subsection{Definitions and properties of cartesian 1- and 2-cells}

We wish to define fibrations of 2-categories in a way that fits with usual definition of fibration of categories. It is thus necessary to describe cartesian arrows (and in this case cartesian 2-cells). Let $E$ and $B$ be 2-categories and let $P\colon E \to B$ be a 2-functor.

\begin{Defn} \label{def:cartesian_1-cell_strict}
We shall say a 1-cell $f\colon x \to y$ in $E$ is \emph{cartesian} when it has the following two properties.
\begin{enumerate}
\item For all $h\colon z\to y$ and $u\colon Pz \to Px$ with $Ph = Pf . u$, there is a unique $\hat{u}\colon z \to x$ with $P\hat{u}=u$ and $h = f . \hat{u}$;\label{cart_prop_one}
\begin{equation*}
\cd[@C+48pt]{
z \ar[dr]^{h} \ar@{.>}[d]_{\hat{u}} & & Pz \ar[dr]^{Ph} \ar[d]_{u} & \\
x \ar[r]_{f} & y & Px \ar[r]_{Pf} & Py
}.
\end{equation*}
We call $\hat{u}$ the \emph{lift} of $u$.
\item For all $\sigma\colon h \To k$, $\tau\colon u \To v$ with $P\sigma = Pf . \tau$ and lifts $\hat{u}$,$\hat{v}$ of $u$,$v$, there is a unique $\hat{\tau}\colon \hat{u} \To \hat{v}$ with $P\hat{\tau}=\tau$ and $\sigma = f . \hat{\tau}$.
\begin{equation*}
\cd[@C+48pt@R+12pt]{
z \pair[8pt]{dr}{h}{k} \dltwocell{dr}{\sigma} \pair[8pt]{d}{\hat{u}}{\hat{v}} \dotltwocell{d}{\hat{\tau}} & & Pz \pair[8pt]{dr}{Ph}{Pk} \dltwocell[0.5]{dr}{P\sigma} \pair[8pt]{d}{u}{v} \ltwocell{d}{\tau} & \\
x \ar[r]_{f} & y & Px \ar[r]_{Pf} & Py
}
\end{equation*}
We call $\hat{\tau}$ the \emph{lift} of $\tau$.
\end{enumerate}
\end{Defn}

It is not hard to prove that:
\begin{Prop}
A 1-cell $f\colon x \to y$ in $E$ is cartesian if and only if
\begin{equation*}
\cd[]{
E(z,x) \ar[r]^{f_*} \ar[d]_{P_{zx}} & E(z,y) \ar[d]^{P_{zy}}\\
B(Pz,Px) \ar[r]_{Pf_*}& B(Pz,Py)
}
\end{equation*}
is a pullback in $\Cat$.
\end{Prop}

\begin{Defn}
A 2-cell $\alpha\colon f \To g\colon x \to y$ in $E$ is \emph{cartesian} if it is cartesian as a 1-cell for the functor $P_{xy}\colon E(x,y) \to B(Px,Py)$.
\end{Defn}

We take the time here to establish a few basic properties of cartesian maps.

\begin{Prop} Suppose $P\colon E\to B$ is a 2-functor.
\begin{enumerate}
\item If $f\colon x \to y$ in $E$ is cartesian for $P$, then it is cartesian for the underlying functor $P_0\colon E_0 \to B_0$.
\item If $f\colon x \to y$ and $f'\colon z \to y$ are cartesian and $Pf=Pf'$ then there exists a unique isomorphism $h\colon z \to x$ with $f' = h.f$ and $Ph=1_{Pa}$.
\item If $f\colon x \to y$ is cartesian and $Pf$ is an isomorphism then $f$ is an isomorphism.
\item Suppose $f\colon x \to y$ and $g\colon y \to z$ in $E$. If $f$ and $g$ are cartesian then $g.f$ is cartesian.
\item Suppose $f\colon x \to y$ and $g\colon y \to z$ in $E$. If $g$ and $g.f$ are cartesian then $f$ is cartesian.
\end{enumerate}
\end{Prop}
\begin{proof}
(1) is true because cartesian 1-cells have the ordinary lifting property for 1-cells. (2) and (3) are a consequence of (1). For (4), notice that since $f$ and $g$ are cartesian, the two commuting squares below are pullbacks. Hence, the outer rectangle is a pullback and $g.f$ is cartesian.
\begin{equation*}
\cd[]{
E(w,x) \ar[r]^{f_*} \ar[d]_{P_{wx}}  & E(w,y) \ar[d]^{P_{wy}} \ar[r]^{g_*} & E(w,z) \ar[d]^{P_{wz}}\\
B(Pw,Px) \ar[r]_{Pf_*}  & B(Pw,Py) \ar[r]_{Pg_*}& B(Pw,Pz)
}
\end{equation*}
For (5), use the same diagram as above. Since $g$ and $g.f$ are cartesian, the right square and outer rectangle are pullbacks. Hence, the left square is a pullback and $f$ is cartesian.
\end{proof}

\begin{Prop}
Suppose $P\colon E\to B$ is a 2-functor and that $h\colon y \to z$, $\alpha \colon f \To g \colon x \to y$ in $E$. If $h$ and $h\alpha$ are cartesian then $\alpha$ is cartesian.
\end{Prop}
\begin{proof}
Since $h$ is cartesian, the following square is a pullback.
\begin{equation*}
\cd[]{
E(x,y) \ar[r]^{h_*} \ar[d]_{P_{zx}} & E(x,z) \ar[d]^{P_{zy}}\\
B(Px,Py) \ar[r]_{Ph_*}& B(Px,Pz)
}
\end{equation*}
It is a property of pullbacks that $h_*$ reflects cartesian maps. Now since $h\alpha$ is cartesian, $\alpha$ is cartesian.
\end{proof}


\begin{Defn}
A 2-functor $P\colon E \to B$ is a \emph{2-fibration} if
\begin{enumerate}
\item for any $e \in E$ and $f\colon b \to Pe$, there is a cartesian 1-cell $h\colon a \to e$ with $Ph = f$;
\item for any $g \in E$ and $\alpha\colon f \To Pg$, there is a cartesian 2-cell $\sigma\colon f \To g$ with $P\sigma = \alpha$; and
\item the horizontal composite of any two cartesian 2-cells is cartesian.
\end{enumerate}
We say that $E$ is a \emph{fibred 2-category} when it is the domain of a 2-fibration $P\colon E \to B$.
\end{Defn}

\begin{Rk}
The second condition could equivalently be stated as ``$P_{xy}\colon E(x,y) \to B(Px,Py)$ is a fibration for all $x$, $y$ in $E$''. In this case we say that $P$ is \emph{locally fibred}.
\end{Rk}

\begin{Rk}
The third condition could equivalently be stated as ``cartesian 2-cells are closed under pre-composition and post-composition with arbitrary 1-cells''. This is a consequence of the middle-four interchange and the fact that cartesian 2-cells are closed under vertical composition.
\end{Rk}

\begin{Rk}
The first definition of 2-fibration was given by Hermida in \cite{hermida1999}. His local characterisation of 2-fibrations (Theorem 2.8) is identical to our definition except that it only requires cartesian 2-cells to be closed under pre-composition with any 1-cell. We insist that cartesian 2-cells also be closed under \emph{post}-composition with any 1-cell. The two definitions are not equivalent. This is illustrated by the following example.

Let $\Fib$ be the 2-category whose objects are fibrations $P$, $Q$ in $\Cat$; 1-cells are pairs of functors $(F,G)$ such that $QF = GP$ and $F$ is cartesian; and 2-cells are pairs of natural transformations $(\alpha,\beta)$ such that $Q\alpha = \beta P$. 
\begin{equation*}
\cd[@R+12pt@C+24pt]{
 E \ar[d]_P \pair{r}{F}{F'} \dtwocell{r}{\alpha} & D \ar[d]^Q \\
 B \pair{r}{G}{G'} \dtwocell{r}{\beta } & C
}
\end{equation*}
There is an obvious projection $\cod\colon\Fib\to\Cat$ that sends $P$ to $B$, $(F,G)$ to $G$ and $(\alpha,\beta)$ to $\beta$. 

This is the 2-functor Hermida takes as a prototype for 2-fibrations. Its cartesian 1-cells are pullback squares and its cartesian 2-cells are pairs of natural transformations whose first component is point-wise cartesian. Pre-composition with 1-cells obviously preserves cartesian 2-cells because it amounts to re-indexing the natural transformations. Post-composition with 1-cells also preserves cartesian 2-cells because the 1-cells in $\Fib$ are cartesian in their first component.

Suppose we modify $\Fib$ by not requiring 1-cells to be cartesian in their first component.
In this case $\cod$ is a Hermida-style 2-fibration and not a 2-fibration by our definition. Thus the definitions are not equivalent.
We will show later that the post-composition property is required to give a correspondence with 2-functors $B^\coop \to \two\Cat$.
\end{Rk}

\begin{Defn}
A \emph{cleavage} for a 2-fibration $P$ is a function $\c{-}{-}$ that describes a \emph{choice} of cartesian lifts.
The 1-cell $\c{f}{e}$ is the chosen cartesian lift of $f\colon b\to Pe$ at $e$.
The same notation is used for chosen cartesian lifts of 2-cells.
A 2-fibration with a cleavage is called a \emph{cloven} 2-fibration.
A \emph{split} 2-fibration is a cloven 2-fibration with the property that
\begin{equation} \label{split_composition}
\begin{aligned}
\c{fg}{e} & = \c{f}{e}.\c{g}{f^*e} \\
\c{\alpha\beta}{k} & = \c{\alpha}{k}.\c{\beta}{\alpha^*k} \\
\c{\alpha*\gamma}{kj} & = \c{\alpha}{k}*\c{\gamma}{j}
\end{aligned}
\end{equation}
and
\begin{equation} \label{split_identities}
\begin{aligned}
\c{1_{Pe}}{e} & =1_{e}\\
\c{1_{Pk}}{k} & =1_{k}.
\end{aligned}
\end{equation}
These conditions specify that cartesian maps be closed under all forms of composition and that cartesian lifts of identities are identities. We say that $P$ is \emph{locally split} when each $P_{xy}$ is split (the second conditions in (1) and (2) above). We say that $P$ is \emph{horizontally split} when chosen cartesian 2-cells are closed under horizontal composition (the third condition in (1) above).
\end{Defn}

\begin{Rk}
Every 2-fibration can be equipped with a cleavage using the axiom of choice.
Since cartesian lifts are unique up to isomorphism the choice of cleavage for a 2-fibration does not significantly affect its behaviour.
As a result we usually suppose that every 2-fibration is cloven and rarely distinguish between one cleavage and another.
\end{Rk}

\begin{Rk}
The three conditions in (\ref{split_composition}) insist that chosen cartesian lifts are closed under composition of 1-cells, composition of 2-cells and horizontal composition of 2-cells. There is some subtlety in the third condition. We could equally well ask that chosen cartesian 2-cells be closed under pre- and post-composition with arbitrary 1-cells:
\begin{equation}
\begin{aligned}
\c{\alpha}{k}j & =\c{\alpha Pj}{kj} \\
k\c{\gamma}{j} & =\c{Pk\gamma}{kj}
\end{aligned}
\end{equation}
for all $j$ and $k$.
The equivalence of these two statements relies on the first two conditions in (1) and the middle-four interchange.
\end{Rk}

\begin{Defn}
Suppose $P\colon E \to B$ and $Q \colon D\to B$ are 2-fibrations. A 2-functor $\eta\colon E \to D$ between fibred 2-categories is \emph{cartesian} when it preserves all cartesian maps and has $Q\eta = P$.
A cartesian 2-functor $\eta\colon E \to D$ is \emph{split} when it preserves choice of cartesian maps:
\begin{equation}
\begin{aligned}
\eta(\c{f}{e}) & = \c{f}{\eta(e)}\\
\eta(\c{\alpha}{k}) & = \c{\alpha}{\eta(k)}.
\end{aligned}
\end{equation}
A 2-natural transformation $\alpha \colon \eta \To \tau$ is \emph{vertical} when $Q\alpha = 1_{P}$.
A modification $\Gamma \colon \alpha \Rrightarrow \beta$ is \emph{vertical} when $Q\Gamma = 1_{1_{P}}$.
\end{Defn}

Suppose that $B$ is a 2-category.
Let $\Fib_s(B)$ be the 3-category of split 2-fibrations over $B$, split cartesian 2-functors, vertical 2-natural transformations and vertical modifications.
Let $\two\Cat$ be the 3-category of 2-categories, 2-functors, 2-natural transformations and modifications.
Let $[B^\coop,\two\Cat]$ be the 3-category of contravariant 2-functors from $B$ to $\two\Cat$, 2-natural transformations, modifications and perturbations.

The following two results will prove useful.

\begin{Prop} \label{prop:chosen_2-cells_enough_strict}
Suppose that $P\colon E\to B$ is a 2-functor and that $\alpha$ and $\beta$ are cartesian. If $\beta*\alpha$ is cartesian then all cartesian 2-cells over $P\alpha$ and $P\beta$ are closed under horizontal composition.
$$\cd[@C+24pt]{ a \dtwocell{r}{\alpha} \pair{r}{f}{h} & b \pair{r}{g}{k} \dtwocell{r}{\beta} & c} $$
\end{Prop}
\begin{proof}
Suppose that $\alpha$, $\beta$ and $\beta*\alpha$ are cartesian and $\gamma$ and $\delta$ are cartesian lifts of $P\alpha$ and $P\beta$.
There exist unique isomorphisms $\eta$ and $\tau$ such that $\gamma = \alpha.\eta$ and $\delta = \beta.\tau$.
Then $\delta*\gamma = (\beta.\tau)*(\alpha.\eta) = (\beta*\alpha).(\tau*\eta)$.
Since $\beta*\alpha$ is cartesian and $\tau*\eta$ is an isomorphism $\delta*\gamma$ is cartesian.
\end{proof}

\begin{Prop} \label{prop:preserve_chosen_maps_enough_strict}
Suppose that $P$ and $Q$ are 2-fibrations and the following diagram commutes.
\begin{equation*}
\cd[]{
 E \ar[dr]_P \ar[rr]^F & & D \ar[dl]^Q \\
 & B &
}
\end{equation*}
If $F$ preserves chosen cartesian maps then it preserves all cartesian maps. 
\end{Prop}
\begin{proof}
Suppose that $\alpha \colon f \To h$ is cartesian in $D$. It factors as $\alpha = \c{P\alpha}{h}.\sigma$ where $\sigma$ is an isomorphism. Then $F\alpha = F\c{P\alpha}{h}.F\sigma$ and since $F$ preserves chosen cartesian maps and isomorphisms $F\alpha$ is cartesian.
The reasoning for cartesian 1-cells is exactly the same.
\end{proof}

\subsection{The Grothendieck construction}

We will describe the Grothendieck construction for split fibred 2-categories: an equivalence
$$\el \colon [B^\coop,\two\Cat] \to \Fib_s(B).$$
We are mostly concerned with its action on objects and use ``Grothendieck construction'' to mean both the action on objects and the whole 3-functor.
This generalises the Grothendieck construction for ordinary fibrations. 
We found that the Grothendieck construction for general 2-fibrations could not be described so neatly as for split 2-fibrations; this is discussed in Remark \ref{strict_non-split_2-fib}.
In section \ref{sec:weak} we will give a more general result: the Grothendieck construction for fibred bicategories.

\begin{Con}[The Grothendieck construction for 2-categories] 
\label{twoGrothCon}
Suppose that $F\colon B^\coop \to \two\Cat$. Let $\el F$ be the 2-category:
\begin{itemize}
\item 0-cells are pairs $(x,\e{x})$ where $x\in B$ and $\e{x}\in Fx$.

\item 1-cells are pairs $(f,\e{f})\colon (x,\e{x}) \to (y,\e{y})$ where
$f\colon x \to y$ and $\e{f}\colon \e{x} \to Ff(\e{y})$.


\item 2-cells are pairs $(\alpha,\e{\alpha})\colon (f,\e{f}) \To (g,\e{g})\colon (x,\e{x}) \to (y,\e{y})$ where
$\alpha\colon f \To g$ and
$$
\cd[@C+24pt]{
\e{x} \ar[r]^-{\e{f}} \ar[dr]_-{\e{g}} & Ff(\e{y}) \\
\drtwocell[0.6]{ur}{\e{\alpha}} & Fg(\e{y}) \ar[u]_{F\alpha_{\e{y}}}}~.$$

\item If $(\alpha,\e{\alpha})$ as above and $(\gamma,\e{\gamma})\colon (g,\e{g}) \To (h,\e{h})\colon (x,\e{x}) \to (y,\e{y})$ then the composite $ (\gamma,\e{\gamma}).(\alpha,\e{\alpha}) $ has first component $\gamma.\alpha$ and second component
\begin{equation*}
\cd[@C+24pt]{
 \e{x} \ar[r]^-{\e{f}} \ar[dr]_-{\e{g}} \ar@/_10pt/[ddr]_-{\e{h}} & Ff(\e{y}) \\
\drtwocell[0.6]{ur}{\e{\alpha}} & Fg(\e{y}) \ar[u]_{F\alpha_\e{y}} \\
 & Fh(\e{y}) \ar[u]_{F\gamma_\e{y}} \drtwocell[0.4]{uul}{\e{\gamma}}
}.
\end{equation*}

\item If $(\alpha,\e{\alpha})$ as above and $(\beta,\e{\beta})\colon (j,\e{j}) \To (k,\e{k})\colon (y,\e{y}) \to (z,\e{z})$ then the composite $ (\beta,\e{\beta})*(\alpha,\e{\alpha})$ has first component $\beta*\alpha$ and second component
\begin{equation*}
\cd[@C+24pt]{
\e{x} \ar[r]^{\e{f}} \ar[dr]_{\e{g}} & Ff(\e{y}) \ar[r]^{Ff(\e{j})} \ar[dr]_{Ff(\e{k})} & FfFj(\e{z}) \\
\drtwocell[0.6]{ur}{\e{\alpha}} & Fg(\e{y}) \ar[dr]_{Fg(\e{k})} \ar[u]_{F\alpha_\e{y}} \drtwocell[0.6]{ur}{Ff(\e{\beta})} & FfFk(\e{z}) \ar[u]_{Ff(F\beta_\e{z})} \\
& & FgFk(\e{z}) \ar[u]_{F\alpha_{Fk(\e{z})}} \\
}.
\end{equation*}

\item Identity 1-cells are $ (1_x, 1_{\e{x}})\colon (x,\e{x}) \to (x,\e{x})$
and identity 2-cells are $ (1_f, 1_{\e{f}})\colon (f,\e{f}) \To (f,\e{f})$.
\end{itemize}

By projecting onto the first component of $\el F$ we obtain a 2-functor $P_F\colon \el F \to B$.
\end{Con}

\begin{Prop}
$P_F \colon \el F \to B$ is a split 2-fibration.
\end{Prop}
\begin{proof}
It is easy to show that $\el F$ is a 2-category. Associativity and unit laws rely on the fact that $F$ is a 2-functor and that it maps into $\two\Cat$. Now notice that in the first component of $\el F$ composition is just composition in $B$. Thus $P_F$ is a 2-functor.

We need to show that $P_F$ is a 2-fibration. Suppose that $(y,\e{y})$ in $\el F$ and $f \colon x\to y$ in $B$. We claim that $(f,1_{Ff(\e{y})}) \colon (x,Ff(\e{y})) \to (y,\e{y})$ is cartesian over $f$. This can be verified by examining the diagram
\begin{equation*}
\cd[@R+12pt@C+24pt]{
(z,\e{z}) \pair{drr}{(g,\e{g})}{(g',\e{g}')} \dltwocell{drr}{(\alpha,\e{\alpha})} \ar@{.>}@/^13pt/[d]^{(h,\e{g})} \dotltwocell{d}{(\beta,\e{\alpha})} \ar@{.>}@/_13pt/[d]_{(h',\e{g}')} & & \\
(x,Ff(\e{y})) \ar[rr]_{(f,1_{Ff(\e{y})})} & & (y,\e{y})
}
\hspace{1cm}
\cd[@R+18pt@C+12pt]{
z \pair{drr}{g}{g'} \dltwocell{drr}{\alpha} \pair{d}{h}{h'} \ltwocell{d}{\beta} & & \\
x \ar[rr]_f & & y
}
\end{equation*}
and showing that the indicated lifts are unique.

Suppose that $(g,\e{g}) \colon (x,\e{x}) \to (y,\e{y})$ in $\el F$ and $\alpha \colon f\To g \colon x\to y$ in $B$. We claim that $(\alpha,1_{F\alpha_\e{y} \e{g}}) \colon (f,F\alpha_\e{y} \e{g})\To(g,\e{g})$ is cartesian over $\alpha$. 
This can be verified by examining the diagram
\begin{equation*}
\cd[@R+12pt@C+12pt]{
(h,\e{h}) \ar@{=>}[dr]^{(\beta,\e{\beta})} \ar@{:>}[d]_{(\sigma,\e{\beta})} & \\
(f,F\alpha_\e{y}\e{g}) \ar@{=>}[r]_-{(\alpha,1_{F\alpha_\e{y}\e{g}})} & (g,\e{g})
}
\hspace{1cm}
\cd[@R+12pt@C+30pt]{
h \ar@{=>}[d]_{\sigma} \ar@{=>}[dr]^{\beta} & \\
f \ar@{=>}[r]_{\alpha} & g
}
\end{equation*}
and showing that the indicated lift is unique.

When
\begin{equation*}
\cd[]{
(x,\e{x}) \ar[r]^{(f,\e{f})} & (y,\e{y}) \ar[r]^{(g,\e{g})} & (z,\e{z}) \\
x \pair{r}{f'}{f} \dtwocell{r}{\alpha} &
y \pair{r}{g'}{g} \dtwocell{r}{\sigma} &
z
}
\end{equation*}
the chosen cartesian lifts of $\alpha$ and $\sigma$ compose to give
\begin{align*}
(\alpha,1_{F\alpha_\e{y} \e{f}}) * (\sigma,1_{F\sigma_\e{z} \e{g}})
& = (\alpha*\sigma, Ff'(1_{F\sigma_\e{z}\e{g}})*1_{F\alpha_\e{y}\e{f}})\\
& = (\alpha*\sigma, 1_{Ff'(F\sigma_\e{z}\e{g})F\alpha_\e{y}\e{f}})\\
& = (\alpha*\sigma, 1_{F(\sigma*\alpha)_\e{z}Ff(\e{g})\e{f}})
\end{align*}
because $Ff$ is a 2-functor, $F\alpha$ is 2-natural and $F$ is a 2-functor. Thus \emph{chosen} cartesian 2-cells are closed under horizontal composition and by Proposition \ref{prop:chosen_2-cells_enough_strict} \emph{all} cartesian 2-cells are closed under horizontal composition. Thus $P_F$ is a 2-fibration.

We need to show that $P_F$ is split. The equations above demonstrate that $P_F$ is horizontally split. The other two conditions are a matter of routine verification.


\end{proof}

\begin{Con}[Pseudo-inverse to the Grothendieck construction] 
\label{twoGrothConInv}
Suppose that $P\colon E \to B$ is a split 2-fibration.
We define a functor $F_P\colon B^\coop \to \two\Cat$ as follows:

\vspace{4pt}
\noindent \emph{on 0-cells:} $F_Pb = E_b$ for all $b\in B$. $E_b$ is the fibre of $P$ over $b$ i.e.\ the sub-category of $E$ with 0-, 1-, and 2-cells being those that map to $b$, $1_b$ and $1_{1_b}$.

\vspace{4pt}
\noindent \emph{on 1-cells:} $F_Pf = f^*\colon E_{b'} \to E_b$ is the 2-functor described by the following diagram.
\begin{equation*}
\cd[@C+32pt@R+18pt]{
f^*e \dotpair{d}{f^*k}{f^*h} \dotrtwocell{d}{f^*\alpha} \ar[r]^{\c{f}{e}} &
e \pair{d}{k}{h} \rtwocell{d}{\alpha} \\
f^*e' \ar[r]_{\c{f}{e'}} & e'
}
\Mapsto{P}
\cd[@C+32pt@R+18pt]{
b \ar[r]^f \pair{d}{1_b}{1_b} \twoeq{d} & b' \pair{d}{1_{b'}}{1_{b'}} \twoeq{d}\\
b \ar[r]_f & b'
}
\end{equation*}
It sends $e$ to the domain of $\c{f}{e}$. It sends $h$, $k$ to the unique $f^*h$, $f^*k$ over $1_b$ generated by the cartesian 1-cell $\c{f}{e'}$ and $\alpha$ to the unique $f^*\alpha$ over $1_{1_b}$.

\vspace{4pt}
\noindent \emph{on 2-cells:} $F_P\sigma=\sigma^*\colon g^* \To f^*\colon E_{b'} \to E_b$ is the 2-natural transformation described by the following diagram.
\begin{equation*}
\cd[@R-24pt@C+20pt]{
 & f^*e \ar@/^5pt/[dr]^{\c{f}{e}} & \\
g^*e \ar@{.>}@/^5pt/[ur]^{\sigma_e^*} \ar@/_16pt/[rr]_{\c{g}{e}} \dtwocell[0.35]{rr}{~\c{\sigma}{\c{g}{e}}} &  & e \\
}
\Mapsto{P}
\cd[@C+20pt]{
b \pair[16pt]{rr}{f}{g} \dtwocell{rr}{\sigma} & & b'
}
\end{equation*}
We take the cartesian lift of $\sigma$ at $\c{g}{e}$ and uniquely factorise its domain as $\c{f}{e} \sigma_e^*$. Then $(F_P\sigma)_e = \sigma_e^*$. This unique factorisation is explained in Proposition \ref{prop:1-cells_factor_strict}.

\end{Con}

\begin{Prop}
$F_P$ is a 2-functor $B^\coop \to \two\Cat$.
\end{Prop}
\begin{proof}
First of all, it is clear that $Fb = E_b$ is well defined as a 2-category.
Second, when $f\colon b \to b'$ in $B$ we get $ \c{f}{e'}(f^*\alpha.f^*\beta) = (\alpha.\beta)\c{f}{e}$.
Thus by the uniqueness of liftings along $\c{f}{e}$ we have $f^*(\alpha.\beta)= f^*\alpha. f^*\beta$.
The diagrams are:
\begin{equation*}
\cd[@C+1.0cm@R+0.6cm]{
f^*e \ar[r]^{\c{f}{e }} \triple{d}{f^*k}{f^*h}{f^*g}{18pt} \lltwocell{d}{f^*\alpha}{f^*\beta} & e \triple{d}{k}{h}{g}{18pt} \lltwocell{d}{\alpha}{\beta}\\
f^*e' \ar[r]_{\c{f}{e'}} & e'
}
\hspace{1cm}
\cd[@C+1.0cm@R+0.6cm]{
f^*e \ar[r]^{\c{f}{e }} \pair[12pt]{d}{f^*h}{f^*h} \ltwocell{d}{f^*1_h} & e \pair[12pt]{d}{h}{h} \ltwocell{d}{1_h}\\
f^*e' \ar[r]_{\c{f}{e'}} & e'
}
\end{equation*}
Similarly, we have $ \c{f}{e'}1_{f^*h} = 1_h\c{f}{e}$ and thus $f^*1_h = 1_{f^*h}$.
Very similar arguments tell us that $f^*(\beta*\alpha) = f^*\beta * f^*\alpha$ and $f^*(1_e) = 1_{f^*e}$. Thus $f^*$ is a 2-functor.

On 2-cells, when $\sigma\colon g \To f$ we display the 2-natural transformation $\sigma^*$ as
\begin{equation*}
\cd[@C+1.0cm]{
 & g^*e \ar[dr]^{\c{g}{e }} \dotpair{dd}{}{} \dotltwocell[0.65]{dd}{} & \\
f^*e \dtwocell[0.35]{urr}{\c{\sigma}{\c{g}{e }}} \ar[rr]_(0.75){\c{f}{e }} \ar[ur]^{\sigma^*_{e }} \pair{dd}{f^*h}{f^*k} \ltwocell{dd}{f^*\alpha} & & e \pair{dd}{h}{k} \ltwocell{dd}{\alpha}\\
 & g^*e' \ar[dr]^(0.3){\c{g}{e'}}  & \\
f^*e' \ar[rr]_{\c{f}{e'}} \ar[ur]^(0.7){\sigma^*_{e'}} \dtwocell[0.35]{urr}{\c{\sigma}{\c{g}{e'}}} & & e'
}
\end{equation*}
where the 2-cell at the back is $g^*\alpha$.
We want to show that $g^*\alpha. \sigma^*_e$ equals $\sigma^*_{e'} .f^*\alpha $. 
For simplicity, let $\eta = \c{\sigma}{\c{g}{e }}$ and $\tau = \c{\sigma}{\c{g}{e'}}$. 
Since $g^*\alpha \sigma^*_e$ and $\sigma^*_{e'} f^*\alpha $ both map down to $1_b$, if we can show that $\c{g}{e'} g^*\alpha \sigma^*_e = \c{g}{e'} \sigma^*_{e'} f^*\alpha $ then they must be equal because of uniquness of lifts along cartesian 1-cells. 
By the same reasoning, this last equation will hold if we can show that $\tau f^*k.\c{g}{e'} (g^*\alpha) \sigma^*_e = \tau f^*k.\c{g}{e'} (\sigma^*_{e'}) f^*\alpha$. 
Now we use various combinations of the middle-four interchange, cartesian 2-cells are split, the front square commutes and back-right square commutes to obtain:
\begin{align*}
\tau f^*k.\c{g}{e'} (\sigma^*_{e'}) f^*\alpha & = \c{f}{e'} f^*\alpha . \tau f^*h \\
& = \alpha \c{f}{e} . h\eta \\
& = \alpha \c{g}{e} \sigma^*_e . k\eta \\
& = \tau.\c{g}{e'} (g^*\alpha) \sigma^*_e\ .
\end{align*}
Thus the back-left composites are equal and $\sigma^*$ is a 2-natural transformation.

Now we need to show that $F_P$ preserves composition and identities.
For vertical composition of 2-cells, we get
\begin{equation*}
\cd[@R-16pt@C+16pt]{
 & h^*e \ar[dr]^{\c{h}{e}} & \\
g^*e \ar[ur]^{(\alpha\beta)^*_e} \ar@/_18pt/[rr]_{\c{g}{e}}
\dtwocell[0.35]{rr}{\c{\alpha\beta}{\c{g}{e}}} & & e \\
}
\hspace{1cm}
\cd[@R-16pt@C+16pt]{
& & h^*e \ar[ddr]^{\c{h}{e}} & \\
& f^*e \ar[ur]^{\beta^*_e} \ar@/_4pt/[drr]_(0.6){\c{f}{e}} \dtwocell[0.3]{rr}{\c{\beta}{\c{f}{e}}} & &\\
g^*e \ar[ur]^{\alpha^*_e} \ar@/_18pt/[rrr]_{\c{g}{e}} \dtwocell[0.35]{rr}{\c{\alpha}{\c{g}{e}}} & & & e \\
}
\end{equation*}
and assert that
$\c{\alpha\beta}{\c{g}{e}} = \c{\alpha}{\c{g}{e}} . \c{\beta}{\c{f}{e}}\alpha^*_e$
because of the splitness conditions. Since $(\alpha\beta)^*_e$ is defined via unique factorisation it follows directly that $(\alpha\beta)^* = \beta^* \alpha^*$. Similarly, $\c{1_g}{\c{g}{e}} = 1_{\c{g}{e}}$ and hence $ (1_g)^* = 1_{g^*}$. Thus 2-cell composition is preserved.

For composition of 1-cells, we get
\begin{equation*}
\cd[@C+48pt@R+24pt]{
g^*f^*e \ar[r]^{\c{g}{f^*e}} \pair[14pt]{d}{g^*f^*h}{g^*f^*k} \rtwocell{d}{g^*f^*\alpha}
& f^*e \ar[r]^{\c{f}{e}} \pair[14pt]{d}{f^*h}{f^*k} \rtwocell{d}{f^*\alpha}
& e \pair[14pt]{d}{k}{h} \rtwocell{d}{\alpha} \\
g^*f^*e' \ar[r]_{\c{g}{f^*e'}} & f^*e' \ar[r]_{\c{f}{e'}} & e'
}
\end{equation*}
and assert that
$\c{gf}{e} = \c{f}{e} . \c{g}{f^*e}$
because of the splitness conditions. It follows directly that $f^*g^*=(gf)^*$. Similarly, $\c{1_{b'}}{e} = 1_e$ and hence $(1_{b'})^* = 1_{E_{b'}}$. Thus 1-cell composition is preserved.

For horizontal composition of 2-cells we get
\begin{equation*}
\cd[@!C@C+6pt@R-24pt]{
 & & h^*k^*e \ar[dr]^{\c{h}{k^*e}} & &\\
 & h^*f^*e \ar[dr]^(0.7){\c{h}{f^*e}} \ar[ur]^{h^*\tau^*_e} & & k^*e \ar[dr]^{\c{k}{e}} & \\
g^*f^*e \dtwocell[0.35]{rr}{\c{\sigma}{\c{g}{f^*e}}} \ar@/_12pt/[rr]_{\c{g}{f^*e}} \ar[ur]^{\sigma^*_{f^*e}} & & f^*e \dtwocell[0.35]{rr}{\c{\tau}{\c{f}{e}}} \ar[ur]^{\tau^*_e} \ar@/_12pt/[rr]_{\c{f}{e}} & & e
}
\end{equation*}
and
\begin{equation*}
\cd[@!C@C-6pt@R-24pt]{
 & & hk^*e \ar[ddrr]^{\c{hk}{e}} & & \\
 & & & & \\
gf^*e \dtwocell[0.4]{rrrr}{\c{\sigma\tau}{\c{gf}{e}}} \ar@/_12pt/[rrrr]_{\c{gf}{e}} \ar[uurr]^{(\sigma\tau)^*_{e}} & & & & e
}.
\end{equation*}
The top-left 1-cells are the components of $(\sigma\tau)^*$ and $\tau^*\sigma^*$. By the splitness conditions on 2-cells these two diagrams are equal as 2-cells. By the splitness conditions on 1-cells the top-right 1-cells are equal. Finally, by uniqueness of factorisation, the top-left 1-cells are equal.
\end{proof}

To prove that the Grothendieck construction is surjective up to isomorphism, we will need the following two results.

\begin{Prop} \label{prop:1-cells_factor_strict}
Suppose $P\colon E\to B$ is a 2-fibration. Every $f\colon x \to z$ in $E$ factors uniquely as
\begin{equation*}
\cd[@C+12pt]
{
x \ar[d]_{\hat{f}} \ar[dr]^f & \\
 y \ar[r]_{\c{Pf}{z}} & z
}
\end{equation*}
where $P\hat{f} = 1_{Px}$.
\end{Prop}
\begin{proof}
Simply note that
\begin{equation*}
\cd[@C+12pt]
{
 x \ar[dr]^f & \\
 y \ar[r]_{\c{Pf}{z}} & z
}
\hspace{1cm}\text{over}\hspace{1cm}
\cd[@C+12pt]
{
Px \ar[d]_{1_{Px}} \ar[dr]^{Pf} & \\
 Px \ar[r]_{Pf} & Pz
}
\end{equation*}
and there exists a unique $\hat{f}\colon x \to y$ with $P\hat{f} = 1_{Px}$ and $f = \c{Pf}{z} \hat{f}$.
\end{proof}

\begin{Prop} \label{prop:2-cells_factor_strict}
Suppose $P\colon E\to B$ is a 2-fibration. Every $\alpha\colon f \To g$ factors uniquely as
\begin{equation*}
\cd[]
{
w \pair{drr}{f}{g} \dtwocell{drr}{\alpha} & & \\
 & & z
}
=
\cd[@C+16pt@R-6pt]
{
w \ar@/^12pt/[rr]^{\hat{f}} \ar[dr]_{\hat{g}} \dtwocell[0.5]{rr}{\hat{\alpha}} & & y \ar[dr]^{\c{Pf}{z}} & \\
 & x \ar[ur]^(0.6){\hat{h}} \ar@/_12pt/[rr]_{\c{Pg}{z}} \dtwocell[0.3]{rr}{\c{P\alpha}{\c{Pg}{z}}} & & z
}
\end{equation*}
where $P\hat{f} = P\hat{g} = P\hat{h} = 1_{Pw}$ and $P\hat{\alpha}=1_{1_{Pw}}$.
\end{Prop}
\begin{proof}
The proof is similar to that above but somewhat more involved.
Begin by uniquely factoring $f = \c{Pf}{z}\cdot\hat{f}$ and $g = \c{Pg}{z}\cdot\hat{g}$ by Proposition \ref{prop:1-cells_factor_strict}.
Now take the cartesian lift of $P\alpha$ at $\c{Pg}{z}$:
\begin{equation*}
\cd[@C=2.8cm]
{
 x  \pair[14pt]{r}{h}{\c{Pg}{z}} \dtwocell[0.22]{r}{\c{P\alpha}{\c{Pg}{z}}} & z
}
\hspace{1cm}\text{over}\hspace{1cm}
\cd[@C+18pt]
{
 Px \pair[14pt]{r}{Pf}{Pg} \dtwocell[0.45]{r}{P\alpha} & Pz
}.
\end{equation*}
Then $\c{P\alpha}{\c{Pg}{z}}\hat{g}$ is cartesian over $P\alpha$ and
\begin{equation*}
\cd[@C+18pt]
{
 f \ar@{=>}[dr]^{\alpha} & \\
 h\hat{g} \ar@{=>}[r]_{\c{P\alpha}{\dots}\hat{g}} & g
}
\hspace{1cm}\text{over}\hspace{1cm}
\cd[@C+18pt]
{
Pf \ar@{=>}[d]_{1_{Pf}} \ar@{=>}[dr]^{P\alpha} & \\
 Pf \ar@{=>}[r]_{P\alpha} & Pg
}
\end{equation*}
so there exists a unique $\eta$ with $P\eta = 1_{Pf}$ and
\begin{equation*}
\cd[]
{
w \pair{drr}{f}{g} \dtwocell{drr}{\alpha} & & \\
 & & z
}
=
\cd[@C+16pt@R-6pt]
{
w \ar@/^12pt/[rr]^{\hat{f}} \ar[dr]_{\hat{g}} & & y \ar[dr]^{\c{Pf}{z}} &  \\
\dtwocell[0.7]{urr}{\eta} & x \pair[12pt]{rr}{h}{\c{Pg}{z}} \dtwocell[0.3]{rr}{\c{P\alpha}{\c{Pg}{z}}} & & z
}.
\end{equation*}
By Proposition \ref{prop:1-cells_factor_strict} we factor $h = \c{Pf}{z}\cdot\hat{h}$ uniquely where $P\hat{h} = 1_{Px}$. Finally, we observe that
\begin{equation*}
\cd[@C+38pt@R+18pt]{
w \pair[8pt]{dr}{f}{} \dltwocell{dr}{\eta} \pair[8pt]{d}{\hat{f}}{\hat{h}\hat{g}} & \\
y \ar[r]_{\c{Pf}{z}} & z
}
\hspace{1cm}\text{over}\hspace{1cm}
\cd[@C+38pt@R+18pt]{
Pw \pair[8pt]{dr}{Pf}{Pf} \dltwocell[0.5]{dr}{1_{Pf}} \pair[8pt]{d}{1_{Pw}}{1_{Pw}} \ltwocell{d}{1_{1}} & \\
Pw \ar[r]_{Pf} & Pz
}
\end{equation*}
so there exists a unique $\hat{\alpha}\colon \hat{f} \To \hat{h}\hat{g}$ over $1_{1_{Pw}}$ with $\eta = \hat{\alpha}\c{Pf}{z}$ and hence a unique factorisation of $\alpha$ as stated.
\end{proof}

\begin{Rk}
This last result (Proposition \ref{prop:2-cells_factor_strict}) is recognised by Hermida in Proposition 2.4 of \cite{hermida1999}. He doesn't explicitly mention the uniqueness of such factorisations.
\end{Rk}

\begin{Rk}
The factorisations of Proposition \ref{prop:1-cells_factor_strict} and \ref{prop:2-cells_factor_strict} are unique up to choice of cartesian lift. In each of these results we have implicitly supposed that $P$ is cloven and that the factorisation occurs through the chosen cartesian lift. In fact, there is a unique factorisation for every cleavage on the fibration.
\end{Rk}

\begin{Con}[Surjective up to isomorphism] 
In order to show that the Grothendieck construction is surjective up to isomorphism, we need to find for every $P$ an invertible map of fibrations
\begin{equation*}
\cd[@C-1pt]{
\el{F_P} \ar[rr]^H \ar[dr]_{\pi}& & \ar[dl]^{P} E \\
 & B &
}.
\end{equation*}
First, what is $\el F_P$? Its data consists of:
\begin{enumerate}[1\emph{-cells:}]
\setcounter{enumi}{-1}
\item pairs $(x,\e{x})$ where $\e{x} \in E$ and $P\e{x}=x$.
\item pairs $(f,\e{f})\colon (x,\e{x}) \to (y,\e{y})$ where $f\colon x\to y$ in $B$ and $\e{f}~\colon~\e{x}\to~f^*(\e{y})$ in $E_x$.
\item pairs $(\alpha,\e{\alpha})\colon (f,\e{f}) \To (g,\e{g})\colon (x,\e{x}) \to (y,\e{y})$ where $\alpha\colon f\to g$ in $B$ and $\e{\alpha}\colon \e{f} \to \alpha^*_{\e{y}} \e{g}$ in $E_x$.
\begin{equation*}
\cd[@C+18pt]{
\e{x} \ar[r]^{\e{f}} \ar[dr]_{\e{g}} & f^*({\e{y}}) \\
\drtwocell[0.65]{ur}{\e{\alpha}} & g^*(\e{y}) \ar[u]_{\alpha^*_{\e{y}}}
}
\end{equation*}
\end{enumerate}

We define $H\colon \el F_P \to E$ by
\begin{equation*}
\cd[@C+1.0cm]{
(x,\e{x}) \pair[16pt]{r}{(f,\e{f})}{(g,\e{g})} \dtwocell[0.4]{r}{(\alpha,\e{\alpha})} & (y,\e{y})
}
\hspace{0.5cm}{\longmapsto}\hspace{0.5cm}
\cd[@C+2.5cm]{
\e{x} \pair[16pt]{r}{\c{f}{\e{y}}\e{f}}{\c{g}{\e{y}}\e{g}} \dtwocell[0.2]{r}{\c{\alpha}{\e{y}}\e{g}.\c{f}{\e{y}}\e{\alpha}} & \e{y}
}.
\end{equation*}
The action on 2-cells is to send $(\alpha,\e{\alpha})$ to
\begin{equation*}
\cd[@C+16pt@R-6pt]{
\e{x} \dtwocell[0.5]{rr}{\e{\alpha}} \ar@/^12pt/[rr]^{\e{f}} \ar[dr]_{\e{g}} & & f^*{\e{y}} \ar[dr]^{\c{f}{\e{y}}} & \\
 & g^*\e{y} \ar[ur]^(0.6){\alpha^*_{\e{y}}} \ar@/_12pt/[rr]_{\c{g}{\e{y}}} \dtwocell[0.37]{rr}{\c{\alpha}{\c{g}{\e{y}}}} & & \e{y}
}.
\end{equation*}
The reader can verify that this is 2-functorial.
\end{Con}

\begin{Prop} \label{H_is_iso} 
$H$ is a split cartesian isomorphism.
\end{Prop}
\begin{proof}
First,
\begin{align*}
\pi(x,\e{x})   = x   &= P\e{x}   = PH(x,\e{x}) \\
\pi(f,\e{f})   = f   &= P\e{f}   = PH(f,\e{f}) \\
\pi(\alpha,\e{\alpha}) = \alpha &= P\e{\alpha} = PH(\alpha,\e{\alpha})
\end{align*}
so $\pi = PH$. Second, the chosen cartesian maps in $\el F_P$ are those with identities in the second component. Since $H$ acts by post-composition with chosen cartesian maps it is split cartesian. Third, for every $e\in E$ there exists a unique $(Pe,e)\in \el F_P$ with $H(Pe,e) = e$ so $H$ is bijective on objects. Then Proposition \ref{prop:1-cells_factor_strict} tells us that for every $f\in E$ there exists a unique $\hat{f}$ with $H(f,\hat{f}) = \c{f}{e}\hat{f} = f$ so $H$ is bijective on 1-cells. Proposition \ref{prop:2-cells_factor_strict} gives the same result on 2-cells. Thus $H$ is an isomorphism.
\end{proof}

\begin{Thm}\label{Thm:GC1}
The Grothendieck construction is the action on objects of a 3-functor
$$\el\colon [B^\coop, \two\Cat] \to \two\Fib_s(B)$$
and this is an equivalence.
\end{Thm}
\begin{proof}
We have already shown that $\el$ is surjective up to isomorphism on objects (Proposition \ref{H_is_iso}).
To show that $\el$ is an equivalence we need to define its action on 1,2,3-cells and show that it is locally an isomorphism.

Suppose $\eta \colon F \To G$ is a 2-natural transformation in $[B^\coop,\two\Cat]$. We define $\el{\eta} \colon \el F \to \el G$ by
\begin{equation*}
\cd[@C+64pt@R+18pt]{
 (x,\e{x}) \pair[20pt]{d}{(g,\e{g})}{(f,\e{f})} \rtwocell{d}{(\alpha,\e{\alpha})} \\
(y,\e{y})
}
\Mapsto{}
\cd[@C+64pt@R+18pt]{
 (x,\eta_x\e{x}) \pair[20pt]{d}{(g,\eta_x\e{g})}{(f,\eta_x\e{f})} \rtwocell{d}{(\alpha,\eta_x\e{\alpha})} \\
(y,\eta_y\e{y})}.
\end{equation*} This is a split cartesian 2-functor from $P_F$ to $P_G$.
Suppose $\Gamma \colon \eta \Rrightarrow \epsilon$ is a modification in $[B^\coop,\two\Cat]$. We define $\el{\Gamma} \colon \el{\eta} \To \el{\epsilon}$ by
$$\el{\Gamma}_{(x,\e{x})} = (1_x,({\Gamma_x})_\e{x}) \colon (x,\eta_x \e{x}) \to (x, \epsilon_x \e{x}).$$
where $ \el{\Gamma}_{(x,\e{x})} \colon \el{\eta}(x,\e{x}) \to \el{\epsilon}(x,\e{x})$. This is a vertical 2-natural transformation.
Suppose $\zeta \colon \Gamma\to\Lambda$ is a perturbation in $[B^\coop,\two\Cat]$. We define $\el{\zeta} \colon \el{\Gamma}\Rrightarrow\el{\Lambda}$ by
$$\el{\zeta}_{(x,\e{x})} = (1_{1_x},({\zeta_x})_\e{x}) \colon (1_x,({\Gamma_x})_\e{x}) \To (1_x, ({\Lambda_x})_\e{x})$$
where $\el{\zeta}_{(x,\e{x})} \colon \el{\Gamma}_{(x,\e{x})} \To \el{\Lambda}_{(x,\e{x})}$. This is a vertical modification.
This defines $\el$ on 1,2,3-cells and it is 3-functorial.

Now suppose that $\eta \colon \el F \to \el G$ is a split cartesian 2-functor. Define $\bar{\eta}\colon F \To G$ by
$\bar{\eta}_x(a) = \pi_2(\eta(x,a))$,
$\bar{\eta}_x(f) = \pi_2(\eta(1_x,f))$,
$\bar{\eta}_x(\sigma) = \pi_2(\eta(1_{1_x},\sigma))$. This is 2-natural because $\eta$ is a split cartesian 2-functor. Then $\el(\bar{\eta}) = \eta$ and is unique with that property. Thus $\el$ is bijective on 1-cells.

Suppose that $\Gamma \colon \el{\eta} \To \el{\epsilon}$ is a vertical 2-natural transformation.
Define $\bar{\Gamma}\colon \eta \Rrightarrow \epsilon$ by $(\bar{\Gamma}_x)_{\e{x}} = \pi_2(\Gamma_{(x,\e{x})})$.
This is a modification because $\Gamma$ is 2-natural and $\eta$, $\epsilon$ are split cartesian.
Then $\el(\bar{\Gamma}) = \Gamma$ and is unique with that property. Thus $\el$ is bijective on 2-cells.

Suppose that $\theta \colon \el{\Gamma} \Rrightarrow \el{\Lambda}$ is a vertical modification.
Define $\bar{\theta}\colon \Gamma \Rrightarrow \Lambda$ by $(\bar{\theta}_x)_{\e{x}} = \pi_2(\theta_{(x,\e{x})})$.
This is a perturbation because $\theta$ is a modification and $\eta$, $\epsilon$ are split cartesian.
Then $\el(\bar{\theta}) = \theta$ and and is unique with that property. Thus $\el$ is bijective on 3-cells.

This makes $\el$ locally an isomorphism and thus an equivalence.
\end{proof}

\begin{Rk}
The action of the Grothendieck construction on objects is described by Bakovic in \cite{bakovic2012} section 6. Section 5 of the same paper gives a partial description of the action on objects of the pseudo-inverse. With some adjustments, we have completed the second construction (Theorem 5.1) and shown that together they form an equivalence of 3-categories. It was in completing Bakovic's description of the pseudo-inverse to the Grothendieck construction that we discovered that cartesian 2-cells must be closed under post-composition with all 1-cells.
\end{Rk}

\begin{Rk}[Non-split 2-fibrations]\label{strict_non-split_2-fib}
The Grothendieck construction for non-split 2-fibrations is somewhat more complicated than demonstrated above. We chose to first build the Grothendieck construction for fibred bicategories and then to observe how the arguments simplify when restricted to 2-fibrations. We found that non-split 2-fibrations correspond to a slightly odd kind of trihomomorphism (see Remark \ref{rk:variation_weak}). We found however that 2-fibrations that are locally and horizontally split correspond to homomorphisms of $\two\Cat$-enriched bicategories $B^\coop \to \two\Cat$ (enriched in $\two\Cat$ as a monoidal bicategory). This is somewhat more pleasing.
\end{Rk}

\begin{Rk}[Dual constructions]
All of these results could be adjusted to describe three other kinds of fibrations: \emph{op-2-fibrations}, \emph{co-2-fibrations} and \emph{coop-2-fibrations}. 
They correspond to `op'-contravariant, `co'-contravariant and covariant 2-functors into $\two\Cat$. 
The dual Grothendieck constructions are obtained by reversing the direction of the second components of 1- and 2-cells in $\el F$ (in 2-cells, in 1-cells or in both).
We could reasonably refer to coop-2-fibrations as \emph{2-opfibrations}. 
In that case cartesian 1- and 2-cells would be defined using pullbacks associated with pre-composition instead of post-composition. 
\end{Rk}

\subsection{Examples} \label{sec:examples_a}

When $C$ is a category, $\Fam(C)$ is the category of families of objects of $C$. The objects of $\Fam(C)$ are pairs $(I,X)$ where $I$ is a set and $X\colon I \to C$ is a functor.
The 1-cells are pairs $(u,\alpha)\colon (I,X) \to (J,Y)$ where $u\colon I \to J$ is a function and $\alpha$ is a natural transformation.
\begin{equation*}
\cd[]
{
 I \ar[dr]_{X} \ar[rr]^{u} & \rtwocell{d}{\alpha} & J \ar[dl]^{Y}\\
 & C &
}
\end{equation*}
Composition and identities are defined in the obvious way.
There is a functor $\pi \colon \Fam(C) \to \Set$ that is the projection onto the first component of $\Fam(C)$.
This is a well-known example of a fibration.

\begin{Con}[Families]
When $B$ is a 2-category we define $\Fam(B)$ to be the 2-category of `1-cell diagrams' in $B$. The objects of $\Fam(B)$ are pairs $(C,X)$ where $C$ is a small category and $ X\colon C^\op \to B$ is a pseudo-functor. The 1-cells are pairs $(F,\alpha)\colon (C,X) \to (D,Y)$ where $F \colon C \to D$ is a functor and $\alpha \colon X \To YF^\op$ is a pseudo-natural transformation.
The 2-cells are pairs $(\sigma, \Sigma) \colon (F,\alpha) \To (G,\beta)$ where $\sigma\colon F \To G$ is a natural transformation and $\Sigma$ is a modification
\begin{equation*}
\cd[@R+6pt]
{
 C^\op \ar[dr]_{X} \ar[rr]^{F^\op} & \rtwocell[0.55]{d}{\alpha} & D^\op \ar[dl]^{Y} & C^\op \ar[dr]_{X} \pair[8pt]{rr}{F^\op}{G^\op} \utwocell{rr}{\sigma^\op} & \rtwocell[0.65]{d}{\beta} & D^\op \ar[dl]^{Y} \\
 & B & \rthreecell{ur}{\Sigma} & & B &
}.
\end{equation*}
Composition and identities are defined in the obvious way.
There is a 2-functor $\pi \colon \Fam(B) \to \Cat$ defined by projection onto the first component of $\Fam(B)$.
\end{Con}

\begin{Ex} \label{ex:Fam_strict}
$\pi \colon \Fam(B) \to \Cat$ is a 2-fibration.
\end{Ex}
\begin{proof}
Suppose $(D,Y)$ in $\Fam(B)$ and $F\colon C \to D$ in $\Cat$. Its cartesian lift is $(F,1_{YF^\op})\colon(C,YF^\op)\to(D,Y)$.
Suppose that $(\sigma,\Sigma)\colon(G,\beta)\To(H,\gamma)\colon(E,Z)\to(D,Y)$ and $\sigma F = \lambda$ where $\lambda \colon J \To K$.
The unique lift of $\lambda$ is $(\lambda,\Sigma)\colon(J,\beta)\To(K,\gamma)\colon(E,Z)\to(C,YF^\op)$. The diagrams are
\begin{equation*}
\cd[@R+12pt@C+36pt]
{
(E,Z) \pair{dr}{(G,\beta)}{(H,\gamma)} \dltwocell{dr}{(\sigma,\Sigma)} \ar@{.>}@/^12pt/[d]^(0.65){(J,\beta)} \ar@{.>}@/_12pt/[d]_{(K,\gamma)} \dotltwocell{d}{(\lambda,\Sigma)} & \\
(C,YF^\op) \ar[r]_{(F,1)} & (D,Y)
}
\hspace{0.5cm}\text{over}\hspace{0.5cm}
\cd[@R+12pt@C+48pt]
{
E \pair[8pt]{dr}{G}{H} \dltwocell{dr}{\sigma} \pair[8pt]{d}{J}{K} \ltwocell{d}{\lambda} & \\
C \ar[r]_{F} & D
}.
\end{equation*}

Suppose $(G,\beta)\colon(C,X)\to(D,Y)$ in $\Fam(B)$ and $\sigma\colon F\To G\colon C \to D$ in $\Cat$. Its cartesian lift is $(\sigma,1_{Y\sigma.\beta})\colon(F,Y\sigma.\beta)\To(G,\beta)$.
Suppose that $(\lambda,\Lambda)\colon(H,\gamma)\To(G,\beta)$ and $\sigma \omega = \lambda$.
The unique lift of $\omega$ is $(\omega,\Lambda)\colon(H,\gamma)\To(F,Y\sigma^\op.\beta)$.
\begin{equation*}
\cd[@R+6pt]
{
(H,\gamma) \ar@{=>}[dr]^{(\lambda,\Lambda)} \ar@{:>}[d]_{(\omega,\Lambda)}& \\
(F,Y\sigma^\op.\beta) \ar@{=>}[r]_-{(\sigma,1)} & (G,\beta)
}
\hspace{1cm}
\cd[@R+6pt@C+16pt]
{
H \ar@{=>}[dr]^{\lambda} \ar@{=>}[d]_{\omega}& \\
F \ar@{=>}[r]_{\sigma} & G
}
\end{equation*}

Suppose that $\sigma\colon F \To G$ and $\tau\colon H \To K$ with $(G,\beta)\colon(C,X) \to (D,Y)$ and $(K,\delta)\colon (D,Y)\to(E,Z))$ and we compose the chosen cartesian lifts of $\sigma$ and $\tau$.
They are $(\sigma,1)$ and $(\tau,1)$ their composite $(\tau,1)*(\sigma,1)$ is isomorphic to $(\tau*\sigma,1)$ which is cartesian.
Thus by Proposition \ref{prop:chosen_2-cells_enough_strict} all cartesian 2-cells are closed under composition.
\end{proof}

\begin{Rk}
$\pi \colon \Fam(B) \to \Cat$ can be obtained by applying the Grothendieck construction to $F \colon \Cat^\coop \to \two\Cat$ defined on objects by $ F(C) = \Hom(C^\op,B)$.
\end{Rk}

\begin{Rk}
The above construction yields a 2-fibration that is split under composition of cartesian 1-cells but it is not split in any other sense. If we modify this construction by replacing pseudo-functors and pseudo-natural transformations with 2-functors and 2-natural transformations then the result is split in every way. This variation can be obtained by applying the Grothendieck construction to $F$ defined by $ F(C) = [C^\op,B]$.
\end{Rk}

\begin{Defn}
We say that an arrow $p\colon e \to b$ in a 2-category $B$ is a (split) fibration when $p_* \colon B(c,e) \to B(c,b)$ is a (split) fibration for all $c$ and the commuting square 
\begin{equation*}
\cd[]
{
B(c,e) \ar[d]_{p_*} \ar[r]^{f^*} & B(c',e) \ar[d]^{p_*}\\
B(c,b) \ar[r]_{f^*} & B(c',b)
}
\end{equation*}
is a (split) morphism of fibrations for all $f\colon c' \to c$.
\end{Defn}

\begin{Defn}
A morphism between (split) fibrations $p\colon e\to b$ and $q\colon e'\to b'$ in a 2-category $B$ is a pair $(f\colon e \to e', g\colon b \to b')$ where $q.f = g.p$ and 
\begin{equation*}
\cd[]
{
B(c,e) \ar[d]_{p_*} \ar[r]^{f_*} & B(c,e') \ar[d]^{q_*}\\
B(c,b) \ar[r]_{g_*} & B(c,b')
}
\end{equation*}
is a (split) morphism of fibrations for all $c$.
\end{Defn}

\begin{Con}[Internal fibrations]
The category of fibrations internal to a 2-category $B$ is denoted by $\Fib_B$.
The objects are fibrations $p\colon e \to b$ in $B$.
The 1-cells are morphisms of fibrations.
The 2-cells are pairs of 2-cells $(\alpha, \beta) \colon (f,g) \To (f',g')$ with $q\alpha = \beta p$.
Composition and identities are the same as in $B^\mathbbm{2}$.

There is a 2-functor $\cod \colon \Fib_B \to B$ defined by projection onto the codomain:
\begin{equation*}
\cd[@C+16pt@R+12pt]
{
e \pair{r}{f}{f'} \dtwocell{r}{\alpha} \ar[d]_{p} & e' \ar[d]^{q} \\
b \pair{r}{g}{g'} \dtwocell{r}{\beta} & b'
}
\Mapsto{\cod}
\cd[@C+16pt@R+12pt]
{
b \pair{r}{g}{g'} \dtwocell{r}{\beta} & b'
}.
\end{equation*}
When $B=\Cat$ we omit the subscript and $\Fib_\Cat$ is just $\Fib$. It is the category of fibrations in $\Cat$ and the codomain 2-functor is $\cod \colon \Fib \to \Cat$.
\end{Con}

\begin{Ex}
When $B$ has 2-pullbacks, $\cod \colon \Fib_B \to B$ is a 2-fibration.
\end{Ex}
\begin{proof}
Suppose $q\colon e' \to b'$ in $\Fib(B)$ and $g'\colon b \to b'$ in $B$, then there exists a map $(g,g')\colon (g')^*q\to q$ defined by taking the 2-pullback
\begin{equation*}
\cd[]
{
e \pullback \ar[r]^{g}  \ar[d]_{g^*q} & e' \ar[d]^{q} \\
b \ar[r]_{g'} & b'
}.
\end{equation*}
Since both pullbacks and fibrations in $B$ are defined representably and fibrations in $\Cat$ are closed under pullback, fibrations in $B$ must also be closed under pullback. 
The same argument ensures that $g$ is cartesian. 
Thus $(g,g')$ is well-defined as a 1-cell. 
To see that this is cartesian, suppose that $(h,h')\colon r \to q$ in $\Fib(B)$ and $h' = g'.f'$ in $B$. 
Then there exists a unique $f$ with $p.f=f'.r$ and $h=g.f$ and hence a unique $(f,f')$ with $(h,h') = (g,g')(f,f')$ and $\pi(f,f') = f'$. 
We know that $f$ is cartesian because $h$ is cartesian and $g$ reflects cartesian maps (again because pullbacks in $\Cat$ reflect cartesian maps). 
This same argument works for 2-cells into $e'$ so $(g,g')$ is cartesian. The diagram is
\begin{equation*}
\cd[]
{
e'' \ar[drr]^(0.3){h} \ar@{.>}[dr]_{f} \ar[d]_(0.4){r} &   & \\
b'' \ar[dr]_{f'} \ar[drr]^(0.3){h'} & e \pullback \ar[r]_{g}  \ar[d]^(0.4){p} & e' \ar[d]^(0.4){q} \\
   & b \ar[r]_{g'}   & b'
}.
\end{equation*}

Suppose that $(g,g')\colon p \to q$ in $\Fib(B)$ and $\alpha' \colon f' \To g'$ in $B$. Since $q$ is cartesian we can take the cartesian lift of $\alpha' p$ at $g$ (call it $\alpha$) and get a 2-cell $(\alpha,\alpha')\colon(f,f')\To(g,g')$. To show that this is cartesian, suppose that $(\gamma,\gamma')\colon(h,h')\To(g,g')$ and $\gamma' = \eta'\alpha'$. Since $\alpha$ is cartesian for $q$ and $q\gamma = \gamma' p = \eta'p. \alpha'p = \eta'p. q\alpha$ there exists a unique $\eta\colon h\To f$ and hence $(\eta,\eta')$ with $(\gamma,\gamma') = (\alpha,\alpha').(\eta,\eta')$. Thus $(\alpha,\alpha')$ is cartesian.
\begin{equation*}
\cd[@C+6pt]
{
h \ar@{=>}[dr]^{\gamma} \ar@{:>}[d]_{\eta} &  \\
f \ar@{=>}[r]_{\alpha} & g
}
\hspace{1.8cm}
\cd[@C+6pt]
{
qh \ar@{=>}[dr]^{q\gamma = \gamma'p} \ar@{=>}[d]_{\eta'p} &  \\
qf \ar@{=>}[r]_{q\alpha = \alpha'p} & qg
}
\end{equation*}

Suppose that we take the cartesian lifts of $\alpha'\colon f' \To g'$ and $\gamma'\colon h' \To k'$ at $(g,g')$ and $(k,k')$ as indicated below.
Since cartesian 2-cells for $r$ are closed under pre-composition with any 1-cell $\gamma g$ is cartesian.
Also since $h$ preserves cartesian maps for $q$ we know that $h \alpha$ is cartesian.
Then because cartesian 2-cells are closed under vertical composition $\gamma * \alpha = \gamma g.h\alpha$ is cartesian.
Thus by Proposition \ref{prop:chosen_2-cells_enough_strict} cartesian 2-cells are closed under composition.
\begin{equation*}
\cd[@C+12pt@R+6pt]
{
e \ar[d]_{p} \pair{r}{}{g} \dtwocell{r}{\alpha} & e' \ar[d]^{q} \pair{r}{h}{} \dtwocell{r}{\gamma} & e'' \ar[d]^{r} \\
b \pair{r}{}{} \dtwocell{r}{\alpha'} & b' \pair{r}{}{} \dtwocell{r}{\gamma'} & b''
}
\hspace{1cm}
\end{equation*}
\end{proof}

\begin{Rk}
If we apply the pseudo-inverse to the Grothendieck construction to $\cod \colon \Fib_B \to B$ we get $F \colon B^\coop \to \two\Cat$ defined by $ F(b) = \Fib_B/b$, the category of fibrations over $b$. Its action on 1-cells is to send $f\colon b \to b'$ to $f^*\colon \Fib_B/b' \to \Fib_B/b$ defined by pullback.
\end{Rk}

\begin{Rk}
Let $\Fib_B^s$ be the sub-2-category of $\Fib_B$ containing split fibrations and split maps. Suppose also that we can choose 2-pullbacks in $B$ in such a way that they are closed under composition in $B^\mathbbm{2}$ (not just up to isomorphism). Then the proof above requires only slight adjustments to show that $\cod\colon \Fib_B^s\to B$ is a split 2-fibration.
\end{Rk}

\begin{Ex}[Enriched Categories]
There is a 2-functor $\Mon\to\two\Cat$ that maps each monoidal category $\V$ to $\V\mhyphen\Cat$. We can use a dual to the Grothendieck construction to get a 2-opfibration $\Enr\to\Mon$. The objects of the total category $\Enr$ are enriched categories: pairs $(\V,A)$ where~$\V$ is a monoidal category and $A$ is a $\V$-enriched category. The rest of the structure can be deduced from the dual Grothendieck construction.
\end{Ex}

\begin{Ex}[Algebras]
Let $\Mnd(K)$ be the 2-category of 2-monads on a 2-category $K$. There is a 2-functor $F\colon\Mnd(K)^\coop\to\two\Cat$ that maps each 2-monad $T$ to the 2-category $T\mhyphen\Alg_l$ of strict $T$-algebras, lax algebra morphisms and algebra 2-cells. Each monad morphisms $\lambda \colon S \to T$ gives a 2-functor that acts on $m \colon TA \to A$ by precomposition with $\lambda_A\colon SA \to TA$. Each monad 2-cell $\Gamma \colon \lambda \To \tau$ gives a 2-natural transformation whose component at $m \colon TA \to A$ is $(1,1,m.\Gamma_A)$
\begin{equation*}
\cd[@C+0.5cm]{
SA \ar[d]_{\tau_A} \ar[r]^{1} & SA \ar[d]^{\lambda_A}\\
TA \ar[d]_{m}         & TA \ar[d]^{m}\\
A \ar[r]_{1} \dtwocell{uur}{m.\Gamma_A} & A
}.
\end{equation*}
$F$ is contravariant on 2-cells because we're using lax morphisms of algebras. If $F$ mapped each monad to $T\mhyphen\Alg_{oplax}$ containing the oplax morphisms then it would be covariant on 2-cells.
Further details on 2-monads can be found in \cite{kelly1974}. 

We can use the Grothendieck construction to construct an 2-fibration $\Alg\to\Mnd$. The objects of the total category $\Alg$ are algebras of a 2-monad: pairs $(S,(A,m))$ where $m\colon SA \to A$ is an $S$-algebra. The 1-cells from $(S,(A,m))$ to $(T,(B,n))$ are pairs $(\lambda,(f,\theta_f))$ where $\lambda$ is a monad morphism from $S$ to $T$ and $(f,\theta_f)\colon \lambda(A,m) \to F\lambda(B,n)$ is a lax algebra morphism
\begin{equation*}
\cd[]{
SA \ar[d]_{Sf} \ar[rr]^{m} \urtwocell{drr}{\theta_f} && A \ar[d]^{f}\\
SB \ar[r]_{\lambda_B} & TB \ar[r]_n & B
}.
\end{equation*}
The 2-cells of $\Alg$ are pairs $(\Gamma,\alpha)\colon (\lambda,(f,\theta_f))\to (\tau,(g,\theta_g))$ where $\Gamma\colon\lambda\to\tau$ is a monad 2-cell and $\alpha$ is an algebra 2-cell
\begin{equation*}
\cd[@C-20pt]{
(A,m) \ar[dr]_{(g,\theta_g)} \ar[rr]^-{(f,\theta_f)} & \dtwocell{d}{\alpha} & F\lambda(B,n) \\
 &  F\tau(B,n) \ar[ur]_{F\Gamma_{(B,n)}} & 
}.
\end{equation*}
The 2-fibration is projection on the first component of $\Alg$. By construction the fibre over $T$ is $T\mhyphen\Alg_l$.
\end{Ex}


\section{Fibred Bicategories}
\label{sec:weak}

What follows is the theory of fibrations developed specifically for bicategories. The concepts are not significantly different from  Section \ref{sec:strict} but the details are \emph{much} more complicated.

\subsection{Definitions and properties of cartesian 1- and 2-cells}
Let $P \colon \E \to \B$ be a homomorphism of bicategories.

\begin{Defn} \label{def:cartesian1_weak}
We say a 1-cell $f\colon x \to y$ in $\E$ is \emph{cartesian} when it has the following two properties:
\begin{enumerate}
\item Suppose that $g \colon z \to y$ in $\E$ with $h\colon Pz \to Px$ and
an isomorphism $\alpha \colon Pf.h \To Pg$
\begin{equation*}
\cd[@C+15pt@R+15pt]{
z  \ar[dr]^{g}& \\
x \ar[r]_{f} & y
}
\hspace{1.5cm}
\cd[@C+15pt@R+15pt]{
Pz \ar[d]_{h}   \ar[dr]^{Pg}& \\
Px \ar[r]_{Pf} \urtwocell[0.3]{ur}{\alpha} & Py
}.
\end{equation*}
Then there exists an $\hat{h}\colon z \to x$ and isomorphisms $\hat{\alpha} \colon f\hat{h} \To g$, $\hat{\beta} \colon P\hat{h} \To h$ such that $\alpha.Pf\hat{\beta} = P\hat{\alpha}.\phi_{hf}$.
We say that $(\hat{h},\hat{\alpha},\hat{\beta})$ is a \emph{lift} of $(h, \alpha)$.

\item Suppose that $\sigma\colon g \To g'$ in $\E$ and $h,h'\colon Pz \to Px$ with isomorphisms
$\alpha \colon Pf. h \To Pg$ and $\alpha' \colon Pf. h' \To Pg'$. Suppose also that $(h, \alpha)$ and $(h', \alpha')$ have lifts $(\hat{h},\hat{\alpha},\hat{\beta})$ and $(\hat{h'},\hat{\alpha'},\hat{\beta'})$. For any $\delta \colon h \To h'$ in $\B$ with $\alpha'.Pf\delta = P\sigma.\alpha$
\begin{equation*}
\cd[@C+2cm@R+1cm]{
z \ar@/_10pt/[d]_{\hat{h}} \ar@{.>}@/^10pt/[d]^{\hat{h'}} \pair{dr}{g'}{g} \urtwocell{dr}{\sigma} & \\
x \ar[r]_{f} \rtwocell[0.2]{ur}{\hat{\alpha}}
 & y \dotrtwocell[0.7]{ul}{\hat{\alpha'}}
}
\hspace{1cm}
\cd[@C+2cm@R+1cm]{
Pz \ar@/_10pt/[d]_{h} \ar@/^10pt/@{.>}[d]^{h'} \doturtwocell{d}{\delta} \pair{dr}{Pg'}{Pg} \urtwocell{dr}{P\sigma} & \\
Px \ar[r]_{Pf} \rtwocell[0.2]{ur}{\alpha} & Py \dotrtwocell[0.7]{ul}{\alpha'}
}
\end{equation*}
there exists a unique $\hat{\delta} \colon \hat{h} \To \hat{h'}$ such that $\hat{\alpha'}.f\hat{\delta} = \sigma.\hat{\alpha}$ and $\delta.\hat{\beta} = \hat{\beta'}.P\hat{\delta}$.
\end{enumerate}
\end{Defn}

Informally this is best stated by saying that $f$ lifts 1-cells up to isomorphism and lifts 2-cells coherently with the lifted isos. The uniqueness of lifted 2-cells implies that lifted 1-cells are unique up to a coherent isomorphism.

\begin{Prop} 
A 1-cell $f\colon x \to y$ is cartesian if and only if
\begin{equation*}
\cd[]{
\E(z,x) \ar[d]_{P_{zx}} \ar[r]^{f_*} & \E(z,y) \ar[d]^{P_{zy}} \twocong{dl} \\
\B(Pz,Px) \ar[r]_{Pf_*} & \B(Pz,Py)
}
\end{equation*}
is a bipullback for all $z$. This isomorphism is the coherence for $P$ on composition.
\end{Prop}

\begin{Rk} 
We use bipullback in the sense of Street and Joyal \cite{joyal1993}: a weakly-universal iso-square over $Pf_*$ and $P_{zy}$. That is, there is a pseudo-natural equivalence
$$ \Hom(A,\E(z,x)) \simeq \Hom(A,\B(Pz,Px))\times_{\cong}\Hom(A,\B(Pz,Py))$$
where the right expression is the pseudo-pullback (iso-comma-category).
\end{Rk}

\begin{Defn}  
A 2-cell $\alpha\colon f \To g\colon x \to y$ in $\E$ is \emph{cartesian} if it is cartesian as a 1-cell for the functor $P_{xy}\colon \E(x,y) \to \B(Px,Py)$.
\end{Defn}

As in Section \ref{sec:strict} we say that $P$ is \emph{locally fibred} when $P_{xy}\colon \E(x,y) \to \B(Px,Py)$ is a fibration for all $x$, $y$ in $\E$.

\begin{Defn}
Let $P \colon \E \to \B$ be a homomorphism. We say that $P$ is a \emph{fibration} when
\begin{enumerate}
\item for any $e \in \E$ and $f \colon b \to Pe$, there is a cartesian 1-cell $h \colon a \to e$ with $Ph = f$;
\item $P$ is locally fibred; and
\item the horizontal composite of any two cartesian 2-cells is cartesian.
\end{enumerate}
We say that $\E$ is a \emph{fibred bicategory} when it is the domain of a fibration.
\end{Defn}

\begin{Rk} 
We could insist that cartesian 1-cells only have $Ph \cong f$ and the definition above would not be any weaker. When $P$ is a fibration it is locally fibred and thus locally has the iso-lifting property. Now cartesian 1-cells isomorphic to a cartesian 1-cell are cartesian (see Proposition \ref{weak_1-cells_iso_cartesian} below). Thus if there is a cartesian lift $h$ with $Ph \cong f$ then there is a cartesian lift $h'$ with $Ph' = f$. The converse is trivial so the two definitions are equivalent.
\end{Rk}

\begin{Rk} 
Corollary 1 in \cite{joyal1993} states that if one leg of a cospan has the iso-lifting property then the pullback of that cospan is a bipullback. When $P$ is a fibration it is locally fibred and thus locally has the iso-lifting property. It follows that if a 1-cell is 2-categorically cartesian (Definition \ref{def:cartesian_1-cell_strict}) then it is bicategorically cartesian ( Definition \ref{def:cartesian1_weak}). As a result fibred 2-categories are also fibred bicategories.
\end{Rk}

We take the time here to establish a few basic properties of cartesian maps.

\begin{Prop} 
\label{weak_1-cells_iso_cartesian}
If $f \cong g$ then $f$ is cartesian if and only if $g$ is cartesian.
\end{Prop}
\begin{proof}
Suppose that $f$ is cartesian and $\alpha \colon f \To g$ is an isomorphism. In the diagram below: the inner isomorphism is the coherence of $P$ on composition with $f$. The isomorphisms above and below are induced by $\alpha$ and $P\alpha$ and thus the pasting is equal to the coherence of $P$ on composition with $g$. 
\begin{equation*}
\cd[@R+0.3cm]{
\E(z,x) \ar[d]_{P_{zx}} \pair[7pt]{r}{g_*}{f_*} \twocong{r} & \E(z,y) \ar[d]^{P_{zx}} \\
\B(Pz,Px) \pair[7pt]{r}{Pf_*}{Pg_*} \twocong{r} & \B(Pz,Py) \twocong[0.4]{ul}
}
\end{equation*}
By definition of cartesian 1-cell the inner isomorphism is a bipullback. Bipullbacks are closed under the pasting of isomorphisms as indicated. Thus whole diagram is a bipullback and $g$ is cartesian.
\end{proof}

\begin{Prop} 
\label{prop:composition2}
Suppose $f \colon w \to x$, $g \colon x \to y$ in $\E$. If $f$ and $g$ are cartesian then $gf$ is cartesian.
\end{Prop}
\begin{proof}
Suppose that $f$ and $g$ are cartesian. In the diagram below: the inner two isomorphisms are coherence of $P$ on composition with $f$ and $g$. The outer isomorphisms are induced by associativity of composition. Thus the pasting is equal to the coherence of $P$ on composition with $gf$.
\begin{equation*}
\cd[]{
\E(z ,w ) \ar[d]_{P_{zw}} \ar[r]^{f_*} \ar@/^0.8cm/[rr]^{(gf)_*} & \E(z,x) \ar[d]^{P_{zx}} \ar[r]^{g_*} & \E(z,y) \ar[d]^{P_{zy}} \\ 
\B(Pz,Pw) \ar[r]_{Pf_*} \ar@/_0.8cm/[rr]_{(Pgf)_*} & \B(Pz,Px) \ar[r]_{Pg_*} \twocong{ul} \ar@{}[ul]|(0.4){\phi_{\mhyphen f}}& \B(Pz,Py) \twocong{ul} \ar@{}[ul]|(0.4){\phi_{\mhyphen g}}
}
\end{equation*}
The inner two isos are bipullbacks by definition of cartesian 1-cell. 
Bipullbacks are closed under the pasting of isomorphisms as indicated. Thus whole diagram is a bipullback and $gf$ is cartesian.
\end{proof}

\begin{Prop} 
\label{prop:left-2-out-of-3_weak}
Suppose $f \colon w \to x$, $g \colon x \to y$ in $\E$. If $g$ and $gf$ are cartesian then $f$ is cartesian.
\end{Prop}
\begin{proof}
This proof is essentially the same as Proposition \ref{prop:composition2}. The only other thing we need to know is that bipullbacks have the same cancellation property as pullbacks.
\end{proof}

\begin{Prop} 
If $f \colon x \to y$ is an equivalence then it is cartesian.
\end{Prop}
\begin{proof}
Suppose that $f$ is part of an adjoint equivalence $(f,f^\centerdot,\eta,\epsilon)$ and
\begin{equation*}
\cd[@C+4pt@R+4pt]{
z \ar[dr]^{g} & \\
x \ar[r]_{f}  & y
}
\hspace{1cm}\text{over}\hspace{1cm}
\cd[@C+4pt@R+4pt]{
Pz \ar[dr]^{Pg} \ar[d]_{h}& \\
Px \ar[r]_{Pf} \urtwocell[0.3]{ur}{\alpha} & Py
}.
\end{equation*}
Let $\hat{h} = f^\centerdot g$ and $\hat{\alpha}$ be the composite $f.(f^\centerdot .g) \cong (f.f^\centerdot ).g \cong 1.g \cong g$. Let $\hat{\beta} \colon P\hat{h} \To h$ be the composite $P(f^\centerdot .g) \cong Pf^\centerdot .Pg \cong Pf^\centerdot .(Pf.h) \cong (Pf^\centerdot .Pf).h \cong P(f^\centerdot .f).h \cong P1.h \cong 1.h \cong h$. This is a lift of $(h,\alpha)$.

To check the 2-cell property suppose that $\gamma \colon g \To g'$ and
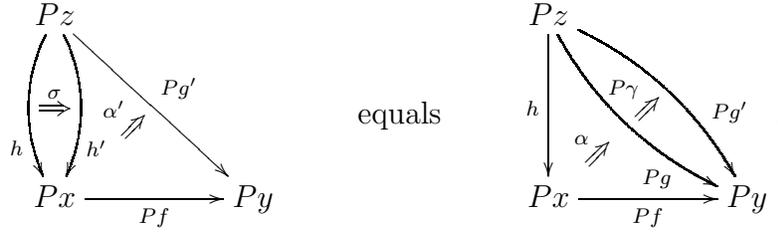
\begin{equation*}
\cd[@C+1cm@R+1cm]{
Pz \ar[dr]^{Pg'} \ar@/_10pt/[d]_(0.7){h} \ar@/^10pt/[d]^(0.7){h'} \rtwocell{d}{\sigma} & \\
Px \ar[r]_{Pf} \urtwocell[0.4]{ur}{\alpha'} & Py
}
\hspace{1cm}\text{equals}\hspace{1cm}
\cd[@C+1cm@R+1cm]{
Pz \ar@/_10pt/[dr]_(0.7){Pg} \ar@/^10pt/[dr]^(0.7){Pg'} \urtwocell{dr}{P\gamma}  \ar[d]_{h}& \\
Px \ar[r]_{Pf} \urtwocell[0.25]{ur}{\alpha} & Py
}.
\end{equation*}
Then suppose that there are lifts $(\hat{h},\hat{\alpha}, \beta)$, $(\hat{h'},\hat{\alpha'}, \beta')$ of $(h,\alpha)$ and $(h',\beta)$ as above. Let $\hat{\sigma}$ be the composite
\begin{equation*}
\cd[@R-0.2cm] {
h  \ar@{=>}[r]^{l} & 1.h  \ar@{=>}[r]^{\eta.h} & (f^\centerdot.f).h  \ar@{=>}[r]^{a} & f^\centerdot.(f.h)  \ar@{=>}[r]^{f^\centerdot.\hat{\alpha}} & f^\centerdot.g \ar@{=>}[d]^{f^\centerdot.\gamma} \\
h' \ar@{<=}[r]_{l} & 1.h' \ar@{<=}[r]_{\eta^{-1}.h'} & (f^\centerdot.f).h' \ar@{<=}[r]_{a} & f^\centerdot.(f.h') \ar@{<=}[r]_{f^\centerdot.\hat{\alpha'}} & f^\centerdot.g'
}
\hspace{0.2cm}.
\end{equation*}
It is unique with the property that  $\hat{\alpha'}.f\hat{\sigma} = \gamma.\hat{\alpha}$ and $\sigma.\hat{\beta} = \hat{\beta'}.P\hat{\sigma}$.
\end{proof}

\begin{Prop} 
If $f \colon x \to y$ is cartesian and $Pf$ is an equivalence then $f$ is an equivalence.
\end{Prop}
\begin{proof}
Suppose $Pf$ is part of an adjoint equivalence $(Pf,Pf^\centerdot,\eta,\epsilon)$. Then we can lift $\epsilon$ to obtain $\hat{\epsilon}\colon f.h \cong 1$. 
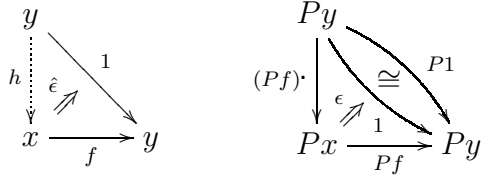
\begin{equation*}
\cd[@C+8pt@R+8pt]{
y \ar[dr]^{1} \ar@{.>}[d]_{h}& \\
x \ar[r]_{f} \urtwocell[0.3]{ur}{\hat{\epsilon}} & y
}
\hspace{1cm}
\cd[@C+8pt@R+8pt]{
Py \ar@/_8pt/[dr]_(0.6){1} \ar@/^8pt/[dr]^(0.6){P1}  \twocong{dr} \ar[d]_{(Pf)^\centerdot}& \\
Px \ar[r]_{Pf} \urtwocell[0.25]{ur}{\epsilon} & Py
}
\end{equation*}
Now since $1_y$ is an equivalence it is cartesian. Then since $f.h \cong 1$, $fh$ is cartesian. Then Proposition \ref{prop:left-2-out-of-3_weak} tells us that $h$ is cartesian. We can then lift $Ph.Pf \cong (Pf)^\centerdot .Pf \cong 1 \cong P1$ as picture on the right to obtain $\hat{\eta}\colon h.k \cong 1$.
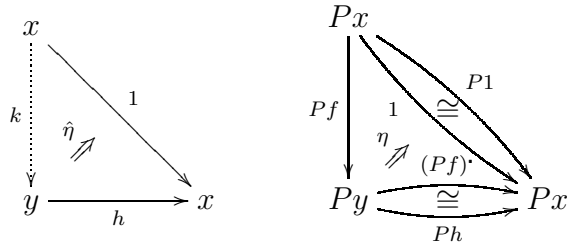
\begin{equation*}
\cd[@C+1cm@R+1cm]{
x \ar[dr]^{1} \ar@{.>}[d]_{k}& \\
y \ar[r]_{h}  \urtwocell[0.3]{ur}{\hat{\eta}} & x
}
\hspace{1cm}
\cd[@C+1cm@R+1cm]{
Px \ar@/_6pt/[dr]_(0.35){1} \ar@/^6pt/[dr]^{P1} \twocong{dr} \ar[d]_{Pf}& \\
Py \pair[6pt]{r}{(Pf)^\centerdot}{Ph} \twocong{r} \urtwocell[0.25]{ur}{\eta} & Px
}
\end{equation*}
Finally, $f \cong f.1 \cong f.(h.k) \cong (f.h.).k \cong 1.k \cong k$ and then $h.f \cong h.k \cong 1$. Thus $f$ is an equivalence.
\end{proof}

\begin{Prop} 
\label{prop:isos_lift_to_isos_weak}
Suppose that $f \colon a \to b$ is cartesian and $\hat{\sigma}$ is the unique lift of $\sigma$ as pictured. If $\sigma$ and $\tau$ are isomorphisms then $\hat{\sigma}$ is an isomorphism.
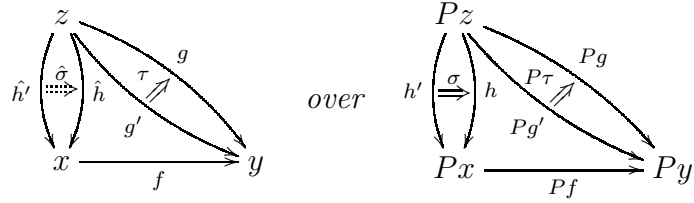
\begin{equation*}
\cd[@C+36pt@R+18pt]{
z \pair[8pt]{d}{\hat{h}}{\hat{h'}} \dotrtwocell{d}{\hat{\sigma}} \pair[8pt]{dr}{g}{g'} \urtwocell{dr}{\tau} & \\
x \ar[r]_{f} & y
}
\hspace{12pt}\text{over}\hspace{12pt}
\cd[@C+36pt@R+18pt]{
Pz \pair[8pt]{d}{{h}}{{h'}} \rtwocell{d}{\sigma} \pair[8pt]{dr}{Pg}{Pg'} \urtwocell{dr}{P\tau} & \\
Px \ar[r]_{Pf} & Py
}
\end{equation*}
The isomorphisms on the front and back are omitted in the diagram.
\end{Prop}
\begin{proof}
If $\sigma$ and $\tau$ are both invertible then the above lifting can be done with $\sigma^{-1}$ and $\tau^{-1}$. This gives a map $\hat{\sigma}^\centerdot\colon \hat{h} \To \hat{h'}$. If we paste these diagrams together then $\hat{\sigma} \hat{\sigma}^\centerdot$ is a lift of $\sigma \sigma^{-1} = 1_h$. However $1_{\hat{h}}$ is also a lift of $1_h$ and thus by uniqueness $\hat{\sigma} \hat{\sigma}^\centerdot = 1_{\hat{h}}$. Pasting the diagrams together the other way gives $\hat{\sigma}^\centerdot \hat{\sigma} = 1_{\hat{h'}}$.
\end{proof}

\begin{Cor}\label{cor:unique_up_to_iso}
The lifts of Definition \ref{def:cartesian1_weak} (1) are unique up to a unique invertible 2-cell.
\end{Cor}
\begin{proof}
Use the above result with $\sigma = 1_h$ and $\tau = 1_g$.
\end{proof}

\begin{Prop} 
Suppose $f \colon a \to b$ is cartesian over $Pf \colon Pa \to Pb$. It is unique up to an equivalence 1-cell and isomorphism 2-cell. This equivalence and isomorphism are unique up to an isomorphism 2-cell. \end{Prop}
\begin{proof}
Suppose that $g\colon c \to b$ is cartesian over $Pf$. Then
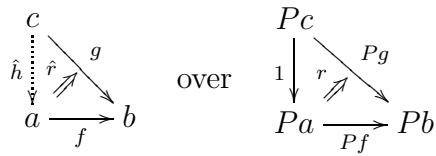
\begin{equation*}
\cd[]{
c \ar[dr]^{g} \ar@{.>}[d]_{\hat{h}} & \\
a \ar[r]_{f} \urtwocell[0.35]{ur}{\hat{r}} & b
}
\hspace{12pt}\text{over}\hspace{12pt}
\cd[]{
Pc \ar[dr]^{Pg} \ar[d]_{1} & \\
Pa \ar[r]_{Pf} \urtwocell[0.35]{ur}{r} & Pb
}
\end{equation*}
and $f$ is cartesian so there exists a lift $(\hat{h},\hat{r},\hat{\beta})$. By Corollary \ref{cor:unique_up_to_iso} this lift is unique up to a unique isomorphism. We have yet to show that $\hat{h}$ is an equivalence.

If we draw the same diagram with $g$ in the base then since $g$ is cartesian there exists a lift $(\hat{h}',\hat{l}',\hat{\beta}')$ of $r\colon Pg.1 \To Pf$. 
These two lifts can be pasted together to form a lift pictured on the left. 
The base forms a commuting shell by coherence in a bicategory. 
Then there exists a unique lift of $r \colon 1.1 \To 1$. It is a 2-cell $\hat{h} \hat{h}' \To 1$ and it is an isomorphism by Proposition \ref{prop:isos_lift_to_isos_weak}. 
\begin{equation*}
\cd[@C+30pt@R+15pt]{
x \ar[ddr]_{f} \ar@/_10pt/[dd]_{1} \ar[d]^{\hat{h}'} \ar@/^12pt/[ddr]^{f}_{=} & \\
z \ar[dr]_(0.65){g} \ar[d]^{\hat{h}} \urtwocell[0.25]{ur}{\hat{r}'} &  \\
x \ar[r]_{f} \urtwocell[0.3]{ur}{\hat{r}} & y
}
\hspace{12pt}\text{over}\hspace{12pt}
\cd[@C+30pt@R+15pt]{
Px \ar[ddr]_{Pf} \ar@/_10pt/[dd]_{1} \ar[d]^{1} \ar@/^12pt/[ddr]^{Pf}_{=} & \\
Pz \ar[dr]_(0.65){Pg} \ar[d]^{1} \urtwocell[0.25]{ur}{r} &  \\
Px \ar[r]_{Pf} \urtwocell[0.3]{ur}{r} & Py
}
\end{equation*}
The isomorphisms omitted in this diagram are all $r$. If we paste the lifts together the other way and follow the same reasoning we get another ismorphism $\hat{h}' \hat{h} \To 1$. Thus $\hat{h}$ is an equivalence.
\end{proof}

\begin{Prop} \label{chosen_2-cells_enough_weak}
Suppose that $P$, $Q$, $F$ are homomorphisms with with $P=QF$.
\begin{enumerate}
\item If $P$ is locally fibred and chosen cartesian 2-cells are closed under horizontal composition then all cartesian 2-cells are closed under horizontal composition.
\item If $P$ and $Q$ are fibrations and $F$ preserves chosen cartesian maps then $F$ preserves all cartesian maps.
\end{enumerate}
\end{Prop}
\begin{proof}
The proofs are essentially the same as for Proposition \ref{prop:chosen_2-cells_enough_strict} and \ref{prop:preserve_chosen_maps_enough_strict}.
\end{proof}

\subsection{Fibrations with stricter properties}

Every fibration is locally fibred and thus locally has the iso-lifting property. We can take advantage of this to make fibrations much easier to handle.

\begin{Prop} 
\label{prop:nicer_lifting}
When $P$ is locally fibred every lift $(\hat{h},\hat{\alpha},\hat{\beta})$ of $(h, \alpha)$ along a cartesian 1-cell can be chosen so that $\hat{\beta} = 1_{h}$. That is, lifts along cartesian 1-cells can be chosen so $P\hat{h}=h$.
\end{Prop}
\begin{proof}
Suppose that $f$ is cartesian and $(\hat{h},\hat{\alpha},\hat{\beta})$ is a lift of $(h, \alpha)$. That is,
\begin{equation*}
\cd[@C+15pt@R+15pt]{
z \ar[dr]^{g} \ar[d]_{\hat{h}} & \\
x \ar[r]_{f} \urtwocell[0.3]{ur}{\hat{\alpha}} & y
}
\hspace{1cm}
\cd[@C+15pt@R+15pt]{
Pz \ar[d]_{h}   \ar[dr]^{Pg}& \\
Px \ar[r]_{Pf} \urtwocell[0.3]{ur}{\alpha} & Py
}
\end{equation*}
where $\alpha.Pf\hat{\beta} = P\hat{\alpha}.\phi_{hf}$. Let $\sigma\colon h^!\To h$ be the cartesian lift of $\hat{\beta}$ at $h$ and let $\alpha^!=\hat{\alpha}.f\sigma$ ($\sigma$ is an isomorphism because it is a cartesian 2-cell over an isomorphism). Then $(h^!,\alpha^!, 1_{h})$ is a lift of $(h, \alpha)$ for $f$. This is proved by $P\alpha^!.\phi_{h^!f} = P\hat{\alpha}.P(f\sigma).\phi_{h^!f} = P\hat{\alpha}.\phi_{hf}.PfP\sigma = \alpha$.
\end{proof}

\begin{Prop} 
Every locally fibred homomorphism $P\colon\E\to\B$ is isomorphic to a locally fibred $P'\colon\E'\to\B$ via a homomorphism $S$ that preserves identities and composition. If $P$ is a fibration then $P'$ is a fibration and $S$ is cartesian.
\begin{equation*}
\cd[@C-10pt]{
\E \ar[dr]_{P} \ar[rr]^{S} & & \E' \ar[dl]^{P'} \\
 & \B &
}
\end{equation*}
\end{Prop}
\begin{proof}
Suppose $P\colon \E \to \B$ is locally fibred.
Let $\E'$ have the same 0-, 1-, and 2-cells as $\E$ with the same vertical composition and 2-cell identities. Then let $P'$ have the same action as $P$ on 0-, 1-, and 2-cells so that $S$ is the identity on 0-, 1-, and 2-cells. We now define horizontal composition in $\E'$. If $f\colon e \to e'$ and $g\colon e' \to e''$ in $\E'$ then $g\circ f$ is the domain of the cartesian lift of $\phi_{gf}\colon PgPf \To P(gf)$ at $gf$. If $\alpha\colon f \To f'$, $\beta\colon g \To g'$ then $\beta\circ\alpha$ is the unique map above $P\beta*P\alpha$ such that
\begin{equation*}
\cd[]{
g \circ f  \ar@{=>}[r]^{\c{\phi}{g f }} \ar@{:>}[d]_{\beta\circ\alpha} & g f  \ar@{=>}[d]^{\beta*\alpha} \\
g'\circ f' \ar@{=>}[r]_{\c{\phi}{g'f'}} & g'f'
}
\hspace{0.5cm}\text{over}\hspace{0.5cm}
\cd[]{
Pg  Pf  \ar@{=>}[r]^{\phi} \ar@{=>}[d]_{P\beta*P\alpha} & P(g f)  \ar@{=>}[d]^{P(\beta*\alpha)} \\
Pg' Pf' \ar@{=>}[r]_{\phi} & P(g'f')}
\end{equation*}
This has the effect that $P'(g\circ f) = P'g .P'f$ and $P'(\beta\circ\alpha) = P\beta*P\alpha = P'\beta*P'\alpha$. Similarly, let the identity 1-cells $\hat{1}_{e}$ be the domain of the cartesian lift of $\phi_{e}\colon 1_{Pe} \To P(1_e)$ at $1_e$. Then $P'(\hat{1}_e) = 1_{P'e}$. The coherence isomorphisms for $\E'$ are obtained as the unique maps in the following diagrams.
\begin{equation*}
\cd[]{
 (h \circ g) \circ f \ar@{=>}[r]^{\varphi} \ar@{:>}[d]_{\hat{a}}& (h\circ g). f  \ar@{=>}[r]^{\varphi.f} & (h.g).f  \ar@{=>}[d]^{a} \\
 h\circ(g\circ f) \ar@{=>}[r]_{\varphi} & h.(g\circ f) \ar@{=>}[r]_{h.\varphi} & h.(g.f)
}
\end{equation*}
\begin{equation*}
\cd[]{
 f \ar@{=>}[drr]^{r} \ar@{:>}[d]_{\hat{r}} & & \\
 f \circ \hat{1} \ar@{=>}[r]_{\varphi} & f.\hat{1} \ar@{=>}[r]_{f.\varphi}& f.1
}
\hspace{1cm}
\cd[]{
 f \ar@{=>}[drr]^{l} \ar@{:>}[d]_{\hat{l}} & & \\
 \hat{1} \circ f \ar@{=>}[r]_{\varphi} & \hat{1}.f \ar@{=>}[r]_{\varphi.f} & 1.f
}
\end{equation*}
The reader can verify that this is a bicategory and that $P'$ is a homomorphism that preserves composition and identities.
Let $S$ be the identity on 0-, 1-, and 2-cells. The coherence morphisms for $S$ are the isomorphisms $\c{\phi}{gf}\colon g \circ f \cong gf$ and $\c{\phi_e}{1_{e}} \colon \hat{1}_e \cong 1_e$ described above. It is easy to check that $P'S = P$ and since $S$ is the identity on 0-, 1-, and 2-cells it is an isomorphism. 

Suppose that $P$ is a fibration. The cartesian 1-cells in $\E'$ are precisely those we obtain in $\E$. Lifts along cartesian 1-cells are $(\hat{h},\alpha^!,\hat{\beta})$ obtained by taking lifts $(\hat{h},\hat{\alpha},\hat{\beta})$ in $\E$ and letting $\alpha^! = \hat{\alpha}\c{\phi}{f\hat h}$. Thus $P'$ has cartesian lifts of 1-cells. Since the action of $P'$ on hom-categories is the same as $P$, $P'$ is locally fibred. Suppose that $\beta$, $\alpha$ are cartesian 2-cells. Their composite in $\E'$ is defined using the diagram
\begin{equation*}
\cd[]{
 g \circ f  \ar@{=>}[r]^{\varphi} \ar@{=>}[d]_{\beta\circ\alpha}& g .f  \ar@{=>}[d]^{\beta*\alpha}  \\
 g'\circ f' \ar@{=>}[r]_{\varphi} & g'.f'
}.
\end{equation*}
Since $P$ is a fibration $\beta*\alpha$ is cartesian. Then the cancellation property of cartesian 1-cells tells us that $\beta\circ\alpha$ is also cartesian. Thus $P'$ is a fibration. Chosen cartesian 1- and 2-cells in $\E'$ are the same as in $\E$. Since $S$ is the identity on 0-, 1-, and 2-cells it is cartesian.
\end{proof}

\begin{Rk}
These last two results rely on the local iso-lifting property of $P$.
The first result corresponds to the lifting of an isomorphism that arises from the bipullback definition of cartesian 1-cell.
\begin{equation*}
\cd[@C-8pt]{
\mathbbm{2} \ar@/_10pt/[ddr] \ar@/^10pt/[drr] \ar[dr] & & \\
&\twocong{u} \twocong{l} \E(z,x) \ar[d]^{P_{zx}} \ar[r]^{f_*} & \E(z,y) \ar[d]^{P_{zy}} \twocong{dl} \\
& \B(Pz,Px) \ar[r]_{Pf_*} & \B(Pz,Py)
}
\Mapsto{}
\cd[@C-8pt]{
\mathbbm{2} \ar@/_10pt/[ddr] \ar@/^10pt/[drr] \pair[8pt]{dr}{}{} \twocong{dr} \twocong{drr} & & \\
& \E(z,x) \ar[d]^{P_{zx}} \ar[r]^{f_*} & \E(z,y) \ar[d]^{P_{zy}} \twocong{dl} \\
& \B(Pz,Px) \ar[r]_{Pf_*} & \B(Pz,Py)
}
\end{equation*}
The isomorphism on the left is the $\hat{\beta}$ associated with lifts along cartesian $f$.
The second result corresponds to the lifting of the isomorphisms associated with the action of $P$ on composition and identities.
\begin{equation*}
\cd[]{
\E(y ,z )\times \E(x ,y ) \ar[d]_{P_{yz}\times P_{xy}} \ar[r]^-{*} \twocong{dr} & \E(x,z) \ar[d]^{P_{xz}}  \\
\B(Py,Pz)\times \B(Px,Py) \ar[r]_-{*} & \B(Px,Pz)
}
\Mapsto{}
\end{equation*}
\begin{equation*}
\cd[]{
\E(y ,z )\times \E(x ,y ) \ar[d]_{P_{yz}\times P_{xy}} \pair{r}{*}{\circ}  \twocong{r} & \E(x,z) \ar[d]^{P_{xz}}  \\
\B(Py,Pz)\times \B(Px,Py) \ar[r]_-{*} & \B(Px,Pz)
}
\end{equation*}
\begin{equation*}
\cd[]{
\mathbbm{1} \ar@/_12pt/[dr]_{1_{Px}} \ar[r]^{1_x}  & \E(x,x) \ar[d]^{P_{xx}}  \\
\twocong[0.6]{ur}  & \B(Px,Px)
}
\Mapsto{}
\cd[]{
\mathbbm{1} \ar@/_12pt/[dr]_{1_{Px}} \pair{r}{1_x}{\hat{1}_x} \twocong{r} & \E(x,x) \ar[d]^{P_{xx}}  \\
  & \B(Px,Px)
}
\end{equation*}
We've borrowed this convenient visual representation of iso-lifting from \cite{joyal1993}.
\end{Rk}

From this point on we will suppose that all fibrations preserve composition and identities and have the stronger lifting property of Proposition \ref{prop:nicer_lifting}. Working under this supposition will vastly simplify future calculations. We can do this without loss of generality because all fibrations are isomorphic to a fibration of this kind.

\begin{Defn}
A homomorphism $\eta \colon \E \to \D$ between fibred categories is called \emph{cartesian} when it preserves cartesian maps and $Q\eta = P$.
\end{Defn}

Let $\Fib(\B)$ be the tricategory whose objects are fibrations over $\B$, whose 1-cells are cartesian homomorphisms,
2-cells are pseudo-natural transformations $\Gamma \colon \eta \To \epsilon$ that have $Q\Gamma = 1_P$; and
3-cells are modifications $\zeta \colon \Gamma \Rrightarrow \Lambda$ that have $Q\zeta = 1_{1_P}$.
Let $\bicat$ be the tricategory of bicategories, homomorphisms, transformations and modifications. Let $[\B^\coop,\bicat]$ be the tricategory of contravariant trihomomorphisms from $\B$ to $\bicat$, tritransformations, trimodifications and perturbations.

\subsection{The Grothendieck construction}

Before describing the Grothendieck construction for bicategories we will unpack the tricategory structure of $\bicat$ and see what a contravariant trihomomorphism from $\B$ to $\bicat$ really means.

\begin{Rk}
An algebraic definition of tricategory can be found in \cite[p.\ 23]{gurski2006}. There is more than one tricategory structure on $\bicat$. We choose the following:
\begin{itemize}
\item Composition of $1$-cells is the usual composition of homomorphisms: $GF(x) = G(F(x))$ and the usual coherence isomorphisms.

\item Composition of $2$-cells: When $\alpha \colon F \To G$ and $\beta \colon G \To H$ we let $(\beta\alpha)_x = \beta_x\alpha_x$ and $(\beta\alpha)_f$ be the associated pasting of 2-cells.
When 
\begin{equation*}
\cd[]{
\B \pair{r}{F}{F'} \dtwocell{r}{\alpha} & \C \pair{r}{G}{G'} \dtwocell{r}{\beta} & \D
}
\end{equation*}
we let $(\beta*\alpha)_x$ be
\begin{equation*}
\cd[]{
GF(x) \ar[r]^{\beta_{F(x)}} & G'F(x) \ar[r]^{G'(\alpha_x)} & G'F'(x)
}.
\end{equation*}
$(\beta*\alpha)_f$ be the associated pasting of 2-cells.

\item Composition of 3-cells is similar to that for 2-cells.

\item Identity 1-cells are the obvious identity homomorphisms $1_\B \colon \B \to \B$. Identity 2-cells are transformations $1_F \colon F \To F$ with 1-cell components $(1_F)_x = 1_{F(x)}$. Identity 3-cells are modifications $1_\alpha$ with $(1_\alpha)_x = 1_{(\alpha_x)}$.
\end{itemize}
The rest of the data consists firstly of pseudo-natural equivalences governing associativity and the action of identities. The final  data are invertible modifications that sit in place of the usual axioms. The details can be found in \cite{gordon1995}.

\item Associativity of composition is governed by a pseudo-natural equivalence
\begin{equation*}
\cd[]{
(HG)F \ar@{=>}[d]_{(\gamma\beta)\alpha} \ar@{=>}[r]^{a_{HGF}} & H(GF) \ar@{=>}[d]^{\gamma(\beta\alpha)}\\
(H'G')F' \ar@{=>}[r]_{a_{H'G'F'}}  \twocong{ur} & H'(G'F') 
}.
\end{equation*}
The 2-cells components are identity transformations $1 \colon H(GF) \To (HG)F$. The 3-cell components are invertible modifications whose components are a composite of coherence isomorphisms in the image of $H'(G'F')$.

\item The unity of identities is governed by two pseudo-natural transformations
\begin{equation*}
\cd[]{
1.F  \ar@{=>}[d]_{1_1.\alpha} \ar@{=>}[r]^{l_F} & F \ar@{=>}[d]^{\alpha}\\
1.F' \ar@{=>}[r]_{l_F'}  \twocong{ur} & F' 
}
\hspace{1cm}
\cd[]{
F \ar@{=>}[d]_{\alpha} \ar@{=>}[r]^{r_F} & F.1 \ar@{=>}[d]^{\alpha.1_1}\\
F' \ar@{=>}[r]_{r_F'}  \twocong{ur} & F'.1
}.
\end{equation*}
The 2-cells components are identity transformations $1 \colon 1F \To F$ and $1 \colon F \To F1$. The 3-cell components are invertible modifications whose components are a composite of coherence isomorphisms in the image of $F'$.

\item There are invertible modifications $\pi$, $\mu$, $\lambda$, $\rho$ that relate various composites of $a$, $l$, $r$ above. In this case they are built from the coherence isomorphisms of assorted bicategories and homomorphisms.

\item There are three axioms that $\pi$, $\mu$, $\lambda$, $\rho$ are required to satisfy. The coherence theorem for bicategories guarantees that they hold.

There is a related tricategory structure given by an alternate composition rule for 2,3-cells. Choosing one or the other will not significantly effect the nature of our results.
\end{Rk}

\begin{Rk}
An algebraic definition of trihomomorphism can be found in \cite[p.\ 29]{gurski2006}.  A trihomomorphism $F\colon \B^\coop \to \bicat$ consists of the following data:
\begin{itemize}
\item An object function $F_0 \colon \ob B \to \ob \bicat$.
\item For objects $a$, $b$ of $\B$, a pseudo-functor $\B(a,b) \to \bicat(Fa,Fb)$. This means that 2-cell composition and identities in $\B$ are preserved up to natural isomorphisms satisfying standard coherence axioms. The data is just isomorphisms $\phi_{\beta\alpha}\colon F(\beta.\alpha) \Rrightarrow F\alpha.F\beta$ and $\phi_f \colon F(1_f) \Rrightarrow 1_{Ff}$.
The components are
\begin{equation*}
\cd[@R-1.0cm]{
Ff(x) \ar@/_6pt/[dr]_{F\alpha_x} \dtwocell{rr}{\phi_{\beta\alpha}} \ar@/^12pt/[rr]^{F(\beta.\alpha)_x} & & Fh(x) \\
 &  Fg(x)  \ar@/_6pt/[ur]_{F\beta_x} & 
}
\hspace{0.5cm}\text{and}\hspace{0.5cm}
\cd[]{
Ff(x) \pair[12pt]{rr}{F(1_f)_x}{1_{Ff(x)}} \dtwocell{rr}{\phi_f}  & & Ff(x)
}.
\end{equation*}

\item For objects $a$, $b$, $c$ of $\B$, an adjoint equivalence $\chi\colon \otimes'\circ (F \times F) \To F \circ \otimes$. This means that horizontal composition is preserved up to adjoint equivalence. The data is 
\begin{equation*}
\cd[]{
Ff.Fg \ar@{=>}[r]^{\chi_{gf}} \ar@{=>}[d]_{F\alpha.F\beta} & F(gf) \ar@{=>}[d]^{F(\beta\alpha)} \\
Ff'.Fg' \dlthreecell{ur}{\chi_{\beta\alpha}} \ar@{=>}[r]_{\chi_{g'f'}} & F(g'f')
}
\end{equation*}
which amounts to
\begin{equation*}
\cd[]{
FfFg(x) \ar[r]^{{\chi_{gf}}_x} \ar[d]_{FfFg(h)} & Fgf(x) \ar[d]^{Fgf(h)} \\
FfFg(y) \dltwocell{ur}{{\chi_{gf}}_h} \ar[r]_{{\chi_{gf}}_y} & Fgf(y)
}
\hspace{0.5cm}\text{and}\hspace{0.5cm}
\cd[@R-4pt]{
FfFg(x) \ar[r]^-{{\chi_{gf}}_x} \ar[d]_{F\alpha_{Fg(x)}} & Fgf(x) \ar[dd]^{F(\beta\alpha)_x} \\
Ff'Fg(x) \ar[d]_{Ff'(F\beta_x)} & \\
Ff'Fg'(x) \dltwocell{uur}{{\chi_{\beta\alpha}}_x} \ar[r]_-{{\chi_{g'f'}}_x} & Fgf(x)
}.
\end{equation*}
The data on the left indicates that $\chi_{gf}$ is a pseudo-natural equivalence and the data on the right is the component at $x$ of the modification $\chi_{\beta\alpha}$.

\item For objects $a$ of $\B$, an adjoint equivalence transformation $\iota\colon I'_{Fa} \To F \circ I_a$. This means that identity 1-cells are preserved up to adjoint equivalence. The data is 
\begin{equation*}
\cd[]{
1_{Fa} \ar@{=>}[r]^{\iota_{a}} \ar@{=>}[d]_{1_{1_{Fa}}} & F1_a \ar@{=>}[d]^{F(1_{1_a})} \\
1_{Fa} \dlthreecell{ur}{\iota_{1_a}} \ar@{=>}[r]_{\iota_{a}} & F1_a
}. 
\end{equation*}
This means
\begin{equation*}
\cd[]{
x \ar[r]^{{\iota_{a}}_x} \ar[d]_{h} & F1(x) \ar[d]^{F1(h)} \\
y \dltwocell{ur}{{\iota_{a}}_h} \ar[r]_{{\iota_{a}}_y} & F1(y)
}
\hspace{0.5cm}\text{and}\hspace{0.5cm}
\cd[@R-4pt]{
x \ar[r]^-{{\iota_{a}}_x} \ar[d]_{1_x} & F1(x) \ar[d]^{{F(1_1)}_x} \\
x \dltwocell{ur}{{\iota_{1_a}}_x} \ar[r]_-{{\iota_{a}}_x} & F1(x)
}.
\end{equation*}
The data on the left indicates that $\iota_{a}$ is a pseudo-natural equivalence and the data on the right is the component at $x$ of the modification $\iota_{1_a}$.

\item For objects $a$, $b$, $c$, $d$ of $\B$, an invertible modification $\omega$ whose component at $fgh$ is itself an invertible modification whose component at $x$ is an invertible pseudo-natural transformation
\begin{equation*}
\cd[@R-16pt@C-8pt]{
   & FfF(h.g)(x) \ar[r]^{\chi_x} & F((h.g).f)(x) \ar[dr]^{Fa} &  \\
FfFgFh(x) \ar[ur]^{Ff(\chi_x)} \dtwocell{rrr}{\omega_{fgh,x}} \ar[dr]_{1} &   &   & F(h.(g.f))(x) \\
   & FfFgFh(x) \ar[r]_{\chi_{Fh(x)}}  & F(g.f)Fh(x)  \ar[ur]_{\chi_x} &
}.
\end{equation*}

\item For objects $a$, $b$ of $\B$, an invertible modification $\gamma$ whose component at $f$ is itself an invertible modification whose component at $x$ is an invertible pseudo-natural transformation
\begin{equation*}
\cd[]{
  & FfF1(x) \ar[r]^{\chi_x} &  F(1.f)(x) \ar[dr]^{Fl_x} &  \\
Ff(x) \ar[ur]^-{Ff(\iota_{x})} \ar[rrr]_{1} & & & Ff(x) \dtwocell[0.5]{ulll}{\gamma_x}  
}.
\end{equation*}

\item For objects $a$, $b$ of $\B$, an invertible modification $\delta$ whose component at $f$ is itself an invertible modification whose component at $x$ is an invertible pseudo-natural transformation
\begin{equation*}
\cd[]{
Ff(x) \ar[dr]_1 \ar[rrr]^{Fr_x} & & & F(f.1)(x) \dtwocell[0.5]{dlll}{\delta_x} \\ 
  & Ff(x) \ar[r]_-{\iota_{Ff(x)}}  &  F1Ff(x) \ar[ur]_{\chi_x} &  
}.
\end{equation*}

\item There are two axioms involving $\omega$, $\delta$, $\gamma$, $\iota$, $\chi$ above. They can be found in \cite{gurski2006, gordon1995}.

%

\end{itemize}

\end{Rk}

We will describe the Grothendieck construction for fibred bicategories: a triequivalence
$$\el \colon [\B^\coop,\bicat] \to \Fib(\B)$$
that generalises the result given in section \ref{sec:strict}.
Again, we are mostly concerned with its action on objects and use ``Grothendieck construction'' to mean both the action on objects and the whole trihomomorphism.

\begin{Con}[The Grothendieck construction for bicategories] 

Suppose $F \colon \B^\coop \to \bicat$ is a trihomomorphism. We define a fibration $P_F \colon \el F \to \B$ as follows. $\el F$ is the bicategory with:

\begin{itemize}
\setcounter{enumi}{-1}
\item 0-cells are pairs $(x,\e{x})$ where $ x\in \B$ and $\e{x} \in Fx$.

\item 1-cells are pairs $(f,\e{f}) \colon (x,\e{x}) \to (y,\e{y})$ where
$f \colon x \to y$ and
$\e{f} \colon \e{x} \to Ff(\e{y})$.

\item 2-cells are pairs $(\alpha,\e{\alpha}) \colon (f,\e{f}) \To (g,\e{g}) \colon (x,\e{x}) \to (y,\e{y})$ where $\alpha \colon f\To g$ and
\begin{equation*}
\cd[@C-2pt]{
 \e{x} \ar[rr]^{\e{f}} \ar[dr]_{\e{g}} & & Ff(\e{y}) \\
 & \dtwocell[0.6]{u}{\e{\alpha}}  Fg(\e{y}) \ar[ur]_{F\alpha_{\e{y}}} &
}.
\end{equation*}

\item If $(g,\e{g}) \colon (y,\e{y})\to (z,\e{z})$ then the composite $(g,\e{g}).(f,\e{f})$ has first component $g.f$ and second component
\begin{equation*}
\cd[@C+16pt]{
\e{x} \ar[r]^{\e{f}} & Ff(\e{y}) \ar[r]^{Ff(\e{g})} & FfFg(\e{z}) \ar[r]^{\chi_{gf}} & Fgf(\e{z}).
}
\end{equation*}

\item If $(\gamma,\e{\gamma}) \colon (g,\e{g}) \to (h,\e{h}) \colon (x,\e{x}) \to (y,\e{y})$ then the vertical composite has first component $\gamma.\alpha$ and second component
\begin{equation*}
\cd[@C+12pt]{
 \e{x} \ar[rrr]^{\e{f}} \ar[drr]_{\e{g}} \ar[ddr]_{\e{h}} & & & Ff(\e{y}) \\
 & & Fg(\e{y}) \dtwocell{u}{\e{\alpha}} \ar[ur]^{F\alpha_{\e{y}}} & \\
 & Fh(\e{y})   \dtwocell[0.4]{uu}{\e{\gamma}} \ar[ur]^{F\gamma_{\e{y}}} \ar@/_36pt/[uurr]_{F\gamma\alpha_{\e{y}}} & \dtwocell[0.5]{u}{\phi_{\alpha\gamma_{\e{y}}}}&
}
\end{equation*}

\item If $(\beta,\e{\beta}) \colon (j,\e{j}) \to (k,\e{k}) \colon (y,\e{y}) \to (z,\e{z})$ then the horizontal composite has first component
$\beta*\alpha$ and second component
\begin{equation*}
\cd[@C+18pt@R+10pt]{
 \e{x} \dtwocell[0.3]{drr}{\e{\alpha}} \ar[r]^{\e{f}} \ar[dr]_{\e{g}} & Ff(\e{y}) \dtwocell[0.3]{drr}{Ff(\e{\beta})}  \ar[dr]_{Ff(\e{k})} \ar[r]^{Ff(\e{j})} & FfFj(\e{z}) \ar[r]^{\chi_{jf}}  & Fjf(\e{z})  \\
 & Fg(\e{y}) \dtwocell[0.4]{r}{F\alpha_{\e{k}}} \ar[u]|{F\alpha_\e{y}}  \ar[dr]_{Fg(\e{k})} & FfFk(\e{z}) \ar[u]|{Ff(F\beta_\e{z})}   &  \\
 &  & FgFk(\e{z}) \ar@/_28pt/[uu]_{(F\alpha F\beta)_\e{z}} \ar[u]|{F\alpha_{Fk(\e{z})}} \ar[r]_{\chi_{kg}} & Fkg(\e{z}) \ar@/_28pt/[uu]_{F\beta\alpha_\e{z}} \dtwocell[0.3]{uul}{\chi_{\beta\alpha}}
}.
\end{equation*}
The 2-cell labelled $Ff(\e{\beta})$ is strictly the composite $\phi.Ff(\e{\beta})$ where $\phi$ is an isomorphism associated with $Ff$. In order to simplify this diagram and those that follow, all isomorphisms associated with such homomorphisms have been omitted.

\item Identity 1-cells are $1_{(x,\e{x})} = (1_x, (i_x)_\e{x})$, the second component is
\begin{equation*}
\cd[@C+12pt@R-12pt]{
1_{Fx}(\e{x}) \ar[r]^{(i_x)_\e{x}}  & F{1_x}(\e{x})
}.
\end{equation*}

\item Identity 2-cells are $1_{(f,\e{f})} = (1_f, \phi_f\e{f}.l_{\e{f}})$, the second component is
\begin{equation*}
\cd[@C+18pt@R+6pt]{
\e{x} \ar@/_12pt/[dr]_{\e{f}} \ar[r]^{\e{f}} \dtwocell[0.4]{dr}{l_{\e{f}}} & Ff(\e{y}) \\
  & Ff(\e{y}) \pair{u}{1}{F{1_f}_\e{y}} \rtwocell{u}{\phi_f}
}.
\end{equation*}

\item For associativity isomorphisms, 2-cells with first component $a_{fgh}$ and second component given by the following composite.
\begin{equation*}
\cd[@R+12pt]
{
\e{w} \ar[r]^{\e{f}} & Ff(\e{x}) \ar[r]^{Ff(\e{g})}  & FfFg(\e{y}) \ar[dr]_{FfFg(\e{h})} \ar[r]^{\chi_{\e{y}}}  & \dtwocell{d}{\chi_{\e{h}}} Fgf(\e{z}) \ar[r]^{Fgf(\e{h})}  & FgfFh(\e{z}) \dtwocell{dr}{\omega^{-1}_{fgh}} \ar[r]^{\chi} & Fh(gf)(\e{z}) \\
 & &   &   FfFgFh(\e{z}) \ar[r]_{Ff(\chi)} \ar[ur]_{\chi_{Fh(\e{z})}} & FfFhg(\e{z}) \ar[r]_{\chi} & F(hg)f(\e{z}) \ar[u]_{Fa_{\e{z}}} 
}
\end{equation*}

\item For left unit isomorphisms, 2-cells with first component $l_f$ and second component given by the following composite.
\begin{equation*}
\cd[@R+12pt@C-6pt]
{
\e{x} \ar[r]^{\e{f}} & Ff(\e{y}) \ar[dr]_-{Ff(i_y)} \ar[rrr]^1 &  &  & Ff(\e{y}) \\
 & & FfF1_y(\e{y}) \ar[r]_-{\chi} \dtwocell[0.5]{ur}{\gamma^{-1}_f} Ff(\e{y})  & F1_yf(\e{y}) \ar[ur]_{Fl_{\e{y}}}
}
\end{equation*}

\item For right unit isomorphisms, 2-cells with first component $r_f$ and second component given by the following composite.
\begin{equation*}
\cd[@R+0pt]
{
\e{x} \ar[r]^{i} \ar@/_8pt/[drr]_{\e{f}} & F1(\e{x}) \ar[r]^{F1(f)} & F1Ff(\e{y})  \ar[rr]^{\chi}  & \ & Ff1(\e{y}) \\
  & & Ff(\e{y})  \dtwocell{ul}{i} \ar[u]_{i_{Ff}} \ar@/_8pt/[urr]_{Fr_{\e{y}}}  & \dtwocell{ul}{\delta^{-1}_f} &
}
\end{equation*}

\end{itemize}

By projecting onto the first component of $\el F$ we obtain a homomorphism $P_F \colon \el F \to \B$.
\end{Con}

\begin{Prop}
$P_F \colon \el F \to \B$ is a fibration.
\end{Prop}
\begin{proof}
To show that $\el F$ is a bicategory we need to use the axioms for a trihomomorphism.
\begin{itemize}
\item $\el F((x,\e{x}),(y,\e{y}))$ is a category: the definition of vertical composition uses the coherence isomorphisms for the homomorphism $F_{xy}$. The three axioms for this homomorphism give associativity and left and right identity in $\el F((x,\e{x}),(y,\e{y}))$.
\item Horizontal composition is functorial: the definition of horizontal composition uses the isomorphism $\chi_{\beta\alpha}$ from the transformation $\chi$. The two axioms for this transformation make horizontal composition in $\el F$ functorial.
\item Coherence axioms: the coherence isomorphisms make use of invertible modifications $\omega$, $\gamma$ and $\delta$. The two axioms for these modifications ensure that the axioms for a bicategory hold.
\end{itemize}

Observe that composition in $\el F$ in the first component is just composition in $\B$, thus $P_F$ is a homomorphism that preserves composition and identities strictly.
We need to show that $P_F \colon \el F \to \B$ is a fibration.

Suppose that $f \colon x \to y$ in $\B$ and $(y,\e{y})$ is in $\el F$. We claim that $ \c{f}{(y,\e{y})} = (f, 1_{Ff(\e{y})}) \colon (x, Ff(\e{y})) \to (y,\e{y})$ is cartesian over $f$.
\begin{equation*}
\cd[@R-6pt]{
Ff(\e{y}) \ar[r]_{1_{Ff(\e{y})}} & Ff(\e{y})
}
\end{equation*}
Whenever
\begin{equation*}
\cd[@C+6pt]{
(z,\e{z}) \ar[dr]^{(g,\e{g})} \ar[d]_{(h,\ )} & \urtwocell[0.65]{dl}{(\alpha,\ )} \\
(x,Ff(\e{y})) \ar[r]_-{(f,1)} & (y,\e{y})
}
\hspace{0.5cm}\text{over}\hspace{0.5cm}
\cd[@C+18pt@R+6pt]{
z \ar[dr]^{g} \ar[d]_{h} & \urtwocell[0.7]{dl}{\alpha} \\
x \ar[r]_f & y
}
\end{equation*}
we can choose $\e{h} = \chi^{\centerdot}.F\alpha_{\e{y}}.\e{g}$ and
$\e{\alpha}$ equal to
\begin{equation*}
\cd[@C+6pt]{
z \ar[r]^{\e{g}} \ar@/_20pt/[drrrr]_{\e{g}} & Fg(\e{y}) \ar[r]^{F\alpha_\e{y}} & Ffh(\e{y}) \ar@/_36pt/[rrr]_(0.3){1} \ar[r]^{\chi^{\centerdot}} \dtwocell{d}{l} & FhFf(\e{y}) \ar[r]^{Fh(1)} \ar@/_20pt/[rr]_(0.3){\chi} \ar@/_4pt/[r]_{1} \dtwocell[0.35]{d}{\epsilon} & \dtwocell[0.3]{d}{r} FhFf(\e{y}) \ar[r]^{\chi} & Ffh(\e{y}) \\
&&&& Fg(\e{y}) \ar@/_8pt/[ur]_{F\alpha_\e{y}} &
}
\end{equation*}
to give a lifting of $(h,\alpha)$ as required. Showing the 2-cell property is not very difficult but requires large diagrams that we will not include here.
Thus $P_F$ has cartesian 1-cells.

Suppose that $\alpha \colon f \To g \colon x \to y$ in $\B$ and $(g,\e{g})$ is in $\el F$. We claim that $ \c{\alpha}{(g,\e{g})} = (\alpha, 1_{F\alpha_{\e{y}}g}) \colon (f, F\alpha_{\e{y}}g) \To (g,\e{g}) \colon (x,\e{x}) \to (y,\e{y})$ is cartesian over $\alpha$.
\begin{equation*}
\cd[@C+16pt]{
x \ar[r]^-{F\alpha_{\e{y}}g} \ar[dr]_{g} & Ff(\e{y}) \\
\dtwocell[0.65]{ur}{1}  & Fg(\e{y}) \ar[u]_{F\alpha_{\e{y}}}
}
\end{equation*}
Whenever
\begin{equation*}
\cd[@C+6pt]{
(h,\e{h}) \ar@{=>}[dr]^{(\gamma,\e{\gamma})} \ar@{=>}[d]_{(\delta,\ )} &  \\
(f,F\alpha_{\e{y}}g \ar@{=>}[r]_-{(\alpha,1)} & (g,\e{g})
}
\hspace{0.5cm}\text{over}\hspace{0.5cm}
\cd[@C+18pt]{
h \ar@{=>}[dr]^{\gamma} \ar@{=>}[d]_{\delta} &  \\
f \ar@{=>}[r]_\alpha & g
}
\end{equation*}
we find that $\e{\delta} = \e{\gamma}.\phi_{\delta\alpha}^{-1} g$ and that this uniquely makes $(\gamma,\e{\gamma}) = (\alpha,1).(\delta,\e{\delta})$. This occurs precisely because $1_{F\alpha_{\e{y}}g}$ is an isomorphism.
Thus $P_F$ has cartesian 2-cells.

The horizontal composite of the two lifts $(\alpha, 1_{F\alpha_{\e{y}}g})$, $(\beta, 1_{F\beta_{\e{z}}k})$ has first component $\beta*\alpha$ and second component a pasting of $\chi_{\beta\alpha}$ and $F\alpha_{\e{g}}$. Since the second component is an isomorphism we can use the argument above to show that this is cartesian.
Thus by Proposition \ref{chosen_2-cells_enough_weak} cartesian 2-cells are closed under horizontal composition.
\end{proof}

\begin{Con}[Pseudo-inverse to the Grothendieck construction] \label{con:GC_inverse_weak}

Suppose that $P\colon \E \to \B$ is a fibration. We define a trihomomorphism $F_P\colon \B^\coop \to \bicat$ as follows:

\vspace{6pt}
\noindent \emph{on 0-cells:} $Fb = \E_b$ for all $b\in B$. $\E_b$ is the fibre of $P$ over $b$. Its 0-, 1- and 2-cells are those in $\E$ that map to $b$, $1_b$ and $1_{1_b}$. Horizontal composition of 1-cells is defined by $g\hat{*}f = r^*(g.f)$: the domain of the cartesian lift of $r\colon 1_b \to 1_b.1_b$ at $g.f$. Composition of 2-cells is defined to be the unique 2-cell in the following diagram.
\begin{equation*}
\cd[]{
g\hat{*}f \ar@{=>}[r]^{\varphi} \ar@{:>}[d]_{\alpha\hat{*}\beta} & g.f \ar@{=>}[d]^{\alpha*\beta} \\
k\hat{*}h \ar@{=>}[r]_{\varphi} & k.h
}
\hspace{0.5cm}\text{over}\hspace{0.5cm}
\cd[]{
1 \ar@{=>}[r]^{r} \ar@{=>}[d]_{1_1} & 1.1 \ar@{=>}[d]^{1_1*1_1} \\
1 \ar@{=>}[r]_{r} & 1.1
}
\end{equation*}
Identities are given by
\begin{equation*}
\cd[@C+6pt]{
\hat{1} \ar@{=>}[r]^{\varphi} \ar@{:>}[d]_{\hat{1_1}} & 1 \ar@{=>}[d]^{1_1} \\
\hat{1} \ar@{=>}[r]_{\varphi} & 1
}
\hspace{0.5cm}\text{over}\hspace{0.5cm}
\cd[]{
1 \ar@{=>}[r]^{\phi} \ar@{=>}[d]_{1_1} & P1 \ar@{=>}[d]^{1_1} \\
1 \ar@{=>}[r]_{\phi} & P1
}
\end{equation*}
and coherence isomorphisms by
\begin{equation*}
\cd[]{
h\hat{*}(g\hat{*}f) \ar@{=>}[r]^{\varphi} \ar@{:>}[d]_{\hat{a}} & h.(g\hat{*}f) \ar@{=>}[r]^{h\varphi} & h.(g.f) \ar@{=>}[d]^{a} \\
(h\hat{*}g)\hat{*}f \ar@{=>}[r]_{\varphi} & (h\hat{*}g).f \ar@{=>}[r]_{\varphi f} & (h.g).f
}
\hspace{0.5cm}\text{over}\hspace{0.5cm}
\cd[]{
1 \ar@{=>}[r]^{r} \ar@{=>}[d]_{1_1} & 1.1 \ar@{=>}[r]^{1.r} & 1.(1.1) \ar@{=>}[d]^{a} \\
1 \ar@{=>}[r]_{r} & 1.1 \ar@{=>}[r]_{r.1} & (1.1).1
}
\end{equation*}
et cetera. 
The uniqueness of these 2-cells guarantees that the middle-four interchange holds and that these isomorphisms satisfy the axioms for a bicategory.

\vspace{6pt}
\noindent \emph{on 1-cells:} $Ff = f^*\colon \E_{b'} \to \E_b$ is the homomorphism described using the following diagram (isomorphisms omitted).
\begin{equation*}
\cd[@C+32pt@R+18pt]{
f^*e \dotpair{d}{f^*k}{f^*h} \dotrtwocell{d}{f^*\alpha} \ar[r]^{\c{f}{e}} &
e \pair{d}{k}{h} \rtwocell{d}{\alpha} \\
f^*e' \ar[r]_{\c{f}{e'}} & e'
}
\hspace{0.5cm}\text{over}\hspace{0.5cm}
\cd[@C+32pt@R+18pt]{
b \ar[r]^f \pair{d}{1_b}{1_b} \twoeq{d} 
& b' \pair{d}{1_{b'}}{1_{b'}} \twoeq{d}\\
b \ar[r]_f & b'
}
\end{equation*}
$Ff$ sends $e$ to the domain of $\c{f}{e}$. Using the cartesian 1-cell $\c{f}{e'}$ we send $h$, $k$ to $f^*h$, $f^*k$ over $1_b$ with an iso-square on the front and back and $\alpha$ is send to the unique $f^*\alpha$ over $1_{1_b}$. The action of $f^*$ on 1-cells is only defined up to a unique isomorphism. The coherence isomorphisms for $f^*$ are precisely the unique 2-cells that arise when comparing $f^*h'\hat{*}f^*h$ to $f^*(h'*h)$ and $\hat{1}_{f^*e}$ to $f^*(1_e)$.

Again, the uniqueness of these maps ensures that $f^*$ preserves vertical composition of 2-cells in the fibres and that the coherence isomorphisms satisfy the appropriate axioms.

\vspace{6pt}
\noindent \emph{on 2-cells:}  $F\sigma=\sigma^*\colon g^* \To f^*\colon E_{b'} \to E_b$ is the transformation described by the following diagrams (isomorphisms omitted).
\begin{equation*}
\cd[@R-12pt@C+12pt]{
 & f^*e \ar@/^5pt/[dr]^{\c{f}{e}} & \\
g^*e \ar@{.>}@/^5pt/[ur]^{\sigma_e^*} \pair{rr}{g_e}{\c{g}{e}} \dtwocell[0.3]{rr}{\c{\sigma}{\c{g}{e}}}  & & e \\
}
\hspace{0.5cm}\text{over}\hspace{0.5cm}
\cd[@R-12pt@C+12pt]{
 & b \ar@/^5pt/[dr]^{f} & \\
b \pair{rr}{f}{g} \dtwocell{rr}{\sigma} \ar@/^5pt/[ur]^{1} & & b'
}
\end{equation*}
We take the cartesian lift of $\sigma$ at $\c{g}{e}$ and factor its domain as $\c{f}{e}.\sigma_e^*$ (see Proposition \ref{prop:1-cells_factor_weak} below). Then $(F\sigma)_e = \sigma_e^*$. Now suppose $k\colon e \to e'$ and consider the action of $f^*$ and $g^*$ on $k$. We construct $F\sigma_k = \sigma^*_k$ as the unique isomorphism in the diagram
\begin{equation*}
\cd[@R-12pt]{
\c{f}{e'}.f^*k*\sigma^*_e \ar@{=>}[r]^{1\varphi} \ar@{:>}[d]_{1\sigma^*_k} & \c{f}{e'}.f^*k.\sigma^*_e \ar@{=>}[r]^{\tau_f 1} & k.\c{f}{e}.\sigma^*_e \ar@{=>}[r]^-{1\varphi} & k.g_e \ar@{=>}[d]^{\tau_g} \\
\c{f}{e'}.\sigma^*_{e'}*g^*k \ar@{=>}[r]_{1\varphi} & \c{f}{e'}.\sigma^*_{e'}.g^*k \ar@{=>}[rr]_{\varphi 1} &  & g_{e'}.g^*k \\
}
\end{equation*}
over
\begin{equation*}
\cd[@R-12pt@C+8pt]{
f.1 \ar@{=>}[r]^{f.r} \ar@{=>}[d]_{f.1_1} & f.1.1 \ar@{=>}[r]^{lr.1}  & 1.f.1 \ar@{=>}[r]^{1.r} & 1.f \ar@{=>}[d]^{lr} \\
f.1 \ar@{=>}[r]_{f.r} & f.1.1 \ar@{=>}[rr]_{r.1} & & f.1
}.
\end{equation*}
where $\tau_f$ is the isomorphism associated with the action of $f^*$ on $k$. It is obtained using the cartesian property of $\c{f}{e'}$ and the indicated diagrams.

Again, the uniqueness of the $\sigma^*_k$ ensures that $\sigma^*$ is actually a transformation.

\vspace{6pt}
$F$ is locally a homomorphism: Suppose $\alpha \colon f \To g \colon b \to b'$ and $\beta \colon g \To h$. Then 
\begin{equation*}
\cd[@R+20pt]
{
 & & f^*e \ar[ddr]^{\c{f}{e}} \twocong[0.35]{dd} &   \\
 & g^*e \ar[ur]^{\alpha^*_e} \twocong{d} \ar@/_10pt/[drr]_(0.4){\c{g}{e}} \ar@/^14pt/[drr]^(0.4){g_e} \dtwocell{drr}{\c{\alpha}{\ }} & &  \\
 h^*e \ar[ur]^{\beta^*_e} \ar@/_10pt/[rrr]_(0.3){\c{h}{e}} \ar@/^10pt/[rrr]^(0.3){h_e}  \dtwocell[0.3]{rrr}{\c{\beta}{\ }} & & & e
}
\hspace{0.3cm}\text{and}\hspace{0.3cm}
\cd[@R+20pt]
{
 & & f^*e \ar[ddr]^{\c{f}{e}} \twocong{dd} &   \\
 & & &  \\
 h^*e \ar[uurr]^{(\alpha.\beta)^*_e} \pair{rrr}{h_e}{\c{h}{e}} \dtwocell{rrr}{\c{\alpha.\beta}{\ }} & & & e
}
\end{equation*}
both sit over 
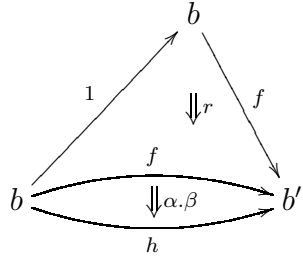
\begin{equation*}
\cd[]
{
 & & b \ar[ddr]^{f} \dtwocell{dd}{r}&   \\
 & & &  \\
 b \ar[uurr]^{1} \pair{rrr}{f}{h} \dtwocell{rrr}{\alpha.\beta} & & & b'
}
\end{equation*}
so there exists a unique isomorphism $(\phi_{\beta \alpha})_e \colon (\alpha.\beta)^*_e \To \alpha^*_e.\beta^*_e$. This is the component of a modification and is one of the coherence isomorphisms for $F_{bb'} \colon \B(b,b') \to \bicat(\E_{b},\E_{b'})$. The isomorphism for identities $\phi_f \colon (1_f)^* \To (1_{f^*})$ is formed in a similar way. Their uniqueness ensures that they satisfy the appropriate axioms.

Horizontal composition is preserved up to pseudo-natural equivalence: Suppose $\alpha \colon f \To g \colon b \to b'$ and $\beta \colon h \To k \colon b' \to b''$ then since cartesian 1-cells are unique up to equivalence we get an equivalence ${\chi_{hf}}_e$
\begin{equation*}
\cd[@C+10pt]
{
 (hf)^*e \ar[drr]^{\c{hf}{e}} \ar[d]_{{\chi_{hf}}_e}  & \urtwocell[0.6]{dl}{\hat{r}} &   \\
 f^*h^*e \ar[r]_-{\c{f}{h^*e}} & h^*e \ar[r]_-{\c{h}{e}} & e
}
\hspace{1cm}\text{over}\hspace{1cm}
\cd[@C+16pt]
{
 b \ar[drr]^{hf} \ar[d]_{1} & \urtwocell[0.6]{dl}{r} &   \\
 b \ar[r]_{f} & b' \ar[r]_{h} & b''
}
\end{equation*}
that is unique up to isomorphism. It is the 1-cell component of a transformation $\chi_{hf} \colon f^*h^* \To (hf)^*$. The 2-cell component of $\chi_{hf}$ is obtained as the unique 2-cell comparing two 1-cell lifts along a given cartesian 1-cell. Essentially every 2-cell isomorphism in this construction is obtained this way and all the the relevant axioms hold by uniqueness. For example, the invertible modification $\chi_{\beta\alpha}$ is described the same way and satisfies the axioms for a modification by uniqueness. 

Identity 1-cells are preserved up to pseudo-natural equivalence: Suppose $1_b \colon b \to b$ then since cartesian 1-cells are unique up to equivalence we get an equivalence ${\iota_{b}}_e$
\begin{equation*}
\cd[@C+10pt]
{
 e \ar[drr]^{1} \ar[d]_{{\iota_{b}}_e} & \urtwocell[0.6]{dl}{\hat{r}} &   \\
 1^*e \ar[rr]_{\c{1}{e}} &  & e
}
\hspace{1cm}\text{over}\hspace{1cm}
\cd[@C+16pt]
{
 b \ar[drr]^{1} \ar[d]_{1} & \urtwocell[0.6]{dl}{r} &   \\ 
b \ar[rr]_{1} & & b''
}
\end{equation*}
that is unique up to isomorphism. It is the 1-cell component of a transformation $\iota_{b} \colon (1_b)^* \To 1_{\E_b}$. The 2-cell component of $\iota_b$ and the modification $\iota_{1_b}$ are constructed in a similar way to that above.

Invertible modifications $\omega$, $\gamma$, $\delta$: As above, these are obtained using the 2-cell property for cartesian 1-cells. For ${\omega_{fgh}}_e$ we use 
$$ \cd[@C+24pt]{ f^*(g^*(h^*e)) \ar[r]^-{\c{f}{g^*(h^*e)}} & g^*(h^*e) \ar[r]^-{\c{g}{h^*e}} & h^*e \ar[r]^-{\c{h}{e}} & e }$$
and the appropriate liftings from the definitions of $\chi_{gf}$ et cetera. The other two are done in a similar way. We then use the uniqueness property to show that they satisfy the relevant axioms.

\end{Con} 

This gives us the following result.

\begin{Prop}
$F_P\colon \B^\coop \to \bicat$ is a trihomomorphism.
\end{Prop}

\begin{Rk}[Fibres]
We defined the fibre over an object $b$ by insisting that each 1-cell and 2-cell sit exactly above $1_b$ and $1_{1_b}$. Then the composition is such that this actually forms a bicategory. We could give a different kind of fibre by asking that 1-cells only have $Pf \cong 1_b$ and that (the image of) 2-cells in the fibre commute with these isomorphisms. The construction would work either way. Our choice of fibre is more stable under the action of the Grothendieck construction and its pseudo-inverse.
\end{Rk}

To prove that the Grothendieck construction is surjective up to equivalence we will need the following two results.

\begin{Prop} \label{prop:1-cells_factor_weak}
Suppose $P \colon \E \to \B$ is a fibration. Every $f\colon x\to z$ in $\E$ can be factored as
\begin{equation*}
\cd[@C+16pt]{
x \ar[dr]^f \ar[d]_{\hat{f}}& \\
y \ar[r]_-{\c{Pf}{z}} \urtwocell[0.35]{ur}{\hat{l}}& z
}
\hspace{0.5cm}\text{over}\hspace{0.5cm}
\cd[@C+16pt]{
Px \ar[d]_{1_{Px}} \ar[dr]^{Pf} & \\
Px \ar[r]_-{Pf} \urtwocell[0.35]{ur}{l} & Pz
}
\end{equation*}
where $\hat{f}$ is unique up to unique isomorphism.
\end{Prop}
\begin{proof}
The cartesian property of $\c{Pf}{z}$ gives a factorisation as shown above.
Suppose that there is another factorisation $(\hat{f}',\hat{l}')$ and consider the commuting shell formed with $1_{Pf}$ and $1_{1_{Px}}$ in the base. There is then a unique isomorphism $\tau\colon \hat{f} \cong \hat{f}'$ with $\hat{l} = \hat{l}'.\c{Pf}{z}\tau$ and $P\tau = 1_{1_{Px}}$.
\end{proof}

\begin{Prop} \label{prop:2-cells_factor_weak}
Suppose $P \colon \E \to \B$ is a fibration. Every $\alpha\colon f\To g\colon w\to z$ in $\E$ can be factored as
\begin{equation*}
\cd[@C+12pt]{
w \ar@/^48pt/[drrr]^{f} \ar@/_48pt/[drrr]_{g} \pair[10pt]{rr}{\hat{f}}{\hat{h}*\hat{g}} \dtwocell{rr}{\hat{\alpha}} \ar@/_4pt/[dr]|{\hat{g}} & & x \ar@/^4pt/[dr]|{\c{Pf}{z}\ \ \ } &  \\
 & y \pair{rr}{h}{\c{Pg}{z}} \dtwocell[0.3]{rr}{\c{P\alpha}{\c{g}{z}}}  \ar[ur]|{\hat{h}} & & z
}
\hspace{0.5cm}\text{over}\hspace{0.5cm}
\cd[@C+8pt]{
Pw \ar@/^48pt/[drrr]^{Pf} \ar@/_48pt/[drrr]_{Pg}  \pair[10pt]{rr}{1}{1} \ar@/_4pt/[dr]|{1} & & Pw \ar@/^4pt/[dr]|{Pf} &  \\
 & Pw \pair{rr}{Pf}{Pg} \dtwocell{rr}{P\alpha}  \ar[ur]|{1} & & Pz
}
\end{equation*}
where $\hat{\alpha}$ is unique up to choice of $\hat{f}$, $\hat{g}$ and $\hat{h}$. (Invertible 2-cells have been omitted in each diagram).
\end{Prop}
\begin{proof}
The structure of this proof is the same as Proposition \ref{prop:2-cells_factor_strict}.
Begin by factoring $f \cong \c{Pf}{z}\cdot\hat{f}$ and $g \cong \c{Pg}{z}\cdot\hat{g}$ using Proposition \ref{prop:1-cells_factor_weak}.
Now take the cartesian lift of $P\alpha$ at $\c{Pg}{z}$.
Then $\c{P\alpha}{\c{Pg}{z}}\hat{g}$ is cartesian over $P\alpha.1$ and
\begin{equation*}
\cd[@C-4pt]
{
 \c{Pf}{z}\hat{f} \ar@{=>}[r]^-{\tau_f} \ar@{:>}[d]_{\eta} & f \ar@{=>}[r]^{\alpha} & g \ar@{=>}[d]^{\tau_g}   \\
 h\hat{g} \ar@{=>}[rr]_-{\c{P\alpha}{\c{Pg}{z}}\hat{g}} & &  \c{Pg}{z}\hat{g}
}
\hspace{0.5cm}\text{over}\hspace{0.5cm}
\cd[@C-4pt]
{
 Pf.1 \ar@{=>}[r]^{r}  \ar@{=>}[d]_{1} & Pf \ar@{=>}[r]^{P\alpha} & Pg \ar@{=>}[d]^{r}   \\
 Pf.1 \ar@{=>}[rr]_{P\alpha.1} & &  Pg.1
}
\end{equation*}
so there exists a unique $\eta \colon \c{Pf}{z}\hat{f} \To h\hat{g}$ with $P\eta = 1_{Pf.1}$ and
\begin{equation*}
\cd[@C+6pt@R+6pt]
{
w \ar@/^6pt/[rr]^{\hat{f}} \ar@/_6pt/[dr]_{\hat{g}} \pair{drrr}{f}{g} \dtwocell{drrr}{\alpha} & & y \ar@/^6pt/[dr]^{\c{Pf}{z}} &  \\
 & x \ar@/_6pt/[rr]_{\c{Pg}{z}}  & & z
}
=
\cd[@C+6pt@R+6pt]
{
w \ar@/^6pt/[rr]^{\hat{f}} \ar@/_6pt/[dr]_{\hat{g}} & & \dtwocell{dl}{\eta} y \ar@/^6pt/[dr]^{\c{Pf}{z}} &  \\
 & x \pair[10pt]{rr}{h}{\c{Pg}{z}} \dtwocell[0.35]{rr}{\c{P\alpha}{\dots}} & & z
}.
\end{equation*}
By Proposition \ref{prop:1-cells_factor_weak} we factor $h \cong \c{Pf}{z}\cdot\hat{h}$ where $P\hat{h} = 1_{Px}$.
We then form the fibre-composite of $\hat{h}$ and $\hat{g}$ by lifting the isomorphism $1\cong 1.1$. Finally, we observe that
\begin{equation*}
\cd[@C+6pt@R-4pt]
{
 \c{Pf}{z}\hat{f} \ar@{=>}[r]^{\eta} \ar@{=>}[dd]_{1} & h\hat{g} \ar@{=>}[d]^{\tau_h \hat{g}}   \\
 & \c{Ph}{z}\hat{h}\hat{g} \ar@{=>}[d]^{\c{Ph}{z}\c{r}{-}} \\
 \c{Pf}{z}\hat{f} \ar@{:>}[r]_-{\c{Ph}{z}\hat{\alpha}} & \c{Ph}{z}(\hat{h}\hat{*}\hat{g})
}
\hspace{0.5cm}\text{over}\hspace{0.5cm}
\cd[@C+24pt@R-4pt]
{
 Pf.1 \ar@{=>}[r]^{1} \ar@{=>}[dd]_{1} & Ph.1 \ar@{=>}[d]^{r1}   \\
 & Pf.1.1 \ar@{=>}[d]^{Pfr} \\
 Pf.1 \ar@{=>}[r]_{Pf.1_1} & Pf.1
}
\end{equation*}
so there exists a unique $\hat{\alpha}\colon \hat{f} \To \hat{h}\hat{*}\hat{g}$ over $1_{1_{Pw}}$ with $\c{Ph}{z}\c{r}{1}.\tau_h\hat{g}.\eta = \c{Pf}{z} \hat{\alpha}$ and hence a decomposition of $\alpha$ as stated.

Suppose that there was another decomposition of $\alpha$ that obtained $\hat{\alpha}'$ using $\hat{f}'$, $\hat{g}'$ and $\hat{h}'$. Then there exist unique $\tau_f$, $\tau_g$ and $\tau_h$ as in Proposition \ref{prop:1-cells_factor_weak}. Using the fact that $\hat{\alpha}'$ is unique, we find that $\hat{\alpha}' = (\tau_h\hat{*}\tau_g).\hat{\alpha}.\tau_f$.
\end{proof}

\begin{Con}[Surjective up to biequivalence]
In order to show that the Grothendieck construction is surjective up to biequivalence we need to find for every $P$ a biequivalence of fibrations 
\begin{equation*}
\cd[@C-1pt]{
\el{F_P} \ar[rr]^H \ar[dr]_{\pi}& & \ar[dl]^{P} E \\
 & B &
}.
\end{equation*}
First, what is $\el F_P$? Its data consists of:
\begin{enumerate}[1\emph{-cells:}]
\setcounter{enumi}{-1}
\item pairs $(x,\e{x})$ where $\e{x} \in \E$ and $P\e{x}=x$.
\item pairs $(f,\e{f})\colon (x,\e{x}) \to (y,\e{y})$ where $f\colon x\to y$ in $\B$ and $\e{f}~\colon~\e{x}\to~f^*(\e{y})$ in $\E_{x}$.
\item pairs $(\alpha,\e{\alpha})\colon (f,\e{f}) \To (g,\e{g})\colon (x,\e{x}) \to (y,\e{y})$ where $\alpha\colon f\To g$ in $\B$ and $\e{\alpha}\colon \e{f} \To \alpha^*_{\e{y}}\hat{*}\e{g}$ in $\E_x$.
\end{enumerate}
\begin{equation*}
\cd[@!C@C+8pt@R-4pt]{
\e{x} \pair[9pt]{rr}{\e{f}}{\alpha^*_{\e{y}}*\e{g}} \ar[dr]_{\e{g}} \dtwocell{rr}{\e{\alpha}}  & & f^*({\e{y}}) \\
& g^*(\e{y}) \ar[ur]_{\alpha^*_{\e{y}}} &
}
\end{equation*}

Then $H\colon \el F_P \to \E$ is defined on 0-cells by $H(x,\e{x})=\e{x}$.
On 1-cells by $H(f,\e{f}) = r^*(\varphi.\e{f})$, the domain of the cartesian lift of $r\colon f \To f.1$ at
\begin{equation*}
\cd[@C+16pt]{
\e{x} \ar[r]^{\e{f}} & f^*\e{y} \ar[r]^{\c{f}{\e{y}}} & \e{y}
}.
\end{equation*}
On 2-cells $H(\alpha,\e{\alpha})$ is the composite
\begin{equation*}
\cd[@C+10pt]{
\e{x} \ar@/^48pt/[drrr]^{H(f,\e{f})} \ar@/_48pt/[drrr]_{H(g,\e{g})} \pair{rr}{\e{f}}{\alpha^*_\e{y}*\hat{g}} \dtwocell{rr}{\e{\alpha}} \ar@/_4pt/[dr]|{\e{g}} & & f^*\e{y} \ar@/^4pt/[dr]|{\c{f}{\e{y}}\ } &  \\
 & y \pair{rr}{h}{\c{g}{\e{y}}} \dtwocell[0.35]{rr}{\c{\alpha}{\c{g}{\e{y}}}}  \ar[ur]|{\alpha^*_\e{y}} & & \e{y}
}.
\end{equation*}
The coherence isomorphisms are the unique fillers in the following diagrams.
\begin{equation*}
\cd[@C+0pt]{
H(g,\e{g}).H(f,\e{f}) \ar@{=>}[r]^-{\hat{r}*\hat{r}} \ar@{:>}[dd]_{} & \c{Pg}{\ }.\e{g}.\c{Pf}{\ }.\e{f} \ar@{=>}[d]^{\c{Pg}{\ }.\cong.\e{f}} \\
& \c{Pg}{\ }.\c{Pf}{\ }.f^*(\e{g}).\e{f} \ar@{=>}[d]^{\cong.f^*(\e{g}).\e{f}} \\
H((g,\e{g}).(f,\e{f})) \ar@{=>}[r]_-{\hat{r}} & \c{Pgf}{\ }.\chi.f^*(\e{g}).\e{f}
}
\end{equation*}
\begin{equation*}
\cd[@C+0pt]{
H((1,(i_x)_{\e{x}}) \ar@{=>}[r]^{\hat{r}} \ar@{:>}[d]_{} & \c{P1}{\ } (i_x)_{\e{x}} \ar@{=>}[d]^{\cong} \\
1_{H(x,\e{x})}   \ar@{=>}[r]_{1} & 1_{\e{x}}
}
\end{equation*}
Note that the coherence maps $\chi$ and $\iota$ are those given by the inverse to the Grothendieck construction.
\end{Con}
 	
\begin{Prop} \label{prop:H_is_equiv_weak}
The functor $H$ is a cartesian biequivalence.
\end{Prop}
\begin{proof}
First,
\begin{align*}
RF_P(x,\e{x})   = x   &  = PH(x,\e{x}) \\
RF_P(f,\e{f})   = f   &  = PH(f,\e{f}) \\
RF_P(\alpha,\e{\alpha}) = \alpha &  = PH(\alpha,\e{\alpha})
\end{align*}
so $PH = RF_P$.

Second, the chosen cartesian maps in $\el F_P$ are those with identities in the second component.
Since $H$ acts by post-composition with cartesian maps it is cartesian on chosen cartesian maps, thus $H$ is cartesian.

Third, for every $e\in \E$ there exists $(Pe,e)\in \el F_P$ with $H(Pe,e) = e$ so $H$ is surjective on objects. 
Then each 1-cell in the image of $H$ is a composite of a factorisation according to Proposition \ref{prop:1-cells_factor_weak}. Since such factorisations are unique up to isomorphism, $H$ is surjective up to isomorphism on 1-cells. Finally, all 2-cells in the image of $H$ are composites of factorisations according to Proposition \ref{prop:2-cells_factor_weak}. Since such factorisations are unique (up to the given factorisation of the 1-cells), $H$ is appropriately bijective on 2-cells. Thus $H$ is a biequivalence.
\end{proof}

\begin{Thm}
\label{thm:GC_weak}
The Grothendieck construction is the action on objects of a triequivalence
$$\el\colon [\B^\coop, \bicat] \to \Fib(\B).$$

\end{Thm}
\begin{proof}
We have already shown that on objects $\el$ is surjective up to biequivalence (Proposition \ref{prop:H_is_equiv_weak}). Showing that it is locally a biequivalence requires many pages of verification. We will present most of the required data but omit many of the details.

Suppose $\eta \colon F \To G$ is a tritransformation in $[\B^\coop,\bicat]$. We define $\el{\eta} \colon \el F \to \el G$ by
$\el{\eta}(x,\e{x})  = (x, \eta_x(\e{x}))$ on objects and $\el{\eta}(f,\e{f}) = (f, {\eta_f}_{\e{y}}.\eta_x(\e{f}))$ on 1-cells. The first component of $\el{\eta}(\alpha,\e{\alpha})$ is $\alpha$ and the second is the composite
\begin{equation*}
\cd[@R+12pt]
{
\eta_x(\e{x}) \ar[r]^{\eta_x(\e{f})} \ar[dr]_{\eta_x(\e{g})} & \eta_xFf(\e{y}) \ar[r]^{{\eta_f}_\e{y}} & Gf\eta_y(\e{y}) \\
 \dtwocell[0.6]{ur}{\eta_x(\e{\alpha})}  & \eta_xFg(\e{y}) \ar[u]_{\eta_xF\alpha_\e{y}} \ar[r]_{{\eta_g}_\e{y}} & Gg \eta_y(\e{y}) \ar[u]_{{G\alpha \eta_y}_\e{y}} \dtwocell{ul}{{\eta_\alpha}_\e{y}}
}.
\end{equation*}
The coherence isomorphisms $\phi_{fg}$ have first component $1_{fg}$ and second component
\begin{equation*}
\hbox to \textwidth{\hss $
\cd[@R+12pt]
{
\eta_x(\e{x}) \ar[d]^{\eta_x(\e{f})} & & & & \\
 \eta_xFf(\e{y}) \ar[dr]_{\eta_x Fg(\e{f})} \ar[r]^{{\eta_f}_\e{y}} & Gf\eta_y(\e{y}) \ar[r]^-{Gf(\e{g})} & Gf\eta_y Fg(\e{z}) \ar[r]^{Gf{\eta_g}_y} & GfGg\eta_z(\e{z}) \ar[r]^{\chi 1} & Ggf \eta_z(\e{z}) \\
 & \eta_x FgFf \ar[ur]_{\eta_f Fg} \ar[rr]_{\eta_x \chi} \dtwocell{u}{\eta_f} & & \eta_x Fgf \ar[ur]^{\eta_{gf}} \dtwocell{ul}{\Pi} \ar[r]_{\eta_{gf}} & Ggf \eta_z \ar[u]_{{G1_{\eta_z}}_{\e{z}}} \twocong[0.3]{ul}
}
$ }
\end{equation*}
where $\Pi$ is part of the data of a tritransformation.
The coherence isomorphisms $\phi_{x}$ have first component $1_{1_x}$ and second component
\begin{equation*}
\cd[]
{
\eta_x(\e{x}) \ar[dr]_{\iota_{\eta_x(\e{x})}} \ar[r]^{\eta_x(i)} & \eta_xF1(\e{x}) \ar[r]^{\eta_1} & G1\eta_x(\e{x}) \\
 & G1\eta_x(\e{x}) \ar[ur]_{{G1\eta_x}_\e{x}} \dtwocell{u}{M} &
}.
\end{equation*}
$M$ is part of the data of a tritransformation.

This is a cartesian homomorphism from $P_F$ to $P_G$ and defines the action of $\el$ on 1-cells.

Suppose $\Gamma \colon \eta \Rrightarrow \epsilon$ is a trimodification in $[\B^\coop,\bicat]$. We define a transformation $\el{\Gamma} \colon \el{\eta} \to \el{\epsilon}$. The first component of $\el{\Gamma}_{(x,\e{x})}$ is $1_x$ and second component is
\begin{equation*}
\cd[@R+12pt]
{
\eta_x(\e{x}) \ar[r]^{{\Gamma_x}_\e{x}}  & \tau_x(\e{x}) \ar[r]^{i} & G1\tau_x(\e{x})
}.
\end{equation*}
The first component of $\el{\Gamma}_{(f,\e{f})}$ is $lr$ and second component is
\begin{equation*}
\hbox to \textwidth{\hss $
\cd[@R+12pt]
{
\eta_x(\e{x}) \ar[r]^{\Gamma_x} \ar[dr]_{\eta_x(\e{f})} & \tau_x(\e{x}) \ar[dr]^{\tau_x(\e{f})} \ar[r]^{\iota \tau} & G1\tau_x(\e{x}) \ar[r]^-{G1\tau_x(\e{f})} & G1\tau_xFf(\e{y}) \ar[r]^{G1 \tau_f} & G1Gf\tau_x(\e{y}) \ar[r]^{\chi 1} & Gf1\tau_x(\e{y}) \\
  & \eta_xFf(\e{y}) \twocong{u} \ar[r]^{\Gamma_xFf} \ar[dr]_{\eta_f} & \tau_xFf(\e{y}) \twocong{u} \ar[ur]_{\iota \tau Ff} \ar[dr]^{\tau_f} &&& Gf\tau_y(\e{y}) \ar[u]|{Gr\tau_{\e{y}}} \dtwocell{ul}{\delta} \\
 & & Gf\eta_y(\e{y}) \ar[r]_{Gf\Gamma_y} \dtwocell{u}{m} & Gf\tau_y(\e{y}) \twocong{uu} \ar[urr]^{1} \ar[uur]_{\iota Gf\tau} \ar[r]_{Gf\iota \tau_y} & GfG1\tau_y(\e{y})  \ar[r]_{\chi 1} & G1f \tau_y(\e{y}) \ar[u]|{Gl\tau_{\e{y}}} \ar@/_24pt/[uu]_{{Glr\tau_y}_{\e{y}}} \dtwocell{ul}{\gamma}
}
$ \hss }
\end{equation*}
where $m$ is part of the data of a trimodification. The unlabelled isomorphisms are associated with $\iota$.
This is a vertical transformation from $\el{\eta}$ to $\el{\tau}$ and defines $\el$ on 2-cells.

Suppose $\zeta \colon \Gamma\to\Lambda$ is a perturbation in $[\B^\coop,\bicat]$. We define a modification $\el{\zeta} \colon \el{\Gamma}\to\el{\Lambda}$. The first component of $\el{\zeta}_{(x,\e{x})}$ is $1_{1_x}$ and the second component is
\begin{equation*}
\cd[@R+12pt@C+10pt]
{
\eta_x(\e{x}) \pair{r}{\Gamma_x}{\Lambda_x} \dtwocell{r}{\zeta_x} & \tau_x(\e{x}) \ar[r]^{\iota } \ar[dr]_{\iota } & G1\tau_x(\e{x}) \\
 & \twocong[0.6]{ur} & G1\tau_x(\e{x}) \pair{u}{1}{{G1_1\tau_x}_\e{x}} \twocong{u}
}.
\end{equation*}
This is a vertical modification from $\el{\Gamma}$ to $\el{\Lambda}$ and defines $\el$ on 3-cells.

It can be verified that $\el$ is a trihomomorphism.

Suppose $\alpha \colon \el F \to \el G$ is a cartesian homomorphism. We will define a tritransformation $\hat\alpha \colon F \To G$ with $\el\hat\alpha \simeq \alpha$. This means homomorphisms $\hat\alpha_x \colon Fx \to Gx$ and pseudo-natural equivalences
$\hat\alpha_x . F(\mhyphen) \To  G(\mhyphen).\hat\alpha_y$
\begin{equation*}
\cd[@R+12pt]
{
\hat\alpha_x . Fg \ar@{=>}[d]_{\hat\alpha_x F\gamma} \ar@{=>}[r]^{\hat\alpha_g} & Gg.\hat\alpha_y \ar@{=>}[d]^{G\gamma\hat\alpha_x} \\
\hat\alpha_x . Ff \ar@{=>}[r]_{\hat\alpha_f} \urthreecell{ur}{\hat\alpha_\gamma} & Gf.\hat\alpha_y
}
\end{equation*}
together with two invertible modifications $\Pi$ and $M$.

Suppose $\sigma \colon h \To k \colon a \to b$ in $Fx$. Define $\hat \alpha_x \colon Fx \to Gx$ to be the composite $Fx \to \el F_x \to \el G_x \to Gx$ where $\el F_x$ is the fibre over $x$. The first map sends 
\begin{equation*}
\cd[@C+12pt]
{
 a \pair{r}{h}{k} \dtwocell{r}{\sigma} & b
}
\mapsto
\cd[@C+12pt]
{
 (x,a) \pair{r}{(1_x,\iota.h)}{(1_x,\iota.k)} \dtwocell[0.3]{r}{(1_{1_x},\delta)} &  (x,b)
}
\hspace{0.5cm}\text{where $\delta$ is}\hspace{0.5cm}
\cd[@C+12pt]
{
a \ar[r]^h \ar[dr]_k & b \ar[r]^-{\iota_b} \ar@{<-}[d]^{1} & F1(b) \ar@{<-}[d]^{F(1_1)_b} \\
   \drtwocell[0.65]{ur}{\sigma}   & b \ar[r]_-{\iota_b} \drtwocell{ur}{\iota_{1_b}} & F1(b) 
}.
\end{equation*}
The second map is $\alpha$ restricted to fibres and the final map is $\pi_2$, the projection onto the second component.
We choose not to include descriptions of $\hat\alpha_f$, $\hat\alpha_\sigma$, $\Pi$ or $M$.

We find that $\alpha(x,\e{x})$ equals $\el\hat\alpha(x,\e{x})$ and that $1_{\alpha(x,\e{x})}$ are the 1-cell components of a pseudo-natural equivalence. Thus $\el$ is locally surjective up to equivalence.

Suppose $\Gamma \colon \el\alpha \to \el \beta$ is a vertical transformation. We will define a trimodification $\hat\Gamma \colon \alpha \To \beta$ with $\el\hat\Gamma \cong \Gamma$. This means transformations $\hat\Gamma_x \alpha_x \To \beta_x$ together with invertible modifications $m$.

Remember that the action of $\el\alpha$ on objects is $\el\alpha(x,\e{x})=(x,\alpha_x(\e{x}))$. Thus the second component of $\Gamma_{(x,\e{x})}\colon (x,\alpha_x(\e{x})) \to (x,\beta_x(\e{x}))$ is a 1-cell $\alpha_x(\e{x}) \to G1\beta_x(\e{x})$. Then let $(\hat\Gamma_x)_\e{x}$ be this 1-cell composed with $\iota^\centerdot \colon G1\beta_x(\e{x}) \to \beta_x(\e{x})$. We choose not to include descriptions of $(\hat\Gamma_x)_f$ or $m$.

We find that $(\el \hat\Gamma)_{(x,\e{x})} = (1_x,\iota.(\hat\Gamma_x)_\e{x}) = (1_x,\iota.\iota^\centerdot.\pi_2\Gamma_{(x,\e{x})}) \cong (1_x,\pi_2\Gamma_{(x,\e{x})}) = \Gamma_{(x,\e{x})}$ and that this invertible 2-cell is the component of a modification. Thus $\el$ is surjective up to isomorphism on 2-cells.

Suppose $\zeta \colon \el\Gamma \to \el \Lambda$ is a vertical modification. We will define a perturbation $\hat\zeta \colon \Gamma \To \Lambda$ with $\el\hat\zeta = \zeta$. This means modifications $\hat\zeta_x \colon \Gamma_x \Rrightarrow \Lambda_x$ satisfying the appropriate axioms.

The 2-cell components of $\zeta_{(x,\e{x})}$ have first component $1_{1_x}$ and second component a 2-cell pictured on the left. 
We define $\hat\zeta$ by letting $(\hat\zeta_x)_\e{x}$ equal the pasting given on the right.
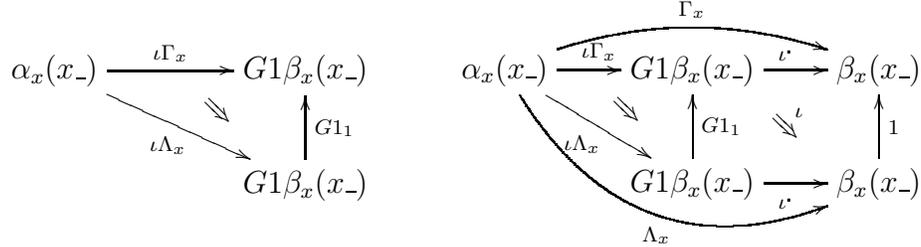
\begin{equation*}
\cd[@C+24pt]{
\alpha_x(\e{x}) \ar[r]^-{\iota\Gamma_x} \ar[dr]_{\iota\Lambda_x} & G1\beta_x(\e{x}) \\
 \drtwocell[0.65]{ur}{} & G1\beta_x(\e{x}) \ar[u]_{G1_1}
}
\hspace{1cm}
\cd[]{
\alpha_x(\e{x}) \ar@/^16pt/[rr]^{\Gamma_x} \ar@/_36pt/[drr]_{\Lambda_x} \ar[r]^{\iota\Gamma_x} \ar[dr]_{\iota\Lambda_x} & G1\beta_x(\e{x}) \ar[r]^{\iota^\centerdot} & \beta_x(\e{x}) \\
 \drtwocell[0.65]{ur}{} & G1\beta_x(\e{x}) \ar[u]_{G1_1}  \ar[r]_{\iota^\centerdot} & \beta_x(\e{x}) \ar[u]_{1} \drtwocell{ul}{\iota}
}
\end{equation*}
This makes $\hat\zeta$ a perturbation and it is unique with the property that $\el\hat\zeta = \zeta$. Thus $\el$ is bijective on 3-cells.

\end{proof}

\begin{Rk}[Variations on the Grothendieck construction] \label{rk:variation_weak}
Suppose that $P \colon E \to B$ is a 2-fibration and apply the inverse Grothendieck construction with the ordinary notion of fibre. When we inspect the reasoning we find:
\begin{itemize}
\item $Fb = E_b$ is a 2-category.
\item $Ff = f^*$ is a 2-functor.
\item $F\alpha = \alpha^*$ is a pseudo-natural transformation (2-natural when $P$ is horizontally split).
\item $F$ is still locally a homomorphism (locally a 2-functor when $P$ is locally split and horizontally split).
\item $\chi$ is a pseudo-natural isomorphism and $\chi_{gf}$ is 2-natural ($\chi$ is 2-natural when $P$ is horizontally split and
 $\chi_{gf}$ is an identity when $P$ is split on 1-cells).
\item $\iota$ is a pseudo-natural isomorphism and $\iota_{b}$ is 2-natural
($\iota$ is 2-natural when $P$ is locally split and $\iota_{b}$ is an identity when $P$ is split on 1-cells).
\item $\omega$, $\delta$ and $\gamma$ are identities.
\end{itemize}
This amounts to a trihomomorphism $F \colon B^\coop \to \Gray$ where $\chi$, $\iota$ are invertible, $\omega$, $\delta$ and $\gamma$ are identities and $\chi_{gf}$, $\iota_a$ are 2-natural. Trihomomorphisms of this kind certainly do give 2-fibrations under the Grothendieck construction.

When $P$ is locally and horizontally split these trihomomorphisms map into $\two\Cat$ and are just homomorphisms of ``$\two\Cat$-enriched bicategories''. Functors of this kind certainly do give (appropriately split) 2-fibrations under the Grothendieck construction.
\end{Rk}

\begin{Rk}
As in section 2: the action of the Grothendieck construction on objects is described by Bakovic in \cite{bakovic2012} section 6. Section 5 of the same paper gives a partial description of the action on objects of the pseudo-inverse. With some adjustments, we have completed the second construction (Theorem 5.1) and shown that together they form an equivalence of 3-categories. It was in completing Bakovic's description of the pseudo-inverse to the Grothendieck construction that we discovered that cartesian 2-cells must be closed under post-composition with all 1-cells.
\end{Rk}

\subsection{Examples}

\begin{Con}[Families]
When $\B$ is a bicategory we define $\Fam(\B)$ as the bicategory of `1-cell diagrams' in $\B$. An object of $\Fam(\B)$ is a pair $(\C,X)$ of where $\C$ is a category and $ X\colon \C^\op \to \B$ is a pseudo-functor. A 1-cell is a pair $(F,\alpha)\colon (\C,X) \to (\D,Y)$ where $F \colon \C \to \D$ is a functor and $\alpha \colon X \To YF^\op$ is a pseudo-natural transformation.
A 2-cell is a pair $(\sigma, \Sigma) \colon (F,\alpha) \To (G,\beta)$ where $\sigma\colon F \To G$ is a natural transformation and $\Sigma \colon \alpha \Rrightarrow Y\sigma^\op.\beta$ is a modification.
\begin{equation*}
\cd[@R+6pt]
{
\C^\op \ar[dr]_{X} \ar[rr]^{F^\op} & \rtwocell[0.55]{d}{\alpha} & \D^\op \ar[dl]^{Y} &  \C^\op \ar[dr]_{X} \pair[8pt]{rr}{F^\op}{G^\op} \utwocell{rr}{\sigma^\op} & \rtwocell[0.65]{d}{\beta} & \D^\op \ar[dl]^{Y}  \\
 & \B & \rthreecell{ur}{\Sigma} & & \B &
}
\end{equation*}
Composition and identities are not hard to describe. The coherence isomorphisms $a$, $l$ and $r$ are modifications obtained from the corresponding coherence isomorphisms in $\B$.
There is an obvious functor $\pi \colon \Fam(\B) \to \Cat$ defined by projection onto the first component.
\end{Con}

\begin{Ex} \label{ex:Fam_weak}
$\pi \colon \Fam(\B) \to \Cat$ is a fibration.
\end{Ex}
\begin{proof}
First, suppose that $(\D,Y)$ in $\Fam(\B)$ and $F\colon \C \to \D$ in $\Cat$. Let the cartesian lift of $F$ at $(\D,Y)$ be $(F,1_{YF^\op})\colon (\C,YF^\op) \to (\D,Y)$.
Now suppose that $(G,\beta) \colon (\C,X) \to (\D,Y)$ in $\Fam(\B)$ and $\sigma \colon F \To G$ in $\Cat$. 
Let the cartesian lift of $\sigma$ at $(G,\beta)$ be $(\sigma, 1_{Y\sigma.\beta}) \colon (F, Y\sigma.\beta) \To (G, \beta)$.
The details here are somewhat more complicated but the basic behaviour is the same as Example \ref{ex:Fam_strict}.
\end{proof}

\begin{Defn}
Suppose that $\B$ is a bicategory. An arrow $p\colon a \to b$ in $\B$ is called a \emph{Street fibration} when $p_* \colon \B(c,e) \to \B(c,b)$ is a Street fibration for all $c$ and the square 
\begin{equation*}
\cd[]
{
\B(c,e) \ar[d]_{p_*} \ar[r]^{f^*} & \B(c',e) \ar[d]^{p_*}\\
\B(c,b) \ar[r]_{f^*} \twocong{ur} & \B(c',b)
}
\end{equation*}
is a morphism of Street fibrations for all $f\colon c' \to c$.
\end{Defn}

This means that for each $g\colon e\to a$ and 2-cell $\alpha \colon h \to pg$ there exists a 2-cell $\c{\alpha}{g} \colon \alpha^*g \To g$ and isomorphism $\eta \colon h \To p \alpha^*g$  where $\c{\alpha}{g}$ is cartesian for $p_*$ and
\begin{equation*}
\cd[@C+24pt]
{
e \ar[r]^{g} \ar@/_10pt/[dr]_{h} & a \ar[d]^{p} \\
 \urtwocell[0.6]{ur}{\alpha} & b
}
=
\cd[@C+24pt]
{
e \ar@/^10pt/[r]^{g} \ar@/_10pt/[r]_(0.4){\alpha^*g} \ar@/_10pt/[dr]_{h} \utwocell[0.35]{r}{\c{\alpha}{g}} \urtwocell[0.7]{dr}{\eta} & a \ar[d]^{p} \\
 & b
}.
\end{equation*}
It also means that cartesian 2-cells are closed under pre-composition with arbitrary 1-cells.

\begin{Defn}
A morphism between Street fibrations $p\colon e\to b$ and $q\colon e'\to b'$ is a pair $(f\colon e \to e', g\colon b \to b')$ where $q.f \cong g.p$ and 
\begin{equation*}
\cd[]
{
\B(c,e) \ar[d]_{p_*} \ar[r]^{f_*} & \B(c,e') \ar[d]^{q_*}\\
\B(c,b) \ar[r]_{g_*} \twocong{ur} & \B(c,b')
}
\end{equation*}
is a morphism of Street fibrations for all $c$.
\end{Defn}

\begin{Con}[Internal fibrations] \label{street_fibrations}
We define $\Fib_\B$ to be the bicategory with:
\begin{itemize}
\item Objects are Street fibrations $g\colon a \to b$ in $\B$. These are sometimes written as a triple $(a,g,b)$.
\item 1-cells are triples $(h,\sigma,h') \colon (a,g,b) \to (c,g',d)$ where $\sigma$ is an isomorphism
\begin{equation*}
\cd[]
{
a \ar[r]^{h}  \ar[d]_{g} \dtwocell{dr}{\sigma}& c \ar[d]^{g'} \\
b \ar[r]_{h'} & d
}
\end{equation*}
and $h$ is a cartesian 1-cell.
\item 2-cells are pairs of 2-cells $(\alpha, \alpha') \colon (h,\sigma,h') \to (k,\tau,k')$ so that
\begin{equation*}
\cd[@R+4pt@C+18pt]
{
a \ar[r]^{h}  \ar[d]_{g} \dtwocell[0.3]{dr}{\sigma} & c \ar[d]^{g'} \\
b \pair{r}{h'}{k'} \dtwocell{r}{\alpha'} & d
}
=
\cd[@R+4pt@C+18pt]
{
a \pair{r}{h}{k} \dtwocell{r}{\alpha} \ar[d]_{g} \dtwocell[0.7]{dr}{\tau} & c \ar[d]^{g'} \\
b \ar[r]_{k'} & d
}
\end{equation*}
in $\B$. We sometimes write $(\alpha,1,\alpha')$ where $1$ is representative of the commuting condition.
\end{itemize}
Composition and identities are easy to describe.

There is a homomorphism $\cod \colon \Fib_\B \to \B$ defined by projection onto the third component.
It is called the codomain functor because it projects onto the codomain of the objects of $\Fib_\B$.
It is modelled on $\cod \colon \Fib \to \Cat$ which was used by Hermida to guide the definition of 2-fibrations.
\end{Con}

\begin{Ex}
When $\B$ has bipullbacks $\cod \colon \Fib_\B \to \B$ is a fibration.
\end{Ex}
\begin{proof}
Suppose that $(c,q,d)$ in $\Fib_\B$ and $h\colon b \to d$ in $\B$. The cartesian lift of $h$ is obtained by taking the bipullback of $q$ and $h$ 
\begin{equation*}
\cd[]
{
a \ar[r]^{\hat{h}}  \ar[d]_{\hat{q}} \dtwocell{dr}{\sigma} & c \ar[d]^{q} \\
b \ar[r]_{h} & d
}.
\end{equation*}
See \cite{street1974} for a proof that internal Street fibrations are closed under bipullback.

Now suppose that $(h,\sigma,h')$ in $\Fib_\B$ and $\alpha' \colon k' \To g'$ in $\B$. Then $\sigma. \alpha'g \colon k'g \To g'h$ and since $g'$ is a Street fibration, we get a cartesian lift $\alpha\colon k \To h$ with $\eta \colon g'k \cong k'g$
\begin{equation*}
\cd[@C+12pt]
{
a  \ar@/^10pt/[dr]^{h} \ar[d]_{g}  &  \urtwocell[0.6]{dl}{\sigma}   \\
b  \pair{dr}{h'}{k'} &  c \ar[d]^{g'} \urtwocell[0.5]{dl}{\alpha'} \\
   &  d
}
=
\cd[@C+12pt]
{
a  \pair{dr}{h}{k} \ar[d]_{g}  &  \urtwocell[0.5]{dl}{\alpha}   \\
b  \ar@/_10pt/[dr]_{k'} &  c \ar[d]^{g'}  \\
\urtwocell[0.6]{ur}{\eta}   &  d
}.
\end{equation*}
Then the cartesian lift of $\alpha'$ is $(\alpha,1,\alpha') \colon (k,\eta,k') \to (h,\sigma,h')$.
\end{proof}

\begin{Ex}[Algebras]
Let $\Mnd(\K)$ be the bicategory of pseudo-monads on a bicategory $\K$ (called \emph{doctrines} in \cite{street1980} 2.10). There is a trihomomorphism $\Mnd(\K)^\coop\to\bicat$ that maps each pseudo-monad $T$ to the bicategory $T\mhyphen\Alg$ of $T$-algebras, lax algebra morphisms and algebra 2-cells. We can use the Grothendieck construction to construct a fibration $\Alg\to\Mnd$. The objects of the total category $\Alg$ are algebras for a pseudo-monad: pairs $(S,(A,m))$ where $m\colon SA \to A$ is a $S$-algebra. The 1-cells from $(S,(A,m))$ to $(T,(B,n))$ are pairs $(\lambda,(f,\theta_f))$ where $\lambda$ is a monad morphism from $S$ to $T$ and $(f,\theta_f)\colon (A,m) \to F\lambda(B,n)$ is a lax algebra morphism.
\begin{equation*}
\cd[]{
SA \ar[d]_{Sf} \ar[rr]^{m} \urtwocell{drr}{\theta_f} && A \ar[d]^{f}\\
SB \ar[r]_{\lambda_B} & TB \ar[r]_n & B
}.
\end{equation*}
The 2-cells of $\Alg$ are pairs $(\Gamma,\alpha)\colon (\lambda,(f,\theta_f))\to (\tau,(g,\theta_g))$ where $\Gamma\colon\lambda\to\tau$ is a monad 2-cell and $\alpha$ is an algebra 2-cell
\begin{equation*}
\cd[@C-20pt]{
(A,m) \ar[dr]_{(g,\theta_g)} \ar[rr]^-{(f,\theta_f)} & \dtwocell{d}{\alpha} & F\lambda(B,n) \\
 &  F\tau(B,n) \ar[ur]_{F\Gamma_{(B,n)}} & 
}.
\end{equation*}
The fibration is projection on the first component of $\Alg$.
By construction the fibre over $T$ is equivalent to $T\mhyphen\Alg$.
\end{Ex}

\begin{Ex}[Equivalence lifting]
A homomorphism $P \colon \E \to \B$ is said the have the equivalence lifting property when:
\begin{enumerate}
\item for each object $e$ and equivalence $f \colon b \to Pe$ there is an equivalence $\hat f \colon a \to e$ with $P\hat f = f$; and
\item for each arrow $h \colon e \to e'$ and isomorphism $\alpha \colon g \to Ph$ there is an isomorphism $\hat \alpha \colon k \to h$ with $P\hat \alpha = \alpha$.
\end{enumerate}
When $P$ is strict, these are precisely the fibrations in Lack's Quillen model structure on $\bicat_s$ \cite{lack2004}.

Every fibration has the equivalence lifting property. Further more, when $\E$ and $\B$ are bigroupoids (bicategories in which all 1-cells are equivalences and all 2-cells are isomorphisms) every homomorphism with the equivalence lifting property is a fibration.

\end{Ex}

\section{Composition, Commas and Pullbacks}

In this section we show that fibrations are closed under composition and closed under equiv-comma, and that projections from oplax comma bicategories are fibrations.

\subsection{Composition}

\begin{Prop}
If $P \colon \D \to \B$ and $Q \colon \E \to \D$ are fibrations then $PQ \colon \E \to \B$ is a fibration.
\end{Prop}
\begin{proof}
We first need to show that it has cartesian lifts of 1-cells.
Suppose that $e\in E$ and $f \colon a \to PQe$ in $B$, then let $f'' = \c{\c{f}{Qb}}{b}$. This is the cartesian double-lift:
\begin{equation*}
\cd[@R-12pt]
{
f'^*b \ar[r]^{f''} & b \\
f^*Qb \ar[r]^-{F'} & Qb \\
a \ar[r]^{f} & PQb
}
\end{equation*}
where $f'$ is cartesian for $P$ over $f$ and $f''$ is cartesian for $Q$ over $f'$. We need to show that $f''$ is cartesian for $PQ$ over $f$. This is trivial when we consider the definition of cartesian 1-cell by bipullback.
\begin{equation*}
\cd[]{
E(z,x) \ar[d]_{Q_{zx}} \ar[r]^{f''_*} & E(z,y) \ar[d]^{Q_{zy}} \ar@{}[dl]|{\cong} \\
D(Qz,Qx) \ar[d]_{P_{QzQx}} \ar[r]^{f'_*} & D(Qz,Qy) \ar[d]^{P_{QzQy}} \ar@{}[dl]|{\cong} \\
B(PQz,PQx) \ar[r]_{f_*} & B(PQz,PQy)
}
\end{equation*}
Since both squares below are bipullbacks, the composite is a bipullback and $f''$ is a cartesian 1-cell over $f$. This could also be proved by explicitly using the properties of cartesian 1-cells.

Now we need to show that $PQ$ locally fibred. Since $PQ_{xy}$ is defined by the composite
\begin{equation*}
\cd[]
{
E(x,y) \ar[r]^{Q_{xy}} & D(Qx,Qy) \ar[r]^{P_{QxQy}} & B(PQx,PQy)
}
\end{equation*}
and $P$ and $Q$ are locally fibred, $PQ_{xy}$ is a fibration. Thus $PQ$ is locally fibred. The cartesian lifts of 2-cells in $B$ are the double-lifts similar to those described above.

Finally, we need to show that cartesian 2-cells closed under horizontal composition.
Suppose that $\alpha$ and $\beta$ are the chosen cartesian lifts of $PQ\alpha$ and $PQ\beta$; then $\alpha * \beta$ is cartesian over $Q(\alpha * \beta)$ because $Q$ is a fibration. Notice also that $Q\alpha * Q\beta$ is cartesian over $P(Q\alpha * Q\beta)$ because $P$ is a fibration. But
\begin{equation*}
\cd[]
{
Qhf \ar@{=>}[r]^{Q(\alpha*\beta)} & Qkg
}
=
\cd[]
{
Qhf \ar@{=>}[r]^{\phi} & QhQf \ar@{=>}[r]^{Q\alpha*Q\beta} & QkQg \ar@{=>}[r]^{\phi^{-1}} & Qkg
}
\end{equation*}
where $\phi$ is the composition coherence isomorphism for $Q$. Now since isomorphisms are cartesian, $Q(\alpha*\beta)$ is cartesian over $PQ(\alpha*\beta)$. Thus $\alpha*\beta$ is a cartesian lift of a cartesian lift and so it is cartesian for $PQ$. We've proven this for the chosen cartesian lift only but by Proposition \ref{prop:chosen_2-cells_enough_strict} this is enough.
\end{proof}

\begin{Prop}
If $P \colon D \to B$ and $Q \colon E \to D$ are 2-fibrations then $PQ \colon E \to B$ is a 2-fibration.
\end{Prop}
\begin{proof}
This proof is essentially the same as that for fibrations. There are no major differences.
\end{proof}

\subsection{Oplax Comma bicategories}


\begin{Con}[Oplax Comma bicategory]
Suppose that $F$, $G$ are homomorphisms.
Let $\comma{F}{G}$ be the bicategory whose objects are $(x,x',\tau_x)$ where $x\in C$, $x'\in B$ and $\tau_x \colon Fx \to Gx'$.
The arrows are triples $(f,f',\tau_f)\colon (x,x',\tau_x) \to (y,y',\tau_y)$ where $f \colon x \to y$, $f' \colon x' \to y'$ and $\tau_f\colon \tau_yFf \To Gf'\tau_x$.
The 2-cells are triples $(\alpha,\alpha',\tau_\alpha)\colon (f,f',\tau_f) \To (g,g',\tau_g)$ where $\alpha \colon f \To g$, $\alpha' \colon f' \To g'$ and $\tau_\alpha$ is an equality:
\begin{equation*}
\cd[@C+16pt]
{
 Fx \ar[r]^{\tau_x} \pair{d}{Fg}{Ff} \rtwocell{d}{F\alpha} & Gx' \ar[d]^{Gg'}\\
 Fy \ar[r]_{\tau_y} & Gy' \urtwocell[0.35]{ul}{\tau_g}
}
\hspace{0.1cm}=\hspace{0.1cm}
\cd[@C+16pt]
{
 Fx \ar[r]^{\tau_x} \ar[d]_{Ff} & Gx' \pair{d}{Gg'}{Gf'} \rtwocell{d}{G\alpha'}\\
 Fy \ar[r]_{\tau_y} \urtwocell[0.35]{ur}{\tau_f} & Gy'
}.
\end{equation*}
Composition and identities are given in the obvious way. By projection onto the first and second components we obtain pseudo-functors $d_0$ and $d_1$. Both of these preserve composition and identities on the nose.

This gives rise to an oplax natural transformation whose components are given by projecting $\comma{F}{G}$ onto its third component.
\begin{equation*}
\cd[]
{
\comma{F}{G} \ar[r]^{d_1} \ar[d]_{d_0} & \B \ar[d]^{G} \\
\C \ar[r]_{F} & \D \urtwocell{ul}{\tau}
}
\end{equation*}
\end{Con}

\begin{Rk}[Lax Comma bicategories]
The obvious variation on this construction is called ``2-comma-category'' in \cite{gray1969} and ``lax comma category'' in \cite{kelly1974}. We have merely adjusted it to fit the context of bicategories.
\end{Rk}

\begin{Rk}[Weighted limits]
The oplax comma bicategory construction could also be defined as some kind of weighted limit. The same is true of other constructions later in this section. We leave the details to the interested reader.
\end{Rk}

\begin{Prop}
For any $F\colon \C \to \D$ and $G\colon \B \to \D$, $d_0\colon \comma{F}{G}\to \C$ is a fibration.
\end{Prop}
\begin{proof}
We need to show that $d_0$ has cartesian lifts of 1-cells. Suppose that $(y,y',\sigma_y)$ in $\comma{F}{G}$ and $f\colon x \to y$ in $\C$. Then there exists $(f,1,\sigma_{f1})\colon(x,x',\sigma_x)\to(y,y',\sigma_y)$ where $x'=y'$, $\sigma_x = \sigma_y.Ff$ and $\sigma_{f1}$ is
\begin{equation*}
\cd[@C+12pt@R+12pt]
{
 Fx \ar[r]^{\sigma_x} \ar[d]_{Ff} \ar[dr]_{\sigma_x} & Gy' \pair{d}{G1}{1} \twocong{d}\\
 Fy \ar[r]_{\sigma_y} \twocong[0.7]{ur} & Gy'~.
}.
\end{equation*}
where the unlabelled isomorphisms are coherences for $G$ and $\D$.
This is a cartesian lift of $f$.

We need to show that $d_0$ has cartesian lifts of 2-cells. Suppose that $(g,g',\sigma_g)\colon(x,x',\sigma_x)\to(y,y',\sigma_y)$ in $\comma{F}{G}$ and $\alpha\colon f \To g$ in $\C$. Then there exists $(\alpha,1_{g'})\colon(f,f',\sigma_f)\To(g,g',\sigma_g)$ where $f'=g'$, $\sigma_f = \sigma_g . \sigma_yF\alpha$ and thus
\begin{equation*}
\cd[@C+16pt]
{
 Fx \ar[r]^{\sigma_x} \pair{d}{Fg}{Ff} \rtwocell{d}{F\alpha} & Gx' \ar[d]^{Gg'}\\
 Fy \ar[r]_{\sigma_y} & Gy' \urtwocell[0.35]{ul}{\sigma_g}
}
\hspace{0.5cm}=\hspace{0.5cm}
\cd[@C+16pt]
{
 Fx \ar[r]^{\sigma_x} \ar[d]_{Ff} & Gx' \pair{d}{Gg'}{Gg'} \rtwocell{d}{G1}\\
 Fy \ar[r]_{\sigma_y} \urtwocell[0.35]{ur}{\sigma_f} & Gy'
}.
\end{equation*}
This is a cartesian lift of $\alpha$.

We need to check that cartesian 2-cells are closed under horizontal composition. Examining the chosen cartesian 2-cells we find that
$(\alpha,1_{g'})*(\beta,1_{k'}) = (\alpha*\beta,1_{g'k'})$ and thus they're closed under composition. We've proven this for the chosen cartesian lifts only, but by proposition \ref{prop:chosen_2-cells_enough_strict} this is enough.
\end{proof}

\begin{Rk}
When $F$ is an identity $1_\D$ we use $\D/G \to \D$ instead of $\comma{1_D}{G} \to \D$. This is referred to by Bakovic \cite{bakovic2012} as the ``canonical fibration associated to F''. We call it the free fibration on $F$ (see the following remark).
\end{Rk}

\begin{Rk}[Free fibrations]
If $H \colon \A \to \B$ is a homomorphism then $FH = d_0 \colon {\B}/{H} \to \B$ is the free fibration on $H$. There is a biequivalence
$$ (\Fib/\B)(FH,P) \simeq (\Cat/\B)(H,UP)$$
where $UP$ is $P$ regarded simply as a homomorphism. The details are very similar to the standard result for ordinary fibrations. 
The only extra result we need is to know that you can lift arbitrary 2-cells along cartesian 1-cells (the 2-cells in the base need not be isomorphisms). This relies on the fact that fibrations are locally fibred.

The naturality of this biequivalence is very weak: it is the 1-cell component of a tritransformation
\begin{equation*}
\cd[@C+2cm]
{
 (\bicat/\B)^\op \times \Fib/\B \pair{r}{\Fib/\B(F\mhyphen,\mhyphen)}{\bicat/\B(\mhyphen,U\mhyphen)}
 \dtwocell[0.7]{r}{} & \bicat
}.
\end{equation*}
\end{Rk}

\begin{Rk}[Two-sided fibrations]
This construction also makes $d_1$ a `coop-fibration' and the span $(d_0, d_1)$ has some of the characteristics of a two-sided discrete fibration.
By `coop-fibration' we mean the notion dual to fibration: a homomorphism that is locally an opfibration, has opcartesian lifts of 1-cells and opcartesian 2-cells are closed under horizontal composition.

Locally, the span $(d_0, d_1)$ is a discrete two-sided fibration: whenever $(\alpha,\alpha')\colon (f,f',\sigma_f) \To (g,g',\sigma_g)$ we have
\begin{align*}
 d_1(\c{\alpha}{(g,g',\sigma_g)}) & = 1_{g'} \\
 d_0(\c{\alpha'}{(f,f',\sigma_f)})&  = 1_{f} \\
 \c{\alpha'}{(f,f',\sigma_f)}.\c{\alpha}{(g,g',\sigma_g)}) & = (\alpha,\alpha').
\end{align*}
This means that the cartesian lifts for $d_0$ are identities under the action of $d_1$; and vice versa; and that the cartesian lift of $d_0(\alpha,\alpha')$ and the opcartesian lift of $d_1(\alpha,\alpha')$ compose to give $(\alpha,\alpha')$.

On 1-cells the behaviour is somewhat weaker. Whenever $(f,f',\sigma_f)\colon (x,x',\sigma_x) \to (y,y',\sigma_y)$ we have
\begin{align*}
d_1(\c{f}{(y,y',\sigma_y)})  & = 1_{y'} \\
d_0(\c{f'}{(x,x',\sigma_x)}) & = 1_{x} 
\end{align*}
and
\begin{equation*}
\cd[@C-6pt]
{
 (x,x',\sigma_x) \ar[rrrr]^{(f,f',\sigma_f)} \ar[dr]_{(1,f',r.\phi)} & & & & (y,y',\sigma_y) \\
 & (x,y',Gf'\sigma_x) \ar[rr]_-{(1,\phi.l. \sigma_f .r. \phi)} \dtwocell{urr}{(r.r,l.l)} & & (x,y',\sigma_yFf) \ar[ur]_{(f,1,\phi.l)} & 
}
\end{equation*}
where the $\phi$ are coherence isomorphisms for $F$ and $G$.

\end{Rk}

\begin{Prop}
For any 2-functors $F\colon C \to D$ and $G\colon B \to D$, $d_0\colon \comma{F}{G}\to C$ is a 2-fibration.
\end{Prop}
\begin{proof}
This proof is essentially the same as that for fibrations. There are no major differences.
\end{proof}

\subsection{Pullbacks}


\begin{Con}[equiv-comma]
Suppose $F\colon \C \to \D$ and $G\colon \B \to \D$ are homomorphisms.
The equiv-comma $\C \times_{\simeq} \B$ is the subcategory of $\comma{F}{G}$ containing:
all objects $(x,x',\tau_x)$ where $\tau_x \colon Fx \to Gx'$ is an \emph{equivalence};
all 1-cells $(f,f',\tau_f)$ where $\tau_f$ is an \emph{isomorphism}; and all 2-cells.
The functors $G'$ and $F'$ are the projections onto the first and second components of $\C \times_{\simeq} \B$. This gives a pseudo-natural equivalence
\begin{equation*}
\cd[]
{
\C \times_{\simeq} \B \ar[r]^-{F'} \ar[d]_{G'} & \B \ar[d]^{G} \\
\C \ar[r]_{F} & \D \urtwocell{ul}{\tau}
}.
\end{equation*}
\end{Con}

\begin{Prop}
\label{prop:equiv_pullback}
Let $\A\times_{\simeq} \E$ be the equiv-comma of $P \colon \E \to \B$ and $F \colon \A \to \B$. If $P$ is a fibration then $P'$ is a fibration and $F'$ is cartesian.
\end{Prop}
\begin{proof}
We want to show that $P'$ is a fibration. We first need to show that it has cartesian lifts of 1-cells.
Suppose $(y,y',\tau_y)$ and $f\colon x \to y$ and let $\hat{\tau} \colon \tau^*y' \to y'$ be the cartesian lift of $\tau_y.Ff \colon Fx \to Fy \to Py'$ in $\B$. Then
$(f,\hat{\tau},\eta)\colon(x,\tau^*y',1_{Fx}) \to (y,y',\tau_y)$
where $\eta$ is 
\begin{equation*}
\cd[@R-6pt]
{
Fx \ar[dd]_{Ff} \ar[r]^{1} \ar[dr]_{Ff} & Fx \ar[d]^{Ff}\\
 \dtwocell[0.7]{ur}{r} & Fy \ar[d]^{\tau_y} \\
Fy \ar[r]_{\tau_y} & Py'
}.
\end{equation*}
This is a cartesian lift of $f$ for $P'$.

We need to show that there are cartesian lifts of 2-cells.
Suppose $(g,g',\tau_g)\colon (x,x',\tau_x) \to (y,y',\tau_y)$ and $\alpha \colon f \To g \colon x \to y$. Suppose also that $\tau_x$ is an adjoint equivalence. Now let $\hat{\alpha}\colon f' \to g'$ be the cartesian lift of the following composite at $g'$.
\begin{equation*}
\cd[@R-6pt@C+16pt]
{
 & Fx \dtwocell{r}{F\alpha} \ar@/^10pt/[r]^{Ff} \ar@/_10pt/[r]|{Fg} \ar[dr]_{\tau_x} & Fy \ar[dr]^{\tau_y} \dtwocell{d}{\tau_g} & \\
 Px' \ar[ur]^{\tau_x^\centerdot} \ar[rr]|1 \ar@/_10pt/[rrr]_{Pg'} & \dtwocell{u}{\epsilon}   & Px' \ar[r]^{Pg'}  & Py'
}
\end{equation*}
Then $(\alpha,\hat{\alpha})\colon (f,f',\tau_f)\To(g,g',\tau_g)$ where $\tau_f$ is
\begin{equation*}
\cd[@C+6pt]
{
Fx \ar[ddd]_{Ff} \ar[r]^{\tau_x} \ar[dr]_{1} & Px' \ar@/^12pt/[ddd]^{Pf'} \ar[d]_{\tau^\centerdot_x} \\
\urtwocell[0.65]{ur}{\eta} & Fx' \ar[d]_{Ff} \\
  & Fy'  \ar[d]_{\tau_y} \\
Fy \ar[r]_{\tau_y} \urtwocell{uur}{r} & Py'
}.
\end{equation*}
This is a cartesian lift of $\alpha$ for $P'$.

We need to show that cartesian 2-cells are closed under horizontal composition. The chosen cartesian lifts are cartesian precisely because their second component is cartesian for $P$. Since $P$ is a fibration we know that the second component of $(\alpha,\hat{\alpha})*(\beta,\hat{\beta}) = (\alpha*\beta,\hat{\alpha}*\hat{\beta})$ is also cartesian. Thus the chosen cartesian lifts are closed under horizontal composition and by proposition \ref{prop:chosen_2-cells_enough_strict} this is enough.
This makes $P'$ a fibration. 

Notice that cartesian lifts for $P'$ have cartesian maps in their second component so $F'$ is cartesian.
\end{proof}

\begin{Rk}
We say that a cartesian functor $F$ reflects cartesian maps when $Ff$ cartesian implies $f$ cartesian. It is worth noting that the $F'$ above reflects cartesian maps.
\end{Rk}


\begin{Rk}
Suppose $F\colon \C \to \D$ and $G\colon \B \to \D$ are homomorphisms. Technically, we cannot form the pullback of $F$ and $G$.
The following construction is in some sense the closest we can get. The ``pullback'' of $F$ and $G$ is the subcategory of $\comma{F}{G}$ containing:
all objects $(x,x',\tau_x)$ where $\tau_x \colon Fx \to Gx'$ is an \emph{identity};
all 1-cells $(f,f',\tau_f)$ where $\tau_f$ is an \emph{isomorphism}; and all 2-cells.
The functors $G^{''}$ and $F''$ are the projections onto the first and second components of $\C \times_{\D} \B$. This gives a pseudo-natural equivalence
\begin{equation*}
\cd[]
{
\C \times_\D \B \ar[r]^-{F''} \ar[d]_{G''} & \B \ar[d]^{G} \\
\C \ar[r]_{F} & \D \urtwocell{ul}{\tau}
}
\end{equation*}
This is actually an iso-comma object in $\bicat_2$: the 2-category of bicategories, homomorphisms and icons in the sense of \cite{lack2010}.

Suppose now that $G$ is a fibration. We can prove that $G^{''}$ is a fibration using essentially the same argument as Proposition \ref{prop:equiv_pullback}. Alternatively, we can prove an analogue of a result by Joyal and Street: when $G$ is a fibration the pullback has the same universal property as the equiv-comma (in the sense that the induced comparison is a biequivalence). Thus $G'$ and $G''$ are biequivalent in the slice over $\C$ and $G''$ is a fibration.
\end{Rk}

We've proved the following theorem.

\begin{Prop}
Let $\A\times_\B \E$ be the ``pullback'' of $P \colon \E \to \B$ and $F \colon \A \to \B$. 
If $P$ is a fibration then $P'$ is a fibration and $F'$ is cartesian.
\end{Prop}

A quick investigation reveals that 2-fibrations are not closed under equiv-comma. They are however closed under iso-comma as defined below.

\begin{Con}[Iso-comma]
Suppose $F\colon C \to D$ and $G\colon B \to D$ are 2-functors.
The iso-comma $C \times_{\cong} B$ is the subcategory of $\comma{F}{G}$ containing:
all objects $(x,x',\tau_x)$ where $\tau_x \colon Fx \to Gx'$ is an \emph{isomorphism};
all 1-cells $(f,f',\tau_f)\colon (x,x',\tau_x) \to (y,y',\tau_y)$ where $\tau_f$ is an \emph{identity}; and
all 2-cells.
The functors $G'$ and $F'$ are the projections onto the first and second components of $\C \times_{\cong} \B$.
This gives a 2-natural isomorphism $\tau$ as follows.
\begin{equation*}
\cd[]
{
C \times_{\cong} B \ar[r]^-{F'} \ar[d]_{G'} & B \ar[d]^{G} \\
C \ar[r]_{F} & D \urtwocell{ul}{\tau}
}
\end{equation*}
\end{Con}

\begin{Prop}\label{prop:iso-comma}
Let $A\times_{\cong} E$ be the iso-comma of $P \colon E \to B$ and $F \colon A \to B$.
If $P$ is a 2-fibration then $P'$ is a 2-fibration and $F'$ is cartesian.
\end{Prop}
\begin{proof}
This proof is essentially the same as that for equiv-commas. There are no major differences.
\end{proof}

\begin{Con}[Pullback]
Suppose $F\colon C \to D$ and $G\colon B \to D$ are 2-functors.
The pullback $C \times_D B$ is the subcategory of $\comma{F}{G}$ containing:
all objects $(x,x',\tau_x)$ where $\tau_x \colon Fx \to Gx'$ is an \emph{identity};
all 1-cells $(f,f',\tau_f)\colon (x,x',\tau_x) \to (y,y',\tau_y)$ where $\tau_f$ is an \emph{identity}; and
all 2-cells.
The functors $G'$ and $F'$ are the projections onto the first and second components of $C \times B$.
This gives commuting square.
\begin{equation*}
\cd[]
{
C \times_D B \ar[r]^-{F'} \ar[d]_{G'} & B \ar[d]^{G} \\
C \ar[r]_{F} & D
}
\end{equation*}
\end{Con}

\begin{Prop}
Let $A\times_B E$ be the pullback of $P \colon E \to B$ and $F \colon A \to B$.
If $P$ is a 2-fibration then $P'$ is a 2-fibration and $F'$ is cartesian.
\end{Prop}
\begin{proof}
The proof is essentially the same as Proposition \ref{prop:iso-comma} but much easier because the isomorphisms are equalities.
\end{proof}

\bibliographystyle{amsalpha}

\providecommand{\bysame}{\leavevmode\hbox to3em{\hrulefill}\thinspace}
\providecommand{\MR}{\relax\ifhmode\unskip\space\fi MR }
\providecommand{\MRhref}[2]{%
  \href{http://www.ams.org/mathscinet-getitem?mr=#1}{#2}
}
\providecommand{\href}[2]{#2}

\end{document}